\documentclass[11pt]{article}
\usepackage{amssymb,amsmath,bm,mathrsfs,makeidx,amsfonts,graphicx,amsthm}
\usepackage[square, comma, sort&compress, numbers]{natbib}
\usepackage{epsfig}
\usepackage{rotating}
\usepackage{mathtools}
\usepackage{bm}
\usepackage{extarrows}
\usepackage[colorlinks,linkcolor=red,anchorcolor=blue,citecolor=blue]{hyperref}
\usepackage{booktabs}
\numberwithin{equation}{section}
\def\beginn{\begin{eqnarray*}}
\def\endn{\end{eqnarray*}}
\def\beginy{\begin{eqnarray}}
\def\endy{\end{eqnarray}}
\def\begine{\begin{enumerate}}
\def\ende{\end{enumerate}}

\def\be{\begin{equation}}
\def\ee{\end{equation}}
\def\bea{\begin{eqnarray}}
\def\eea{\end{eqnarray}}

\numberwithin{equation}{section}
\theoremstyle{plain}
\newtheorem{thm}{Theorem}[section]

\newtheorem{lem}{Lemma}
\newtheorem{coro}{Corollary}
\newtheorem{rmk}{Remark}

\newtheorem{deff}{Definition}

\newtheorem{cond}{Condition}
\newcommand{\non}{\nonumber \\}
\newcommand{\bbA}{{\bf A}}

\newcommand{\bba}{{\bf a}}
\newcommand{\bbB}{{\bf B}}
\newcommand{\bbC}{{\bf C}}
\newcommand{\bbc}{{\bf c}}
\newcommand{\bbD}{{\bf D}}

\newcommand{\bbe}{{\bf e}}
\newcommand{\bbE}{{\bf E}}

\newcommand{\bbF}{{\bf F}}
\newcommand{\bbP}{{\bf P}}

\newcommand{\bbG}{{\bf G}}
\newcommand{\bbK}{{\bf K}}
\newcommand{\bbH}{{\bf H}}

\newcommand{\bbw}{{\bf w}}
\newcommand{\bbI}{{\bf I}}
\newcommand{\bbi}{{\bf i}}

\newcommand{\bbM}{{\bf M}}

\newcommand{\bbR}{{\bf R}}

\newcommand{\bbs}{{\bf s}}

\newcommand{\bbS}{{\bf S}}

\newcommand{\bbt}{{\bf t}}
\newcommand{\bbT}{{\bf T}}

\newcommand{\bbU}{{\bf U}}
\newcommand{\bbu}{{\bf u}}
\newcommand{\bbV}{{\bf V}}
\newcommand{\bbv}{{\bf v}}

\newcommand{\bbW}{{\bf W}}

\newcommand{\bbX}{{\bf X}}

\newcommand{\bbx}{{\bf x}}

\newcommand{\bbY}{{\bf Y}}
\newcommand{\bby}{{\bf y}}
\newcommand{\bbZ}{{\bf Z}}
\newcommand{\bbz}{{\bf z}}
\newcommand{\bbb}{{\bf b}}

\newcommand{\bbL}{{\bf L}}

\newcommand{\ep}{\ensuremath{\epsilon}}

\newcommand{\bbgam}{{\boldsymbol\Gamma}}
\newcommand{\bbzero}{{ 0}}
\newcommand{\bbone}{{\bf 1}}
\newcommand{\bbmu}{{\boldsymbol\mu}}
\newcommand{\bbom}{{\bold{\Omega}}}

{
\begin{document}

\title{ A unified matrix model including both CCA and F matrices in multivariate analysis:  the largest eigenvalue and its applications}
 \author{Xiao Han\footnotemark[1], Guangming Pan\footnotemark[2], and Qing Yang\footnotemark[3]}
\renewcommand{\thefootnote}{\fnsymbol{footnote}}
\footnotetext[1]{Xiao Han, School of Physical and Mathematical Sciences, Nanyang Technological University, Singapore, 637371(Email: xhan011@e.ntu.edu.sg).}
\footnotetext[2]{Guangming Pan, School of Physical and Mathematical Sciences, Nanyang Technological University, Singapore, 637371(Email: gmpan@ntu.edu.sg). This work was partially supported by a MOE Tier 2 grant 2014-T2-2-060 and by a MOE Tier 1 Grant RG25/14 at the Nanyang Technological University, Singapore.}
\footnotetext[3]{Qing Yang, School of Physical and Mathematical Sciences, Nanyang Technological University, Singapore, 637371(Email: qyang1@e.ntu.edu.sg).}
\date{}
\maketitle

\begin{abstract}
Let $\bbZ_{M_1\times N}=\bbT^{\frac{1}{2}}\bbX$ where $(\bbT^{\frac{1}{2}})^2=\bbT$ is a positive definite matrix and $\bbX$ consists of independent 
random variables with mean zero and variance one.
This paper proposes a unified matrix model $$\bold{\bbom}=(\bbZ\bbU_2\bbU_2^T\bbZ^T)^{-1}\bbZ\bbU_1\bbU_1^T\bbZ^T,$$ where
$\bbU_1$ and $\bbU_2$ are isometric with dimensions $N\times N_1$ and $N\times (N-N_2)$ respectively such
that $\bbU_1^T\bbU_1=\bbI_{N_1}$, $\bbU_2^T\bbU_2=\bbI_{N-N_2}$ and $\bbU_1^T\bbU_2=0$. Moreover, $\bbU_1$ and $\bbU_2$ (random or non-random) are independent
 of $\bbZ_{M_1\times N}$ and with probability tending to one, $rank(\bbU_1)=N_1$ and $rank(\bbU_2)=N-N_2$.
 We establish the asymptotic Tracy-Widom distribution for its largest eigenvalue under moment assumptions on $\bbX$ when $N_1,N_2$ and $M_1$ are comparable.

The asymptotic distributions of the maximum eigenvalues of the matrices
used in Canonical Correlation Analysis (CCA) and of F matrices (including centered and non-centered versions) can be both obtained from that of $\bold{\bbom}$ by selecting appropriate matrices $\bbU_1$ and $\bbU_2$.
Moreover, via appropriate matrices $\bbU_1$ and $\bbU_2$, this matrix $\bold{\bbom}$ can be applied to
some multivariate testing problems that cannot be done by the traditional CCA matrix.
To see this, we explore two more applications. One is in the MANOVA approach for testing the equivalence of several high-dimensional mean vectors,
where $\bbU_1$ and $\bbU_2$ are chosen to be two nonrandom matrices.
The other one is in the  multivariate linear model for testing the unknown parameter matrix, where $\bbU_1$ and $\bbU_2$ are random.
 For each application, theoretical results are developed and various numerical studies are conducted to confirm the satisfactory  empirical performance.

\end{abstract}

\section{Introduction}
Rapid development of modern technology nowadays necessitates statistical inference on high-dimensional data in many scientific fields such as image
processing, genetic engineering, machine learning and so on. This raises a boom in pursuing methodologies to remedy classical theories which are designed for the fixed dimensions.  For such a purpose one popular tool is the spectral analysis of high-dimensional matrices in random matrix theory. The readers may refer to the monograph \cite{BS06} and the references therein for a comprehensive reading.

This paper focuses on the largest eigenvalues. Ever since the pioneer work discovering the limiting distribution of the largest eigenvalue for the large Gaussian Wigner ensemble by Tracy and Widom in \cite{TW1994,TW1996},  the largest eigenvalues
of large random matrices have been widely studied. To name a few we mention \cite{BPZ2014a}, \cite{K2007}, \cite{LHY2011} and \cite{KY14}. The largest eigenvalues prove to be fruitful objects of study, playing an important role in
multivariate statistical analysis such as principle component analysis (PCA),
multivariate analysis of variance (MANOVA), canonical correlation analysis (CCA) and discriminant analysis. Among the vast literature, we refer the readers to a seminal work \cite{J08}, as well as a recent work \cite{WY}. 
Johnstone in \cite{J08} considered a double Wishart setting and developed the Tracy-Widom law of its largest root when the dimension of the data matrix $\bbX$ and the sample size are comparable with the dimension being even. This limiting distribution can be applied to conduct various statistical inferences in his companion paper \cite{J09}. Considering that the results in \cite{J08} work for the Gaussian distribution only, the authors in \cite{WY} investigated an F type matrix for the general distributions without even dimension restriction. 
 However, one may notice that the Tracy-Widom result in \cite{WY} is only verified for zero mean data.


 We now set a stage to present our matrix model. The most initial motivation is the matrix frequently used in CCA.
Suppose that we are given two sets of random variables, organized into two random vectors $\bbx$ and $\bby$ with dimensions $M_1$ and $M_2$, respectively.
Without loss of generality, we may assume that $M_1\leq M_2$.
In multivariate analysis, CCA is the favorite method to investigate the correlation structure between two random vectors, which was introduced by Hotelling \cite{H1936} first. The aim of CCA is to seek two vectors $\bba$ and $\bbb$ such that the linear combination of $\bba^T\bbx$ and $\bbb^T\bby$ can get the highest correlation coefficient. i.e.
\begin{eqnarray}\label{1216.2}
\rho(\bba,\bbb)\coloneqq \frac{Cov(\bba^T\bbx,\bbb^T\bby)}{\sqrt{Var(\bba^T\bbx)}\sqrt{Var(\bbb^T\bby)}}.
\end{eqnarray}
If $\rho_1=\rho_1(\bba_1,\bbb_1)\coloneqq \max\limits_{\bba,\bbb}\rho(\bba, \bbb)$, then $\rho_1$ is called the first canonical correlation coefficient. Given the first canonical correlation coefficient, one can continue to seek the second canonical correlation coefficient which is the maximum correlation coefficient of $\bba_2^T\bbx$ and $\bbb_2^T\bby$, uncorrelated to $\bba_1^T\bbx$ and $\bbb_1^T\bby$. Iterating this procedure to the end, we can get the canonical correlation coefficients $\rho_1$, $\rho_2$,...,$\rho_{M_1}$. Denote the population covariance matrix of any two random vectors $\bbu$ and $\bbv$ by $\Sigma_{\bbu\bbv}$.
By (\ref{1216.2}), it is not hard to conclude that in order to find the population canonical correlation coefficients $\rho_1$, $\rho_2$,...,$\rho_{M_1}$, one only need to solve the determinant equation
\begin{eqnarray}\label{1216.3}
det(\Sigma_{\bbx\bby}\Sigma^{-1}_{\bby\bby}\Sigma_{\bbx\bby}^T-\rho^2\Sigma_{\bbx\bbx})=0.
\end{eqnarray}
 If $\bbx$ and $\bby$ are independent, then $\rho_1^2=\cdots=\rho_{M_1}^2=0$ or equivalently the largest eigenvalue of $\Sigma_{\bbx\bbx}^{-1}\Sigma_{\bbx\bby}\Sigma_{\bby\bby}^{-1}\Sigma_{\bbx\bby}^T$, $\rho_1^2=0$. For the moment, we assume that $\mathbb{E}\bbx=\mathbb{E}\bby=0$ for ease of illustration, but bearing in mind that such conditions are not needed in this work. Then under the classical low-dimensional setting, i.e., both $M_1$ and $M_2$ are fixed but $N$ is large, one can safely use $\gamma_1$, the largest eigenvalue of $\bbA_{\bbx\bbx}^{-1}\bbA_{\bbx\bby}\bbA_{\bby\bby}^{-1}\bbA_{\bbx\bby}^T$, to estimate
$\rho_1^2$ since the sample covariance matrices converge to their population
counterparts as $N$ tends to infinity, where
$$\bbA_{\bbx\bbx}=\bbX\bbX^T,\ \ \bbA_{\bby\bby}=\bbY\bbY^T, \ \ \bbA_{\bbx\bby}=\bbX\bbY^T.$$
However, when $M_1$ and $M_2$ are comparable with the sample size $N$, the consistency will no longer hold for the sample covariance matrices and accordingly the largest sample canonical correlation coefficient $\gamma_1$. Putting forward a theory on high-dimensional
CCA is then much needed.

If $\bbx$ or $\bby$ is Gaussian distributed, it is not difficult to derive that the largest eigenvalue of $\bbS_{xy}=\bbA_{\bbx\bbx}^{-1}\bbA_{\bbx\bby}\bbA_{\bby\bby}^{-1}\bbA_{\bbx\bby}^T$ reduces to that of the double Wishart matrices in \cite{J08}, see the equation (\ref{0603.1}) below. Thus after centralizing and re-scaling, it converges to the Type-1 Tracy-Widom distribution as proved in \cite{J08} and \cite{WY}.
However, to our best knowledge, corresponding results are not yet available for non-gaussian distributions,  which is the starting point of this paper. Here we would also remark some other existing work about CCA in the high dimensional case. Central limit theorems of linear spectral statistics of CCA have been established in \cite{YP2015}, which is for zero mean data, while spiked eigenvalues are investigated for CCA in \cite{BHPZ}. There are also a lot of existing work about sparse CCA and we mention \cite{GZ} among others.



Denote the largest eigenvalue of $\bbS_{xy}$ by $\gamma_1$.
Then $\gamma_1$ is also the largest eigenvalue of
$\bbT_{xy}\coloneqq \bbP_{y}\bbP_{x}\bbP_{y}$,
where
$$\bbP_{x}=\bbX^T(\bbX\bbX^T)^{-1}\bbX,\ \ \bbP_{y}=\bbY^T(\bbY\bbY^T)^{-1}\bbY.$$ Equivalently, it is
the largest solution to
  $\det(\bbX\bbP_y\bbX^T-\gamma_1\bbX\bbX^T)=0$.
 Define $\lambda_1=\frac{\gamma_1}{1-\gamma_1}$. Then under the condition that $\liminf_{N\rightarrow \infty}\frac{N}{M_1+M_2}>1$, $\lambda_1$ is also the largest solution of$$\det(\bbX\bbP_y\bbX^T-\lambda_1\bbX(\bbI-\bbP_{y})\bbX^T)=0.$$
The matrix of interest now becomes
 \begin{eqnarray}\label{0603.1}
  (\bbX(I-\bbP_y)\bbX^T)^{-1}\bbX\bbP_y\bbX^T.
  \end{eqnarray}

Inspired by (\ref{0603.1}), we propose a unified matrix model
\begin{equation}\label{b11}\bbom=(\bbZ\bbU_2\bbU_2^T\bbZ^T)^{-1}\bbZ\bbU_1\bbU_1^T\bbZ^T
\end{equation}
where $\bbU_1^T\bbU_1=\bbI_{N_1}$, $\bbU_2^T\bbU_2=\bbI_{N-N_2}$ and $\bbU_1^T\bbU_2=0$ (see (\ref{1016.3}) below for more details).  We establish the asymptotic Tracy-Widom law for its largest eigenvalue in this work. An intriguing observation is that although our Tracy-Widom approximation is theoretically established for diverging dimensions, it keeps accurate for small ones (the dimension $M_1$ can be as small as 5 in Table \ref{t1}).


The motivations behind the construction of such a matrix model $\bbom$ are illustrated as follows. First, the matrix (\ref{0603.1}) used in CCA is
 a special case of $\bbom$ by noticing that $\bbP_y$ and $\bbI-\bbP_y$ are orthogonal projection matrices. In addition, the non-zero mean data 
 can be accommodated by writing $\bbU_2\bbU_2^T=\bbP_N(I-\bbP_{Ny})\bbP_N, \bbU_1\bbU_1^T=\bbP_N\bbP_{Ny}\bbP_N$ and observing that the mean vectors can be absorbed into the matrix $\bbP_N=\bbI_N-\frac{1}{N}\bbone_N\bbone_N^T$, see Remark \ref{y0424.1} below (the definition of $\bbP_{Ny}$ is given there). Further illustrations are given in Section \ref{cca}, where we deal with the independence testing via CCA in detail.

%



Secondly, 
by selecting appropriate matrices $\bbU_1$ and $\bbU_2$ ({\bf random} or {\bf nonrandom})
the Tracy-Widom distribution for the largest eigenvalue of this unified matrix $\bbom$ can be applied to the other multivariate testing problems, which cannot be done by the traditional CCA matrix (\ref{0603.1}). To see this, we explore two more applications.
One is the MANOVA approach in testing the equivalence of $g$ groups' mean vectors.
It is well known that classical MANOVA relies on the eigenvalues of the matrix $\bbV=\bbW^{-1}\bbB$, where $\bbW$ is the within sum of squares and cross-product matrix (SSCP) and $\bbB$ is the between SSCP, see \cite{ander84}. The matrix $\bbV$ can be written in terms of $\bbom$ by choosing \emph{\textbf {nonrandom}} matrices $\bbU_1$ and $\bbU_2$ as in equations (\ref{0412.1})-(\ref{y0424.2}) below, with the derivation details postponed to Section \ref{manova}. The other one is in the
multivariate linear regression model $\bbY=\bbX\bbB+\bbZ$ for testing the unknown parameter matrix $\bbB$. We consider both the linear hypothesis testing $H_0: \bbC_1\bbB=\bbgam_1$ and the general intra-subject hypothesis testing  $H_0: \bbC\bbB\bbD=\bbgam$. Taking the linear one as an example, we can rewrite its testing matrix $\bbM_1=\bbE_1^{-1}\bbH_1$ in the form of $\bbom$ by selecting \emph{\textbf {random}} matrices $\bbU_1\bbU_1^T=\bbP_{\widetilde{\bbX}}$ and $\bbU_2\bbU_2^T=\bbI-\bbP_{\bbX}$ in (\ref{0408.2}), where $\bbE_1$ is the error  SSCP and $\bbH_1$ the hypothesis SSCP described in Section \ref{mul}. Simulation results in Sections \ref{sim3}-\ref{sim4} show that the largest eigenvalue performs well in these two applications for both dense but weak alternative (DWA) and sparse but strong
alternative (SSA).

Thirdly, the matrix $\bbom$ generalizes the models in \cite{J08} and \cite{WY}. We would like to point out that if the matrix $\bbZ$ is generated from Gaussian distribution, then the two terms $(\bbZ\bbU_2\bbU_2^T\bbZ^T)$ and $(\bbZ\bbU_1\bbU_1^T\bbZ^T)$ in $\bbom$ are independent with normal entries, which reduces to the one studied in \cite{J08}.  Without this Gaussian assumption, we indeed investigate a more general case--the two terms can only be considered as uncorrelated with each other. We would also like to highlight that $\bbom$ not only covers the $F$-matrix in \cite{WY}, but also generalize it to any non-zero mean vectors by choosing some special $\bbU_2$ and $\bbU_1$. Detailed explanations will be given in Section \ref{ccagenr}.  We remark that all three applications in Sections \ref{cca}-\ref{mul} can not be done by either \cite{J08} or \cite{WY} because we neither assume Gaussian distribution for $\bbZ$ nor impose independent structure on $(\bbZ\bbU_2\bbU_2^T\bbZ^T)$ and $(\bbZ\bbU_1\bbU_1^T\bbZ^T)$.

This paper is organized as follows. In Section \ref{ccagenr}, the main theorem about the Tracy-Widom distribution for the largest eigenvalue $\lambda_1$ of the unified matrix $\bbom$ is presented. Three applications are introduced in Sections \ref{cca}, \ref{manova} and \ref{mul}, regarding the high-dimensional independence testing via CCA, MANOVA and multivariate linear regression, respectively.  Except for theoretical results developed in previous sections, we also conduct a series of
simulations in Section \ref{numer} to  investigate the accuracy of the proposed asymptotic Tracy-Widom distribution (Section \ref{sim1}) as well as its numerical performance in our three applications (Sections \ref{sim2}-\ref{sim4}).  We give an outline and some key steps for the proof of Theorem \ref{1016-1} in the appendix of Section \ref{app}, while all detailed
proofs are relegated to the supplementary material.

\section{Main result on $\bbom$}\label{ccagenr}

We investigate the largest eigenvalue of the unified matrix
\begin{eqnarray}\label{1016.3}
\bbom=(\bbZ\bbU_2\bbU_2^T\bbZ^T)^{-1}\bbZ\bbU_1\bbU_1^T\bbZ^T
 \end{eqnarray}
 in this section and develop its Tracy-Widom distribution without any specific distribution assumption.
Here $\bbZ_{M_1\times N}=\bbT^{\frac{1}{2}}\bbX$, $\bbT_{M_1\times M_1}$ can be any positive definite matrix and $\bbX=(X_{ij})_{M_1\times N}$  satisfies the following Condition \ref{0603-1}. Assume that $\bbU_1$ and $\bbU_2$ are two isometries with dimensions $N\times N_1$ and $N\times (N-N_2)$, respectively such that $N_1\leq N_2$, $\bbU_1^T\bbU_1=\bbI_{N_1}$, $\bbU_2^T\bbU_2=\bbI_{N-N_2}$ and $\bbU_1^T\bbU_2=0$. Moreover, $\bbU_1$ and $\bbU_2$ (random or non-random) are independent of $\bbX$ and with probability tending to one, $rank(\bbU_1)=N_1$ and $rank(\bbU_2)=N-N_2$. The notation ``0'' may indicate a zero value, a zero vector or a zero matrix in this paper, changing from line to line.


\begin{cond}\label{0603-1} A matrix $\bbX=(X_{ij})_{M_1\times N}$ satisfies Condition \ref{0603-1} if its entries $X_{ij}$ are independent (but not necessarily identically distributed) with all moments being finite and
\begin{eqnarray}\label{20150111}
\mathbb{E} X_{ij}=0,\ \ \mathbb{E} X_{ij}^2=\mathbb{E} X_{it}^2, \ \  1\le i\le M_1,\ \ 1\le j,t\le N.
\end{eqnarray}


\begin{rmk}\label{0406.7}
Note that the matrix $\bbT$ does not influence the largest eigenvalue of $\bbom$ and it can be any positive definite matrix. Indeed, let $\bbom_x=(\bbX\bbU_2\bbU_2^T\bbX^T)^{-1}\bbX\bbU_1\bbU_1^T\bbX^T$. One can easily observe that $\bbom$ and $\bbom_x$ share the same largest eigenvalue by the fact that $AB$ and $BA$ share the same nonzero eigenvalues.
\end{rmk}

\end{cond}

Before stating the main result we now make some comments about the relation between the matrix model $\bbom$ and the existing models in the literature.  First, as stated in the introduction, if the matrix $\bbZ$ is generated from Gaussian distribution, then the two terms $(\bbZ\bbU_2\bbU_2^T\bbZ^T)$ and $(\bbZ\bbU_1\bbU_1^T\bbZ^T)$ in $\bbom$ can be considered as independent terms with normal entries, which reduces to the matrix introduced  in the seminar work \cite{J08}. Secondly, we would like to point out that the matrix $\bbom$ not only covers the $F$-matrix model studied in \cite{WY}, but also generalizes it to the nonzero mean value case. To see this, choose $$\bbZ=(\bbY_{M_1\times n_1},\bbW_{M_1\times n_2}),\quad \bbU_2=\begin{pmatrix}\bbzero\\ \mathcal{P}_2\end{pmatrix}, \quad \bbU_1=\begin{pmatrix}\mathcal{P}_1\\ \bbzero\end{pmatrix}$$ with appropriate dimensions, respectively. Let $$\mathcal{P}_2\mathcal{P}_2^T=\bbI_{n_2}-\frac{1}{n_2}\bbone_{n_2}\bbone_{n_2}^T, \quad \mathcal{P}_1\mathcal{P}_1^T=\bbI_{n_1}-\frac{1}{n_1}\bbone_{n_1}\bbone_{n_1}^T, $$$$\mathcal{P}_2^T\mathcal{P}_2=\bbU_2^T\bbU_2=\bbI_{N-N_2}, \quad \mathcal{P}_1^T\mathcal{P}_1=\bbU_1^T\bbU_1=\bbI_{N_1},$$ where ``$\bbone_{n_i}$'' indicates an $n_i$-dimensional column vector with all entries being one ($i=1,2$). Then
\begin{eqnarray*}
\bbom&=&(\bbZ\bbU_2\bbU_2^T\bbZ^T)^{-1}\bbZ\bbU_1\bbU_1^T\bbZ^T=
(\bbW\mathcal{P}_2\mathcal{P}_2^T\bbW^T)^{-1}\bbY\mathcal{P}_1\mathcal{P}_1^T\bbY^T\\
&=&\left[\bbW(\bbI_{n_2}-\frac{1}{n_2}\bbone_{n_2}\bbone_{n_2}^T)\bbW^T\right]^{-1}
\bbY(\bbI_{n_1}-\frac{1}{n_1}\bbone_{n_1}\bbone_{n_1}^T)\bbY^T.
\end{eqnarray*}
Noticing that the data matrices $\bbW$  and $\bbY$ are centralized in  $\bbom$, we thus extend the results of $F$-matrix under the assumption of zero mean values in \cite{WY} to the nonzero mean vectors.
Finally,  by assigning other forms to $\bbU_1$ and $\bbU_2$ (either random or non-random), the matrix $\bbom$ can be used in various applications including centered and non-centered CCA, see Sections \ref{cca}-\ref{mul}. 

We now state the limiting distribution  for the largest eigenvalue of the unified matrix $\bbom$. 

\begin{thm}\label{1016-1}
Consider the matrix $\bbom$ defined in (\ref{1016.3}). Suppose that $\bbT$ is any positive definite matrix and $\bbX$ satisfies  Condition \ref{0603-1}. Suppose that $\liminf\limits_{N\rightarrow \infty}\frac{N}{M_1+N_2}>1$,  $N_1\leq N_2$, $\frac{N_1}{N_2}$ and $\frac{M_1}{N-N_2}$ are both bounded away from $0$,
 and $\frac{N_1}{M_1}$ is bounded away from $0$ and $\infty$. Denote the largest eigenvalue of  $\bbom$  by $\lambda_1$. Then there exist $\mu_N$ and $\sigma_N$ such that
\begin{equation}\label{h0417.1}
\lim_{N \rightarrow \infty} P(\sigma_{N} N_1^{2/3} (\lambda_{1}-\mu_{N}) \leq s)=F_1(s),
\end{equation}
where $F_1(s)$ is the Type-1 Tracy-Widom distribution. 
Moreover, the mean
$\mu_N$ and variance $\sigma_N$ can be decided as follows. 
Suppose that $c_{N} \in [0, (1-\sqrt{\frac{M_1}{N-N_2}})^2]$ satisfies the equation
\begin{equation}\label{2.7}
\int_{- \infty}^{+\infty} (\frac{c_{N}}{\lambda-c_{N}})^2 dF(\lambda)=\frac{N_1}{M_1},
\end{equation}
where $F(\lambda)$ is the limit spectral density (LSD) of $(\bbX\bbU_2\bbU_2^T\bbX^T)^{-1}$.
Then
\begin{equation}\label{2.8}
\mu_{N}=\frac{1}{c_{N}}(1+\frac{M_1}{N_1}\int_{- \infty}^{+\infty} (\frac{c_{N}}{\lambda-c_{N}}) dF(\lambda))
\end{equation}
and
\begin{equation}\label{2.9}
\frac{1}{\sigma_{N}^3}=\frac{1}{c_{N}^3}(1+\frac{M_1}{N_1}\int_{- \infty}^{+\infty} (\frac{c_{N}}{\lambda-c_{N}})^3 dF(\lambda)).
\end{equation}
\end{thm}

\begin{rmk}\label{h0424-1}
The LSD of the empirical spectral distribution of $(\bbX\bbU_2\bbU_2^T\bbX^T)$ (equivalent to the sample covariance matrix in the Gaussian case) is the famous Marcenko Pastur distribution. From there one can easily find $F(\lambda)$.
\end{rmk}

\begin{rmk}\label{b16}
When $\bbX$ is a complex random matrix, Theorem \ref{1016-1} still holds but the Tracy-Widom distribution $F_1(s)$ should be replaced by $F_2(s)$. One may refer to \cite{TW1996} for the definitions of $F_i(s),i=1,2$.
\end{rmk}

\begin{rmk}\label{h0424-1} The condition $\bbU_1^T\bbU_2=0$ imposed on the matrices $\bbU_1$ and $\bbU_2$ can be relaxed to $\bbU_1^T\bbU_2=(\bbI_{N_1}, 0)$.
 In fact,
if $\bbU_1^T\bbU_2=(\bbI_{N_1}, 0)$, then we can write $\bbU_2$ as $\bbU_2=(\bbU_1,\bbU_4)$ such that $\bbU_1^T\bbU_4=0$. This is because if we denote $\bbU_2=(\bbU_3, \bbU_4)$, then the relation $\bbU_1^T\bbU_2=(\bbI_{N_1}, 0)$ suggests that $\bbU_1^T\bbU_3=\bbI_{N_1}$, $\bbU_1^T\bbU_4=0$ . Denoting the $i$-th columns of $\bbU_1$ and $\bbU_3$ by $\bbu_{1i}$ and $\bbu_{3i}$ respectively, we have $\bbu_{1i}^T\bbu_{3i}=1$. By the Cauchy-Schwarz inequality, we see that
$$1=\bbu_{1i}^T\bbu_{3i}\le \|\bbu_{1i}\|\|\bbu_{3i}\|=1,$$
which forces $\bbu_{1i}=\bbu_{3i}$ and consequently $\bbU_1=\bbU_3$, $\bbU_2=(\bbU_1,\bbU_4)$ with $\bbU_1^T\bbU_4=0$.
By the arguments above (\ref{0603.1}), the largest eigenvalue of $$(\bbX\bbU_2\bbU_2^T\bbX^T)^{-1}\bbX\bbU_1\bbU_1^T\bbX^T=
(\bbX\bbU_1\bbU_1^T\bbX^T+\bbX\bbU_4\bbU_4^T\bbX^T)^{-1}\bbX\bbU_1\bbU_1^T\bbX^T$$
can be transferred to a function of the largest eigenvalue of  $(\bbX\bbU_4\bbU_4^T\bbX^T)^{-1} \bbX\bbU_1\bbU_1^T\bbX^T$ so that Theorem \ref{1016-1} is applicable. Therefore, one can also work out the asymptotic distribution for the largest eigenvalue of the matrix  $(\bbX\bbU_2\bbU_2^T\bbX^T)^{-1}\bbX\bbU_1\bbU_1^T\bbX^T$ under the condition $\bbU_1^T\bbU_2=(\bbI_{N_1}, 0)$.
\end{rmk}
\begin{rmk}\label{h0531-1}
Theorem \ref{1016-1} can be extended to the joint distribution of the first $k$ largest eigenvalues, i.e.
\begin{eqnarray}\label{h0531.1}
&&\lim_{N \rightarrow \infty} P(\sigma_{N} N_1^{2/3} (\lambda_{1}-\mu_{N}) \leq s_1,..., \sigma_{N} N_1^{2/3} (\lambda_{k}-\mu_{N}) \leq s_k)\non
&&=\lim_{N \rightarrow \infty} P(N_1^{2/3} (\lambda_{1}^{GOE}-2) \leq s_1,..., N_1^{2/3} (\lambda_{k}^{GOE}-2) \leq s_k),
\end{eqnarray}
where $\lambda_{1}^{GOE}\ge...\lambda_{k}^{GOE}$ are the first $k$ largest eigenvalues of $N_1\times N_1$ GOE matrix and $k$ is a finite number independent of $N$.
In fact, such an extension can be accomplished by a discussion parallel to Corollary 3.19 of \cite{KY14} since we show the local behavior of the steitljes transform near the edge (such as Theorem \ref{0817-1}). Here we omit the proof.
\end{rmk}


A pleasant surprise from the simulated results in Section \ref{sim1} is that although our Tracy-Widom approximation is theoretically developed for large dimensions, it keeps accurate for small dimensions regardless of the data distribution, see Table \ref{t1} where even for $M_1=5$, the estimated quantiles are well matched with theoretical ones.

In the next three sections, we propose three applications of this limiting Tracy-Widom distribution for $\lambda_1$.  The first one is our motivation of studying $\bbom$ as stated in the introduction, the high-dimensional independence testing by using canonical correlation analysis. The second one is the MANOVA approach in testing the equivalence of $g$ groups' mean vectors. And the last one is the  unknown parameter matrix testing in the multivariate linear model.

\section{Unified matrix in CCA}\label{cca}

Suppose that we have two sets of random variables, organized into two random vectors $\bbz=(z_1,\cdots,z_{M_1})^T$ and $\bby=(y_1,\cdots,y_{M_2})^T$, with mean vectors and covariance matrices ($\bbmu_z,\Sigma_{\bbz\bbz}$) and ($\bbmu_y,\Sigma_{\bby\bby}$), respectively. For each of them, $N$ observations are measured and the data matrices are denoted as $\bbZ=(\bbz_1,\cdots,\bbz_N)_{M_1\times N}$ and $\bbY=(\bby_1,\cdots,\bby_N)_{M_2\times N}$. We want to test
\begin{equation}\label{0408.1}
H_0: \quad \bbz \ \ \text{and}\ \  \bby \ \ \text{are independent}.
\end{equation}
 As illustrated in the introduction, if $\bbz$ and $\bby$ are independent, the largest eigenvalue $\rho_1^2$ of the matrix $\Sigma_{\bbz\bbz}^{-1}\Sigma_{\bbz\bby}\Sigma_{\bby\bby}^{-1}\Sigma_{\bbz\bby}^T$ should be zero. The corresponding sample version is
 {\small
\begin{eqnarray}\label{y0419.1}
\bbS_{zy}&=&\left(\sum\limits_{i=1}^N(\bbz_i-\bar\bbz)(\bbz_i-\bar\bbz)^T\right)^{-1}
\left(\sum\limits_{i=1}^N(\bbz_i-\bar\bbz)(\bby_i-\bar\bby)^T\right)\\
&\times&\left(\sum\limits_{i=1}^N(\bby_i-\bar\bby)(\bby_i-\bar\bby)^T\right)^{-1}
\left(\sum\limits_{i=1}^N(\bbz_i-\bar\bbz)(\bby_i-\bar\bby)^T\right)^T\nonumber\\
&=&(\bbZ\bbP_N\bbZ^T)^{-1}(\bbZ\bbP_N\bbY^T)(\bbY\bbP_N\bbY^T)^{-1}(\bbZ\bbP_N\bbY^T)^T,
\end{eqnarray}}
where $\bbP_N=\bbI_N-\frac{1}{N}\bbone_N\bbone_N^T$ and $\bbone_{N}$ indicates an $N$-dimensional column vector with all entries being one. Denote the largest eigenvalue of $\bbS_{zy}$ by $\gamma_1^{\bbS}$ and let $\lambda_1^{\bbS}=\frac{\gamma_1^{\bbS}}{1-\gamma_1^{\bbS}}$. Note that $\bbP_N$ is a projection matrix. Then the property of $\lambda_1^{\bbS}$ is a special case of $\lambda_1$ in Theorem \ref{1016-1} by observing that  we can equivalently consider $\lambda_1^{\bbS}$ as the largest eigenvalue of the matrix
$$(\bbZ\bbP_N(I-\bbP_{Ny})\bbP_N\bbZ^T)^{-1}\bbZ\bbP_N\bbP_{Ny}\bbP_N\bbZ^T,$$ where $\bbP_{Ny}=(\bbY\bbP_N)^T(\bbY\bbP_N\bbY^T)^{-1}(\bbY\bbP_N)$. This equivalence has been specified in the introduction, see the derivation of (\ref{0603.1}). It is easy to check that $(\bbP_N(I-\bbP_{Ny})\bbP_N)(\bbP_N\bbP_{Ny}\bbP_N)=0$.
Since both $\bbP_N(I-\bbP_{Ny})\bbP_N$ and $\bbP_N\bbP_{Ny}\bbP_N$ are projection matrices such that $rank(\bbP_N(I-\bbP_{Ny})\bbP_N)=N-M_2$ and $rank(\bbP_N\bbP_{Ny}\bbP_N)=M_2$ with high probability by Lemma \ref{1121-1} in the supplement, we can take $N_1=N_2=M_2$ in Theorem \ref{1016-1} to obtain the following Corollary \ref{0413.1}.

\begin{coro}\label{0413.1}
Suppose that the data matrix $\bbZ$ can be written as $\bbZ=\bbT^{\frac{1}{2}}\bbX+\bbmu_z\bbone_N^T$ for some positive definite matrix $\bbT$ and the matrix $\bbX_{M_1\times N}$ satisfies Condition \ref{0603-1}. We do not impose any condition on the random vector $\bby$. Here $\bbmu_z$ is the mean vector of $\bbz$ and can be any possible value. Assume that $\liminf\limits_{N\rightarrow \infty}\frac{N}{M_1+M_2}>1$, $\frac{M_1}{N-M_2}$ is bounded away from $0$ and $\frac{M_2}{M_1}$ is bounded away from $0$ and $\infty$. Denote the largest eigenvalue of $\bbS_{zy}$ by $\gamma_1^{\bbS}$ and let $\lambda_1^{\bbS}=\frac{\gamma_1^{\bbS}}{1-\gamma_1^{\bbS}}$. Then under the null hypothesis (\ref{0408.1}), there exist $\mu_N$ and $\sigma_N$ such that
$$\lim_{N \rightarrow \infty} P(\sigma_{N} M_2^{2/3} (\lambda_{1}^{\bbS}-\mu_{N}) \leq s)=F_1(s),$$
where $F_1(s)$ is the Type-1 Tracy-Widom distribution.
Denote the LSD of $(\bbX\bbP_N(I-\bbP_{Ny})\bbP_N\bbX^T)^{-1}$ by $F(\lambda)$ and suppose that $c_{N} \in [0, (1-\sqrt{\frac{M_1}{N-N_1}})^2)$. Then the mean $\mu_N$ and the variance $\sigma_N$ can be decided in
the same way as in Theorem \ref{0603-1} by replacing $N_1$ and $N_2$ with $M_2$.
\end{coro}

According to Corollary \ref{0413.1}, we suggest to use $\lambda_1^{\bbS}$ for the hypothesis testing (\ref{0408.1}) by comparing the rescaled $\lambda_1^{\bbS}$ value with the theoretical critical point obtained from the Type-1 Tracy-Widom distribution. One can also refer to the numerical studies in Section \ref{sim2}.

\begin{rmk}\label{y0424.1}
One may notice that there is an additional term $\bbmu_z\bbone_N^T$ in the expression of $\bbZ$ in Corollary \ref{0413.1} compared with the one in Theorem \ref{1016-1}. This allows the mean vectors to be any possible values. We would like to point out that this mean vector does not influence the analysis of $\lambda_1^{\bbS}$ due to the observation that $\bbmu_z\bbone_N^T\bbP_N=0$.
\end{rmk}


\begin{rmk}
For the Tracy-Widom distribution in Corollary \ref{0413.1}, a similar result can be concluded if we assume that the data matrix $\bbY=\bbT^{\frac{1}{2}}\bbX+\bbmu_y\bbone_N^T$ for some positive definite matrix $\bbT$ instead and $\bbmu_y$ is the mean vector of $\bby$. In this case, no condition is imposed on the random vector $\bbz$. And we only need to exchange the roles of $M_1$ and $M_2$ in the conclusions of Corollary \ref{0413.1}. This is easy to see according to the fact that the largest eigenvalue of $\bbS_{zy}$ does not change if the roles of $\bbZ$ and $\bbY$ are exchanged in (\ref{y0419.1}).
\end{rmk}

 \begin{rmk}
For the case $N< M_1+M_2$, it is trivial that $\gamma_1^{\bbS}\equiv 1$ and $\lambda_1^{\bbS}=+\infty$.
\end{rmk}

\section{Unified matrix in multivariate analysis of variance (MANOVA)}\label{manova}
Suppose that we have $g$ populations. Let $n_i$ samples ($\bbx_{i1},\cdots, \bbx_{in_i}$) be  available from the $i$th population with mean vector $\bbmu_i$ ($p$-dimensional) and common covariance matrix $\Sigma$ ($i=1,\cdots, g$). The total sample size is denoted by $n=\sum_{i=1}^g n_i$. One frequently discussed problem in multivariate analysis is to investigate whether the $g$ groups have the same mean vector. i.e.
\begin{equation}\label{0409.1}
H_0:\quad \bbmu_1=\cdots=\bbmu_g.
\end{equation}
The MANOVA approach is well-known for this testing problem. Two main SSCPs,  the between SSCP $\bbB$ and the within SSCP $\bbW$ are constructed as
\[
\bbB=\sum_{i=1}^g n_i(\bar\bbx_i-\bar\bbx)(\bar\bbx_i-\bar\bbx)^T,\quad
\bbW=\sum_{i=1}^g\sum_{j=1}^{n_i} (\bbx_{ij}-\bar\bbx_i)(\bbx_{ij}-\bar\bbx_i)^T,
\]
where $\bar\bbx_i=\frac{1}{n_i}\sum\limits_{j=1}^{n_i} \bbx_{ij}$ is the $i$-th group sample mean and $\bar\bbx=\frac{1}{n}\sum\limits_{i=1}^g\sum\limits_{j=1}^{n_i} \bbx_{ij}=\sum\limits_{i=1}^g\frac{n_i}{n}\bar\bbx_i$ is the overall sample mean. The classical testing methods for (\ref{0409.1}) are based on the eigenvalues of the matrix $\bbV=\bbW^{-1}\bbB$.  We can show that under the null hypothesis (\ref{0409.1}), the matrix $\bbV$ can be written as a special form of $\bbom$ in Section \ref{ccagenr} and thus the limiting distribution of its largest eigenvalue $\lambda_1^{\bbV}$ follows from Theorem \ref{0603-1}.
%



To see this, denote
$\bbX_i=(\bbx_{i1},\cdots,\bbx_{in_i})^T$ (of size $n_i\times p$). Note that under the null hypothesis (\ref{0409.1}), the common mean vector does not influence the matrix $\bbV$. Then without loss of generality, we can simply assume that $\bbmu_1=\cdots=\bbmu_g=0$ under $H_0$.   In this section, we use $i$ to denote the $i$th group ($i=1,\cdots,g$) and use $j$ to denote the $j$th observation from the $i$th group ($j=1,\cdots, n_i$). For each $\bbX_i$, let $\bbH_i$ be an $n_i\times n_i$ orthogonal matrix with the first column being $\frac{1}{\sqrt{n_i}}\bbone_{n_i}$. Here $\bbone_{n_i}$ indicates an $n_i$-dimensional column vector with all entries being one. The matrix $\bbI_{n_i}$ indicates an $n_i\times n_i$ identity matrix, $\bbU_{i1}$ indicates the first column of $\bbI_{n_i}$ and $\bbU_{i2}$ indicates the remaining $n_i\times (n_i-1)$ block of $\bbI_{n_i}$. An intuitive example for easy understanding when $n_1=3$ is
\[
\bbI_{n_1}=\begin{pmatrix}1&0&0\\0&1&0\\0&0&1\end{pmatrix}, \quad \bbU_{11}=\begin{pmatrix}1\\0\\0\end{pmatrix},\quad \text{and}\quad
\bbU_{12}=\begin{pmatrix}0&0\\1&0\\0&1\end{pmatrix}.
\]
Arrange these $\bbU_{i1}$'s as blocks placed on the diagonal of a block matrix $\bbU_1$ and $\bbU_{i2}$'s as blocks placed on the diagonal of another block matrix $\bbU_2$, i.e.
{\footnotesize
\begin{equation}\label{0412.1}
\bbU_1 =
 \begin{pmatrix}
  \underset{(n_1\times 1)}{\bbU_{11}} &  &  &  \\
   & \underset{(n_2\times 1)}{\bbU_{21}} &  &  \\
   & & \ddots &   \\
   &  &  & \underset{(n_g\times 1)}{\bbU_{g1}}
 \end{pmatrix}_{n\times g},\quad
 \bbU_2 =
 \begin{pmatrix}
  \underset{(n_1\times (n_1-1))}{\bbU_{12}} &  &  &  \\
   & \underset{(n_2\times (n_2-1))}{\bbU_{22}} &  &  \\
   & & \ddots &   \\
   &  &  & \underset{(n_g\times (n_g-1))}{\bbU_{g2}}
 \end{pmatrix}_{n\times (n-g)}.
\end{equation}}
Consider the orthogonal transformations $\bbZ_i=(\bbz_{i1},\bbz_{i2},\cdots, \bbz_{in_i})^T=\bbH_i^T\bbX_i$ (of size $n_i\times p$). It is easy to find that $\bbz_{i1}=\sqrt{n_i}\bar\bbx_i$.
Furthermore, denote $\bba_g=(\sqrt{\frac{n_1}{n}},\cdots,\sqrt{\frac{n_g}{n}})^T$,  $\bbP_g=\bbI_g-\bba_g\bba_g^T$ and $\bbZ=({\bbZ}_1^T, {\bbZ}_2^T,\cdots,{\bbZ}_g^T)_{p\times n}$.
Considering the relationship $\sqrt{n}\bar\bbx=(\bbz_{11},\cdots,\bbz_{g1})\bba_g$,
we can obtain
\begin{eqnarray}\label{y0424.2}
\bbB&=&\sum_{i=1}^g n_i(\bar\bbx_i-\bar\bbx)(\bar\bbx_i-\bar\bbx)^T=\sum_{i=1}^g n_i\bar\bbx_i\bar\bbx_i^T-\sqrt{n}\bar\bbx\cdot \sqrt{n}\bar\bbx\nonumber\\
&=&(\bbz_{11},\cdots,\bbz_{g1})(\bbI_g-\bba_g\bba_g^T)(\bbz_{11},\cdots,\bbz_{g1})^T=\bbZ\bbU_1\bbP_g\bbU_1^T\bbZ^T
=\bbZ\widetilde{\bbU}_1\widetilde{\bbU}_1^T\bbZ^T,\nonumber\\
\bbW&=&\sum_{i=1}^g\sum_{j=1}^{n_i} (\bbx_{ij}-\bar\bbx_i)(\bbx_{ij}-\bar\bbx_i)^T
=\sum_{i=1}^g(\sum_{j=1}^{n_i}\bbx_{ij}\bbx_{ij}^T-n_i\bar\bbx_i\bar\bbx_i^T)=\sum_{i=1}^g(\bbX_i^T\bbX_i-\bbz_{i1}\bbz_{i1}^T)\nonumber\\
&=&\sum_{j=2}^{n_1} \bbz_{1j}\bbz_{1j}^T+\sum_{j=2}^{n_2} \bbz_{2j}\bbz_{2j}^T+\cdots+\sum_{j=2}^{n_g} \bbz_{gj}\bbz_{gj}^T=\bbZ\bbU_2\bbU_2^T\bbZ^T,
\end{eqnarray}
where $\widetilde{\bbU}_1=\bbU_1\bbP_g$ and $\mathbb{E}(\bbZ)=0$ under $H_0$. According to the construction of $\bbU_1$ and $\bbU_2$ in  (\ref{0412.1}), we can easily conclude that $\widetilde{\bbU}_1^T\bbU_2=0$.
Then the limiting distribution of the largest eigenvalue $\lambda_1^{\bbV}$ of $$\bbV=\bbW^{-1}\bbB=(\bbZ\bbU_2\bbU_2^T\bbZ^T)^{-1}\bbZ\widetilde{\bbU}_1\widetilde{\bbU}_1^T\bbZ^T$$ can follow from Theorem \ref{1016-1} by assigning $M_1=p$, $N_1=g-1$ and $N_2=g$ since $rank(\widetilde{\bbU}_1)=g-1, rank(\bbU_2)=n-g$. See the following Corollary \ref{0412.2}.

 \begin{coro}\label{0412.2}
 Consider the multivariate mean vectors' hypothesis testing problem in (\ref{0409.1}). We use the largest eigenvalue $\lambda_1^{\bbV}$ of the matrix $\bbV=\bbW^{-1}\bbB$ as the test criterion.  Under the null hypothesis, suppose that $\bbZ$ can be written as $\bbZ=\bbT^{\frac{1}{2}}\bbX$ for some positive definite matrix $\bbT_{p\times p}$ and the matrix $\bbX_{p\times n}$ satisfies  Condition \ref{0603-1}. Assume that $\liminf\limits_{n\rightarrow \infty}\frac{n}{p+g}>1$ and $\frac{g-1}{p}$ is bounded away from $0$ and $\infty$.  Then there exist $\mu_n$ and $\sigma_n$ such that
$$\lim_{n \rightarrow \infty} P(\sigma_{n} (g-1)^{2/3} (\lambda_{1}^{\bbV}-\mu_{n}) \leq s)=F_1(s),$$
where $F_1(s)$ is the Type-1 Tracy-Widom distribution.
The mean
$\mu_n$ and the variance $\sigma_n$ can be decided in the same way as in Theorem \ref{1016-1} by replacing $M_1$ with $p$ and $N_1$ with $(g-1)$.
%
 \end{coro}

According to Corollary \ref{0412.2}, if the rescaled $\lambda_1^{\bbV}$ value is smaller than the theoretical critical point obtained from  Type-1 Tracy-Widom distribution, we fail to reject the null hypothesis (\ref{0409.1}), i.e. we do not reject that the $g$ groups share the same mean vector. Otherwise, reject $H_0$.
In the simulation studies of Section \ref{sim3}, regarding the pattern of different mean vectors under the alternative, we consider two cases. One is the dense but weak alternative (DWA),  which means that there are many different entries among the mean vectors, but these differences are faint, see the setting $H_1^{(1)}$ and $H_1^{(1)'}$ in Section \ref{sim3} {\bf (1)}. The other one  is the sparse but strong
alternative (SSA),  which means that the differences are rare, but significant where they
appear, see the alternative $H_1^{(2)}$, where the differences only appear in one out of $p$ components. The numerical results in Table \ref{t3} indicate that this $\lambda_1^{\bbV}$ shows satisfactory performance for both alternatives.

\begin{rmk}
If we assume that all the observations come from multivariate normal distribution as in the classical setting, then the positive definite matrix $\bbT$ in Corollary \ref{0412.2} obviously exists by choosing $\bbT=\Sigma$. This is due to the fact that we can write each $\bbX_i$ as $\bbX_i^T=\Sigma^{\frac{1}{2}}\widetilde{\bbX}_i=\bbT^{\frac{1}{2}}\widetilde{\bbX}_i$ and the entries of $\widetilde{\bbX}_i$ are i.i.d $N(0,1)$. Then
\begin{eqnarray*}
\bbZ=({\bbZ}_1^T, {\bbZ}_2^T,\cdots,{\bbZ}_g^T)=(\bbX_1^T\bbH_1,\bbX_2^T\bbH_2,\cdots,\bbX_g^T\bbH_g)
=\bbT^{\frac{1}{2}}(\widetilde{\bbX}_1\bbH_1,\widetilde{\bbX}_2\bbH_2,\cdots,\widetilde{\bbX}_g\bbH_g)\coloneqq \bbT^{\frac{1}{2}}\bbX,
\end{eqnarray*}
where $\bbX=(\widetilde{\bbX}_1\bbH_1,\widetilde{\bbX}_2\bbH_2,\cdots,\widetilde{\bbX}_g\bbH_g)_{p\times n}$ satisfies Condition \ref{0603-1}, taking into account the orthogonality of each $\bbH_i$ and the independence among each $\widetilde{\bbX}_i$.

\end{rmk}

\section{Unified matrix in high-dimensional multivariate linear model}\label{mul}
 In this section, we investigate one more application of the unified matrix $\bbom$ in the multivariate linear model. Let us consider a linear relationship between $p_2$ response variables $y_1,\cdots,y_{p_2}$ and  $p_1$ explanatory variables $x_1,\cdots, x_{p_1}$. Suppose that there are $N$ observations, organized into two data matrices:
 \[
 \bbY=\begin{pmatrix}
 \bbY_1^T\\
 \vdots\\
 \bbY_N^T
 \end{pmatrix}_{N\times p_2},\quad
 \bbX=\begin{pmatrix}
 \bbX_1^T\\
 \vdots\\
 \bbX_N^T
 \end{pmatrix}_{N\times p_1}.
 \]
 Then the multivariate linear model assumes that
 \begin{equation}\label{0403.1}
\bbY=\bbX\bbB+\bbZ,
 \end{equation}
 where $\bbB$ is a $p_1\times p_2$ unknown parameter matrix and $\bbZ$ is a $N\times p_2$ error matrix with the assumption that the rows of $\bbZ$ are independent having mean zero and common covariance matrix $\Sigma$. We first consider the linear hypothesis testing of the form
 \begin{equation}\label{0406.3}
 H_0:\quad \bbC_1\bbB=\bbgam_1,
 \end{equation}
 where $\bbC_1$ is a $g_1\times p_1$ known matrix of rank $g_1$ and $\bbgam_1$ is a $g_1\times p_2$ known matrix of rank $\min\{g_1,p_2\}$. As an example, in the simulation studies of Section \ref{sim4}, if we select $\bbC_1=\bbC_1^{(b)}=[\bbI_{g_1},\bbzero]$ and $\bbgam_1=\bbgam_1^{(a)}=0$, then the testing problem (\ref{0406.3}) reduces to analyzing whether the first $g_1$ rows of  $\bbB$ equal to zeros.

The initial step in conducting the linear hypothesis testing (\ref{0406.3}) is to estimate the unknown parameter matrix $\bbB$. As stated in Section \ref{ccagenr}, our proposed Tracy-Widom distribution performs well when the dimensions are small so that we can simply apply the classic least square estimator for $\bbB$, which is well-known to be
$
\hat\bbB=(\bbX^T\bbX)^{-1}\bbX^T\bbY.
$
The hypothesis SSCP for testing (\ref{0406.3}) is given by
$
\bbH_1=( \bbC_1\hat\bbB-\bbgam_1)^T[\bbC_1(\bbX^T\bbX)^{-1}\bbC_1^T]^{-1}( \bbC_1\hat\bbB-\bbgam_1)
$
and the error  SSCP is
$
\bbE_1=\bbY^T[\bbI-\bbX(\bbX^T\bbX)^{-1}\bbX^T]\bbY
$.  One can refer to chapter 7 of \cite{YVR2009} for detailed derivations. Under the null hypothesis (\ref{0406.3}), $\bbH_1$ and $\bbE_1$ can be further rewritten as
\begin{eqnarray}\label{0408.2}
\bbH_1&=&[\bbC_1(\bbX^T\bbX)^{-1}\bbX^T\bbZ]^T[\bbC_1(\bbX^T\bbX)^{-1}\bbC_1^T]^{-1}[\bbC_1(\bbX^T\bbX)^{-1}\bbX^T\bbZ]
=\bbZ^T\bbP_{\widetilde{\bbX}}\bbZ,\nonumber\\
\bbE_1&=&(\bbX\bbB+\bbZ)^T[\bbI-\bbX(\bbX^T\bbX)^{-1}\bbX^T](\bbX\bbB+\bbZ)=\bbZ^T[\bbI-\bbP_{\bbX}]\bbZ,
\end{eqnarray}
where $\widetilde{\bbX}=\bbX(\bbX^T\bbX)^{-1}\bbC_1^T$, $\bbP_{\widetilde{\bbX}}=\widetilde{\bbX}(\widetilde{\bbX}^T\widetilde{\bbX})^{-1}\widetilde{\bbX}^T$ and $\bbP_{\bbX}=\bbX(\bbX^T\bbX)^{-1}\bbX^T$. It is easy to check that $\bbP_{\widetilde{\bbX}}(\bbI-\bbP_{\bbX})=0$. Denote the largest eigenvalue of
\begin{equation}
\label{b12}
\bbM_1=\bbE_1^{-1}\bbH_1=(\bbZ^T(\bbI-\bbP_{\bbX})\bbZ)^{-1}\bbZ^T\bbP_{\widetilde{\bbX}}\bbZ
\end{equation}
by $\lambda_1^{\bbM_1}$. As stated in Section \ref{cca}, both $\bbI-\bbP_{\bbX}$ and $\bbP_{\widetilde{\bbX}}$ are projection matrices with $rank(\bbI-\bbP_{\bbX})=N-p_1$ and $rank(\bbP_{\widetilde{\bbX}})=g_1$ with high probability. Assuming $N_2=p_1$, $N_1=g_1$ and $M_1=p_2$ in Theorem \ref{1016-1}, we can develop the following corollary for $\lambda_1^{\bbM_1}$.

\begin{coro}\label{0406.6}
Assume that $\bbZ$ in the multivariate linear model (\ref{0403.1}) can be written as $\bbZ=\bbW\bbT^{\frac{1}{2}}$ for some positive definite matrix $\bbT_{p_2\times p_2}$ and the matrix $\bbW_{N\times p_2}$ satisfies Condition \ref{0603-1}. Suppose that $\liminf\limits_{N\rightarrow \infty}\frac{N}{p_2+p_1}>1$, $\frac{g_1}{p_1}$ and $\frac{p_2}{N-p_1}$ are both bounded away from $0$ and $\frac{g_1}{p_2}$ is bounded away from $0$ and $\infty$. Denote the largest eigenvalue of  $\bbM_1=\bbE_1^{-1}\bbH_1$ by $\lambda_1^{\bbM_1}$. Then under the null hypothesis (\ref{0406.3}), there exist $\mu_N$ and $\sigma_N$ such that
$$\lim_{N \rightarrow \infty} P(\sigma_{N} g_1^{2/3} (\lambda_{1}^{\bbM_1}-\mu_{N}) \leq s)=F_1(s),$$
where $F_1(s)$ is the Type-1 Tracy-Widom distribution.
Denote the LSD of $(\bbW^T(\bbI-\bbP_{\bbX})\bbW)^{-1}$ by $F(\lambda)$ and suppose that $c_{N} \in [0, (1-\sqrt{\frac{p_2}{N-g_1}})^2]$. Then
the mean $\mu_N$ and the variance $\sigma_N$ can be decided in the same way as in Theorem \ref{1016-1} by replacing $N_2$ with $p_1$, $N_1$ with $g_1$ and $M_1$ with $p_2$.
\begin{rmk}
One should notice that $\bbZ$ in this Corollary and Corollary \ref{0406.2} corresponds to $\bbZ^T$ in Theorem \ref{1016-1}. To see this, one may compare (\ref{b12}) with (\ref{1016.3}).
\end{rmk}
\end{coro}

By Corollary \ref{0406.6}, we can use $\lambda_{1}^{\bbM_1}$ for the linear hypothesis testing (\ref{0406.3}) and reject $H_0$ if the rescaled $\lambda_{1}^{\bbM_1}$ is larger than the theoretical critical point
obtained from Type-1 Tracy-Widom distribution.
 In Section \ref{sim4}, we consider the special testing of whether a certain part of $\bbB$, say $\bbB_2$, equals a zero matrix. And as in MANOVA, with regard to the pattern under the alternative, both DWA and SSA are applied, i.e. when many entries of $\bbB_2$ are nonzero but the values are small, see the third combination $(\bbC_1^{(a)},\bbB_2^{(d)},\bbgam_1^{(a)})$, as well as when only two entries of $\bbB_2$ are nonzero but the values are significant, see the last combination $(\bbC_1^{(a)},\bbB_2^{(s)},\bbgam_1^{(a)})$. The numerical results in Table \ref{t4}  show that $\lambda_1^{\bbM_1}$  performs well under both alternatives.


We next consider the intra-subject hypothesis testing of the form
 \begin{equation}\label{0406.1}
H_0:\quad \bbC\bbB\bbD=\bbgam,
 \end{equation}
where
 $\bbC$ is a $g_1\times p_1$ known matrix of rank $g_1$, $\bbD$ is a $p_2\times g_2$ known matrix of rank $g_2$ and $\bbgam$ is a $g_1\times g_2$ known matrix of rank $\min\{g_1,g_2\}$.  The hypothesis and error SSCPs for (\ref{0406.1}) can be obtained from $\bbH_1$ and $\bbE_1$ by modifying the multivariate linear model (\ref{0403.1}) to the following expression
\[
\bbY\bbD=\bbX\bbB\bbD+\bbZ\bbD.
\]
Replacing $\bbY,\bbB$ and $\bbZ$ by $\bbY\bbD,\bbB\bbD$ and $\bbZ\bbD$ respectively, we can then conclude that the SSCPs for conducting the hypothesis testing (\ref{0406.1}) are
\begin{equation}\label{0408.3}
\bbH=(\bbZ\bbD)^T\bbP_{\widetilde{\bbX}}(\bbZ\bbD),\quad
\bbE=(\bbZ\bbD)^T[\bbI-\bbP_{\bbX}](\bbZ\bbD),
\end{equation}
where $\widetilde{\bbX}=\bbX(\bbX^T\bbX)^{-1}\bbC^T$, $\bbP_{\widetilde{\bbX}}=\widetilde{\bbX}(\widetilde{\bbX}^T\widetilde{\bbX})^{-1}\widetilde{\bbX}^T$ and $\bbP_{\bbX}=\bbX(\bbX^T\bbX)^{-1}\bbX^T$. It is easy to check that $\bbP_{\widetilde{\bbX}}(\bbI-\bbP_{\bbX})=0$. Denote the largest eigenvalue of $\bbM=\bbE^{-1}\bbH$ by $\lambda_1^{\bbM}$.
The only difference between the analysis of $\lambda_1^{\bbM_1}$ and $\lambda_1^{\bbM}$ is that $\bbZ_{N\times p_2}$ in (\ref{0408.2}) is replaced by $(\bbZ\bbD)_{N\times g_2}$ in (\ref{0408.3}). So assigning $p_2=g_2$ in  Corollary \ref{0406.6}, we can obviously obtain the following conclusion for $\lambda_1^{\bbM}$.

\begin{coro}\label{0406.2}
For the known matrix $\bbD$ and the error matrix $\bbZ$ in the multivariate linear model (\ref{0403.1}), assume that $\bbZ\bbD$ can be written as $\bbZ\bbD=\bbW\bbT^{\frac{1}{2}}$ for some positive definite matrix $\bbT_{g_2\times g_2}$ and the matrix $\bbW_{N\times g_2}$ satisfies Condition \ref{0603-1}. Suppose that $\liminf\limits_{N\rightarrow \infty}\frac{N}{g_2+p_1}>1$, $\frac{g_1}{p_1}$ and $\frac{g_2}{N-p_1}$ are both bounded away from $0$  and $\frac{g_1}{g_2}$ is bounded away from $0$ and $\infty$. Denote the largest eigenvalue of  $\bbM=\bbE^{-1}\bbH$ by $\lambda_1^{\bbM}$. Then under the null hypothesis (\ref{0406.3}), there exist $\mu_N$ and $\sigma_N$ such that
$$\lim_{N \rightarrow \infty} P(\sigma_{N} g_1^{2/3} (\lambda_{1}^{\bbM}-\mu_{N}) \leq s)=F_1(s),$$
where $F_1(s)$ is the Type-1 Tracy-Widom distribution.
The mean
$\mu_N$ and the variance $\sigma_N$ can be decided in the same way as in Corollary \ref{0406.6} by replacing $p_2$ with $g_2$.
\end{coro}


\section{Numerical studies}\label{numer}
This section is to investigate the accuracy of our proposed asymptotic Tracy-Widom distribution (Section \ref{sim1}) as well as its numerical performance in various applications (Sections \ref{sim2}-\ref{sim4}). Before proceeding to the simulation results, we first introduce an asymptotic substitution of the limiting distribution for the largest eigenvalue in Theorem \ref{1016-1}.
The formulae for calculating $\mu_N$ and $\sigma_N$ in (\ref{2.7})-(\ref{2.9}) are difficult to work with. Referring to \cite{J08} and \cite{WY}, we facilitate the computation by using an approximation in terms of the log transform of $\lambda_1$ in Theorem \ref{1016-1} as
\begin{equation}\label{0414.1}
\lim_{N \rightarrow \infty} P(\frac{\ln\lambda_1-\widetilde{\mu}}{\widetilde{\sigma}} \leq s)=F_1(s),
\end{equation}
where $F_1(s)$ still indicates the Type-1 Tracy-Widom distribution and the new mean $\widetilde{\mu}$ and variance $\widetilde{\sigma}$ are defined by
\[
\widetilde{\mu}=2\ln\tan(\frac{\phi+\varphi}{2}),\quad \widetilde{\sigma}^3=\frac{16}{(N-N_2+N_1-1)^2}\frac{1}{\sin^2(\phi+\varphi)\sin\phi\sin\varphi}.
\]
The angle parameters $\phi$ and $\varphi$ are defined by
\[
\sin^2(\frac{\varphi}{2})=\frac{\min(M_1,N_1)-1/2}{N-N_2+N_1-1},\quad \sin^2(\frac{\phi}{2})=\frac{\max(M_1,N_1)-1/2}{N-N_2+N_1-1}.
\]
The asymptotic equivalence between the approximation (\ref{0414.1}) and the one in Theorem \ref{1016-1}  have been proved in \cite{J08} and \cite{WY}. All simulations in this section are conducted by adopting this $\ln\lambda_1$'s asymptotic expression. In the sequel, we also use the word ``rescaled $\lambda_1$'' to denote the term $\frac{\ln\lambda_1-\widetilde{\mu}}{\widetilde{\sigma}}$ in (\ref{0414.1}).  The values of $\widetilde{\mu}$ and $\widetilde{\sigma}$ in the applications can be obtained simply by replacing $N,N_1,N_2,M_1$  with their corresponding notations in Sections \ref{cca}-\ref{mul}.  All simulated results below are recorded based on 10000 replications of such a re-scaled largest eigenvalue.

\subsection{Approximation accuracy}\label{sim1}
This subsection is to investigate the Tracy-Widom approximation accuracy for the unified matrix $\bbom$ in Section \ref{ccagenr}. Since the positive definite matrix $\bbT$ does not influence $\lambda_1$, we simply let $\bbT=\bbI_{M_1}$. Other settings to be used in the simulation are summarized below.

\begin{description}
\item [(1).] {\bf Data distribution: } Three data distributions will be used to generate the entries of $\bbX$ in the model (\ref{1016.3}).
\begin{itemize}
\item Data 1: Standard Normal distribution $N(0,1)$.
\item Data 2: Discrete distribution with probability mass function $P(x=-\sqrt{3})=P(x=\sqrt{3})=1/6$ and $P(x=0)=2/3$.
\item Data 3: Standardized Gamma distribution $Gamma(4,0.5)$.
\end{itemize}
The three distributions are used to verify Condition \ref{0603-1}, i.e. for the data distribution, we do not need other restrictions except for the first two moments match and all moments are finite. Data 2 supports that the distribution can be a discrete one, while Data 3 is a skewed one with the third and fourth moments different from those of the standard normal distribution.
\item [(2).] {\bf Dimensions $(M_1,N_1,N_2,N)$:}  Considering the restrictions on the dimensions, we set two initial choices: $M^{(1)}=(M_1,N_1,N_2,N)=(5,8,10,30)$ and $M^{(2)}=(M_1,N_1,N_2,N)=(15,8,10,50)$, with $M_1$ being smaller than $(N_1,N_2)$ and larger than $(N_1,N_2)$, respectively. Then we change the magnification factor attached to the initial choices to investigate the performance when the dimensions increase. See the second row of Table \ref{t1}.
\item [(3).] {\bf Matrices $\bbU_1$ and $\bbU_2$:} We randomly generate two matrices $\bbL_{N\times N_2}$ and $\bbD_{N_2\times N_1}$ with entries from standard normal distribution.  Let $\bbU_1\bbU_1^T=(\bbL\bbD)(\bbD^T\bbL^T\bbL\bbD)^{-1}(\bbL\bbD)^T$ and $\bbU_2\bbU_2^T=\bbI_{N}-\bbL(\bbL^T\bbL)^{-1}\bbL^T$ in the model (\ref{1016.3}). It is easy to check that such settings satisfy the conditions on $\bbU_1$ and $\bbU_2$, taking into account the properties of projection matrices.
\end{description}
Simulated results based on above settings are recorded in Table \ref{t1}. The column titled ``Percentile'' lists the percentiles of Tracy-Widom distribution corresponding to quantiles in the column ``TW''. The next ten columns record our estimated cumulative probabilities (i.e. estimated quantiles)  for the rescaled $\lambda_1$ under various settings stated above, i.e. repeating 10000 times and finding 10000 rescaled $\lambda_1$'s, then the proportion of values that are less than corresponding percentiles are recorded in Table \ref{t1}. i.e. $\frac{\# \{\text{rescaled}\ \ \lambda_1\leq \text{``Percentile''}\}}{10000}$. Comparing the empirical results (the last ten columns) with the theoretical ones (the ``TW'' column), we can see that the rescaled $\lambda_1$ matches with the Tracy-Widom law quite well, which supports the accuracy of approximation in Theorem \ref{1016-1}. Moreover, although our theoretical result is developed for large dimensions, Table \ref{t1} indicates that such approximation also works well even when the dimensions are small.

\subsection{Performance in the independence testing}\label{sim2}
This subsection is to investigate the performance of our proposed largest eigenvalue $\lambda_1^{\bbS}$ in the independence testing of Section \ref{cca}. For ease of construction, we let $M_1=M_2$ and consider a series of settings for the two random vectors $\bbz$ and $\bby$ in the following way:
\[
\bbz=\sqrt{1-\tau}\bbx+\sqrt{\tau}\bby,\quad 0\leq \tau\leq 1,
\]
where two ($M_1\times 1$) random vectors  $\bbx$ and $\bby$ are independent and $\tau$ is a parameter determining  the level of dependence between $\bbz$ and $\bby$. When $\tau=0$, $\bbz$ and $\bby$ are independent, which is the null hypothesis (\ref{0408.1}) in Section \ref{cca}. Otherwise, as $\tau>0$ becomes larger, the dependence between $\bbz$ and $\bby$ increases.

Considering the conditions on the dimensions, as in Section \ref{sim1}, we also set an initial choice for $(M_1,M_2,N)$ as $M^{(0)}=(M_1,M_2,N)=(10,10,40)$ and then change the magnification factor to check the influence of dimensionality. The nominal significance level is set to be $\alpha=0.05$. According to Table \ref{t1}, the corresponding theoretical quantile value is $c_{\alpha}=0.98$. That is to say, we compare the rescaled $\lambda_1^{\bbS}$ introduced in Section \ref{cca} with $c_{\alpha}$. If it is smaller than $c_{\alpha}$, then the null hypothesis (\ref{0408.1}) is accepted, i.e. $\bbz$ and $\bby$ are independent. Otherwise, we conclude that they are dependent. We use discrete distribution or Gamma distribution, stated in above Section \ref{sim1} {\bf (1)}, to generate $N$ samples for $\bbx$ and $\bby$. Repeating 10000 times, we can find 10000 rescaled $\lambda_1^{\bbS}$'s and the proportion of values that are larger than $c_{\alpha}$ are recorded in Table \ref{t2}. i.e. $\frac{\# \{\text{rescaled}\ \ \lambda_1^{\bbS}>c_{\alpha}\}}{10000}$.\\
So when $\tau=0$, the fourth row of Table \ref{t2} records the estimated sizes, which are close to 0.05. When $\tau$ changes from 0.1 to 0.4, the corresponding rows give the estimated powers. We can observe that as the dependence between $\bbz$ and $\bby$ becomes stronger and as the dimensions become larger, the power values increase.
 We do not attach the results when $\tau>0.4$ here because the powers are always around 1. One can also expect such a phenomenon according to the trend in Table \ref{t2}.

\subsection{Performance in MANOVA}\label{sim3}
This subsection is to investigate the performance of our proposed largest eigenvalue $\lambda_1^{\bbV}$ in the MANOVA approach of Section \ref{manova}. The nominal significance level is set to be $\alpha=0.05$.  Consider $g=3$ groups with mean vectors $\bbmu_1,\bbmu_2,\bbmu_3$ and common covariance matrix $\Sigma$. We select $\Sigma$ as the covariance matrix of MA(1) model with the parameter $\theta_1=0.2$ and use Gamma distribution stated in Section \ref{sim1} {\bf(1)} to generate the data.
Other settings that will be used in the simulation are summarized below.
\begin{description}
\item [(1).]{\bf Mean vectors: } Let $\bbmu_1=\bbzero_p$, a $p$-dimensional zero vector, $\bba_1=(\tau_1,\cdots,\tau_1)^T$, a $p$-dimensional vector with all entries being $\tau_1$ and $\bba_2=(\tau_2,0,\cdots,0)^T$, a $p$-dimensional vector with only the first entry having a nonzero value $\tau_2$. Three different settings on the mean vectors are considered.
\begin{itemize}
\item $H_0$: $\bbmu_1=\bbmu_2=\bbmu_3=\bbzero_p$. This setting corresponds to the null hypothesis (\ref{0409.1}) in Section \ref{manova}. It is used to check the empirical size performance when the null hypothesis is true. Both of the following two settings are under the alternative hypothesis, i.e. the three groups do not share the same mean vector.
\item $H_1^{(1)}$ and $H_1^{(1)'}$: $\bbmu_1=\bbzero_p$, $\bbmu_2=\bbmu_1+\bba_1$ and $\bbmu_3=\bbmu_2+\bba_1$. This setting reflects the dense but weak alternative (DWA), which means that there are many different entries, but these differences are faint. We choose $\tau_1=0.2$ for $H_1^{(1)}$ and a larger $\tau_1=0.5$ for $H_1^{(1)'}$. The magnitude of the difference vector $\bba_1$ is $\|\bba_1\|^2=\tau_1^2p=0.04p$ or $0.25p$.
\item $H_1^{(2)}$: $\bbmu_1=\bbzero_p$, $\bbmu_2=\bbmu_1+\bba_2$ and $\bbmu_3=\bbmu_2+\bba_2$. This setting reflects the sparse but strong alternative (SSA), which means that the differences are rare, but significant where they appear. We choose $\tau_2=1$. Then the magnitude of the difference vector $\bba_2$ is always 1.
\end{itemize}
\item [(2).] {\bf Dimensions $(n_0,p)$:}  For simplicity,  let $n_1=n_2=n_3\coloneqq n_0$. Then $n=3n_0$. We select two initial choices for $(n_0,p)$ as $M^{(1)}=(p,n_0)=(5,8)$ and $M^{(2)}=(p,n_0)=(8,5)$, with $p<n_0$ and $p>n_0$, respectively. Then we change the magnification factor for the initial choices from 1 to 100 (see the first and sixth columns of Table \ref{t3}) to investigate the influence of dimensions on the numerical performance.
\end{description}
As in the above Section \ref{sim2}, by repeating 10000 times, we can find 10000 rescaled $\lambda_1^{\bbV}$'s and the proportion of values that are larger than $c_{\alpha}$ are recorded in Table \ref{t3}. i.e. $\frac{\# \{\text{rescaled}\ \ \lambda_1^{\bbV}>c_{\alpha}\}}{10000}$. The two columns titled ``$H_0$'' record estimated sizes, from which we can see that the size performance becomes better as the dimensions become larger. This matches with our theoretical conclusion, which relies on $n\rightarrow\infty$. Other columns report estimated powers under different mean vectors' settings. Generally speaking, the powers increase fast as the dimensions become larger, say the power values of the $8M^{(i)}$ row already all exceed 0.8. And for small dimensions, the $M^{(1)}$ domain shows better performance than $M^{(2)}$, which indicates that $\lambda_1^{\bbV}$ prefers $p<n_0$ when both $p$ and $n_0$ are small. However, for moderate and large dimensions, such preference will be weakened since all the power values are close to 1.

\subsection{Performance in multivariate linear model}\label{sim4}
This subsection is to investigate the performance of our proposed largest eigenvalue $\lambda_1^{\bbM_1}$ in the multivariate linear model of Section \ref{mul}. The nominal significance level is set to be $\alpha=0.05$. The covariance matrix $\Sigma$ of the error matrix $\bbZ$ is selected to be a Toeplize matrix with first row $(1,0.5,0.5^2,0.5^3,\cdots,0.5^{p-1})$, i.e. the covariance matrix for the AR(1) model with the parameter $\sigma_1=0.5$. And we use Gamma distribution stated in Section \ref{sim1} {\bf(1)} to generate the data $\bbZ$. According to Section \ref{mul}, the distribution of $\bbX$ does not influence the result. So we simply obtain the entries $\bbX$ from a uniform distribution $U(-2,2)$. Considering the conditions on the dimensions, we  set an initial choice for $(p_1,p_2,N)$ as $M^{(0)}=(p_1,p_2,N)=(10,6,25)$ and then change the magnification factor from 1 to 20 to check the influence of dimensionality. Other settings for the model (\ref{0403.1}) that will be used in the simulation are summarized below.
\begin{description}
\item [(1).]{\bf Parameter matrix $\bbB$: } Set
$\bbB=\begin{pmatrix}(\bbB_1)_{g_1\times p_2}\\ (\bbB_2)_{(p_1-g_1)\times p_2}\end{pmatrix}_{p_1\times p_2}$. For ease of matrix construction, we let $g_1=\frac{1}{2}p_1$ in the simulation.  $\bbB_1$ is chosen to be a $(g_1\times p_2)$ zero matrix, i.e. $\bbB_1=\bbzero_{g_1\times p_2}$. $(\bbB_2)_{g_1\times p_2}$ has two different settings.
\begin{itemize}
\item $\bbB_2^{(d)}$: All entries of $\bbB_2^{(d)}$ are generated from a discrete distribution with probability mass function $P(x=0.1)=P(x=0.2)=P(x=0.3)=1/3$.  Then this $\bbB_2^{(d)}$ consists of nonzero small components. This corresponds to the DWA (dense but weak alternative) stated in the mean vectors' setting of Section \ref{sim3}.
\item $\bbB_2^{(s)}$: The entries of $\bbB_2^{(s)}$ are all zeros except for the first 2 diagonal elements being ones, i.e. $\bbB_2^{(s)}=\begin{pmatrix}\bbI_2&\\ &\bbzero\end{pmatrix}$. This corresponds to the SSA (sparse but strong alternative) stated in the mean vectors' setting of Section \ref{sim3}.
\end{itemize}
The two different settings of $\bbB_2$ are to investigate the power performance of $\lambda_1^{\bbM_1}$ in testing (\ref{0406.3}) under different alternatives.

\item [(2).]{\bf Matrix $\bbC_1$: } We consider two special cases: $\bbC_1^{(a)}=[\bbzero,\bbI_{g_1}]$ and
$\bbC_1^{(b)}=[\bbI_{g_1},\bbzero]$.

\item [(3).]{\bf Matrix $\bbgam_1$: } $\bbgam_1$ is selected to be $\bbgam_1^{(a)}=\bbzero$ or $\bbgam_1^{(b)}=\bbB_2$.
\end{description}
Four combinations of $(\bbC_1,\bbB_2,\bbgam_1)$ are used in Table \ref{t4}. For each combination, as in previous sections, by repeating 10000 times, we can find 10000 rescaled $\lambda_1^{\bbM_1}$'s and the proportion of values that are larger than $c_{\alpha}$ are recorded in Table \ref{t4}. i.e. $\frac{\# \{\text{rescaled}\ \ \lambda_1^{\bbM_1}>c_{\alpha}\}}{10000}$.

The first two combinations are used for size testing. Since the two settings of $\bbB_2$ are constructed to investigate power performance under different alternatives, for size purpose, we just adopt one of them--$\bbB_2^{(d)}$.
The first combination $(\bbC_1^{(b)},\bbB_2^{(d)},\bbgam_1^{(a)})$ is to test whether the first ($g_1\times p_2$) block of $\bbB$ is a zero block, i.e.  $H_0: \bbB_1=\bbzero$. The second combination $(\bbC_1^{(a)},\bbB_2^{(d)},\bbgam_1^{(b)})$ is to test whether the second ($(p_1-g_1)\times p_2$) block of $\bbB$ equals to a given matrix, i.e. $H_0: \bbB_2=\bbgam_1^{(b)}$.  One can observe that the sizes are always close to 0.05, confirming the asymptotic distribution developed for $\lambda_1^{\bbM_1}$ in Section \ref{mul}.

The last two combinations are used for power testing. i.e. testing whether $\bbB_2=\bbzero$. Two alternatives are considered. The third combination $(\bbC_1^{(a)},\bbB_2^{(d)},\bbgam_1^{(a)})$ is for DWA (dense but weak alternative) and the  last one  $(\bbC_1^{(a)},\bbB_2^{(s)},\bbgam_1^{(a)})$ is for  SSA (sparse but strong alternative). We can see that for small dimensions, SSA works better than DWA, while as the dimensions increase, a reversal takes place. This is reasonable because the magnitude of difference for DWA is much involved by values of dimensions. And for appropriate large dimensions, all power values are close to 1.

\begin{table}[!htbp]
\begin{center}
{\footnotesize
\caption{\label{t1} Simulated quantiles for rescaled $\lambda_1$, i.e. the values $\frac{\# \{\text{rescaled}\ \ \lambda_1\leq \text{``Percentile''}\}}{10000}$  based on 10000 replications under different data distributions and different dimensions.}
\begin{tabular}{c c| c c c c c| c c c c c}
\hline
\multicolumn{2}{c|}{Standard Normal}&\multicolumn{5}{c|}{$M^{(1)}=(M_1,N_1,N_2,N)=(5,8,10,30)$}&\multicolumn{5}{c}{$M^{(2)}=(M_1,N_1,N_2,N)=(15,8,10,50)$}\\
\hline
Percentile&TW& $M^{(1)}$& $2M^{(1)}$& $8M^{(1)}$& $16M^{(1)}$& $20M^{(1)}$& $M^{(2)}$& $2M^{(2)}$& $8M^{(2)}$& $16M^{(2)}$& $20M^{(2)}$\\
\hline
   -3.90&0.01&0.0132&0.0083&0.0099&0.0080&0.0108&0.0109&0.0102&0.0085&0.0091&0.0090\\
   -3.18&0.05&0.0546&0.0501&0.0497&0.0502&0.0491&0.0514&0.0495&0.0467&0.0450&0.0476\\
   -2.78&0.10&0.1041&0.1011&0.0995&0.1030&0.0992&0.1028&0.0974&0.0981&0.0956&0.0975\\
   -1.91&0.30&0.2941&0.2948&0.3024&0.3026&0.3028&0.3047&0.3049&0.2944&0.2908&0.3004\\
   -1.27&0.50&0.5031&0.5007&0.5026&0.5114&0.5048&0.5072&0.5077&0.4987&0.4971&0.5009\\
   -0.59&0.70&0.7101&0.7057&0.7116&0.7116&0.7081&0.7074&0.7040&0.7075&0.7037&0.7051\\
    0.45&0.90&0.9138&0.9027&0.9050&0.9062&0.9014&0.9055&0.9019&0.9019&0.9048&0.9038\\
    0.98&0.95&0.9610&0.9507&0.9552&0.9538&0.9519&0.9569&0.9525&0.9502&0.9560&0.9560\\
    2.02&0.99&0.9933&0.9896&0.9898&0.9909&0.9912&0.9916&0.9906&0.9900&0.9910&0.9912\\
\hline
\hline
\multicolumn{2}{c|}{Discrete}&\multicolumn{5}{c|}{$M^{(1)}=(M_1,N_1,N_2,N)=(5,8,10,30)$}&\multicolumn{5}{c}{$M^{(2)}=(M_1,N_1,N_2,N)=(15,8,10,50)$}\\
\hline
Percentile&TW& $M^{(1)}$& $2M^{(1)}$& $8M^{(1)}$& $16M^{(1)}$& $20M^{(1)}$& $M^{(2)}$& $2M^{(2)}$& $8M^{(2)}$& $16M^{(2)}$& $20M^{(2)}$\\
\hline
   -3.90&0.01 &0.0116&0.0093&0.0094&0.0099&0.0082&0.0099&0.0091&0.0098&0.0104&0.0080\\
   -3.18 &0.05 &0.0523&0.0464&0.0503&0.0514&0.0477&0.0496&0.0480&0.0529&0.0495&0.0460\\
   -2.78 &0.10&0.0996&0.0943&0.1034&0.0983&0.0998&0.0951&0.0986&0.1037&0.0978&0.0974\\
   -1.91 &0.30&0.3049&0.2954&0.3054&0.2974&0.3024&0.2933&0.2915&0.3069&0.2968&0.3050\\
   -1.27 &0.50&0.5068&0.5002&0.5114&0.4961&0.4984&0.4989&0.4965&0.5069&0.5015&0.4964\\
   -0.59 &0.70&0.7124&0.7080&0.7065&0.6986&0.7045&0.7065&0.6976&0.7062&0.7042&0.6946\\
    0.45 &0.90&0.9102&0.9098&0.9035&0.9014&0.9021&0.9065&0.9031&0.9058&0.9067&0.8966\\
    0.98 &0.95 &0.9583&0.9565&0.9537&0.9508&0.9512&0.9559&0.9540&0.9515&0.9546&0.9494\\
    2.02 &0.99 &0.9931&0.9917&0.9905&0.9911&0.9894&0.9921&0.9903&0.9915&0.9912&0.9894\\
\hline
\hline
\multicolumn{2}{c|}{Gamma(4,0.5)}&\multicolumn{5}{c|}{$M^{(1)}=(M_1,N_1,N_2,N)=(5,8,10,30)$}&\multicolumn{5}{c}{$M^{(2)}=(M_1,N_1,N_2,N)=(15,8,10,50)$}\\
\hline
Percentile&TW& $M^{(1)}$& $2M^{(1)}$& $8M^{(1)}$& $16M^{(1)}$& $20M^{(1)}$& $M^{(2)}$& $2M^{(2)}$& $8M^{(2)}$& $16M^{(2)}$& $20M^{(2)}$\\
\hline
   -3.90& 0.01& 0.0109& 0.0091& 0.0099& 0.0093& 0.0104& 0.0098& 0.0069& 0.0107& 0.0096& 0.0104\\
   -3.18& 0.05& 0.0507& 0.0502& 0.0500& 0.0501& 0.0507& 0.0494& 0.0452& 0.0503& 0.0486& 0.0495\\
   -2.78& 0.10& 0.1025& 0.1011& 0.0996& 0.1013& 0.0991& 0.1006& 0.0957& 0.1021& 0.0965& 0.1002\\
   -1.91& 0.30& 0.3024& 0.2953& 0.3008& 0.2993& 0.2934& 0.2985& 0.2965& 0.2970& 0.2972& 0.2983\\
   -1.27& 0.50& 0.4992& 0.4994& 0.5013& 0.4967& 0.4865& 0.5009& 0.4890& 0.5028& 0.4995& 0.5010\\
   -0.59& 0.70& 0.7033& 0.7097& 0.6994& 0.6935& 0.6923& 0.7100& 0.7006& 0.7015& 0.7040& 0.7080\\
    0.45& 0.90& 0.9065& 0.9062& 0.9018& 0.9005& 0.9023& 0.9052& 0.9045& 0.8970& 0.9027& 0.9037\\
    0.98& 0.95& 0.9546& 0.9531& 0.9503& 0.9509& 0.9503& 0.9515& 0.9531& 0.9499& 0.9503& 0.9523\\
    2.02& 0.99& 0.9908& 0.9912& 0.9906& 0.9895& 0.9888& 0.9910& 0.9901& 0.9904& 0.9904& 0.9900\\
\hline
\end{tabular}
}
\end{center}
\end{table}

\begin{table}[!htbp]
\begin{center}
{\footnotesize
\caption{\label{t2} Simulated values for $\frac{\#\{\text{rescaled}\ \ \lambda_1^{\bbS}>c_{\alpha}\}}{10000}$ based on 10000 replications. So ``$\tau=0$'' row records estimated sizes and other rows record estimated powers. The significance level is  $\alpha=0.05$.}

\begin{tabular}{c| c c c c c|| c c c c c}
\hline
\multicolumn{11}{c}{$M^{(0)}=(M_1,M_2,N)=(10,10,40)$}\\
\hline
&\multicolumn{5}{c||}{Discrete distribution}&\multicolumn{5}{c}{Gamma distribution}\\
\hline
$\tau$& $M^{(0)}$& $2M^{(0)}$& $4M^{(0)}$& $8M^{(0)}$& $10M^{(0)}$& $M^{(0)}$& $2M^{(0)}$& $4M^{(0)}$& $8M^{(0)}$& $10M^{(0)}$\\
\hline
0     & 0.0663    &0.0618 &   0.0622   & 0.0608&    0.0559	&       0.0672  &  0.0663  &  0.0591   & 0.0589&    0.0563\\
0.1   & 0.2766    &0.5049 &   0.8428   & 0.9978&    0.9998	&	    0.2932  &  0.5117  &  0.8540   & 0.9981&    1.0000\\
0.15   &0.4533    &0.7754 &   0.9915   & 1.0000&    1.0000	&	    0.4641  &  0.7887  &  0.9909   & 1.0000&    1.0000\\
0.2   & 0.6280    &0.9396 &   0.9999   & 1.0000&    1.0000	&	    0.6483  &  0.9463  &  1.0000   & 1.0000&    1.0000\\
0.25 &  0.7828 &   0.9911 &   1.0000   & 1.0000&    1.0000	&    0.7959   & 0.9934   & 1.0000  &  1.0000 &   1.0000\\
0.3  &  0.8959 &   0.9997 &   1.0000   & 1.0000&    1.0000	&    0.9113   & 0.9997   & 1.0000  &  1.0000  &  1.0000\\
0.4  &  0.9908 &   1.0000 &   1.0000   & 1.0000&    1.0000	&    0.9920   & 1.0000    &1.0000  &  1.0000  &  1.0000\\
\hline
\end{tabular}
}
\end{center}
\end{table}

\begin{table}[!htbp]
\begin{center}
{\footnotesize
\caption{\label{t3} Simulated values for $\frac{\#\{\text{rescaled}\ \ \lambda_1^{\bbV}>c_{\alpha}\}}{10000}$ based on 10000 replications. The ``$H_0$'' columns record estimated sizes and other columns record estimated powers.
The significance level is  $\alpha=0.05$.}
\begin{tabular}{c |c c c c|| c |c c c c}
\hline
\multicolumn{5}{c||}{$M^{(1)}=(p,n_0)=(5,8)$}&\multicolumn{5}{c}{$M^{(2)}=(p,n_0)=(8,5)$}\\
\hline
&$H_0$&$H_1^{(1)}$&$H_1^{(1)'}$&$H_1^{(2)}$&&$H_0$&$H_1^{(1)}$&$H_1^{(1)'}$&$H_1^{(2)}$\\
\hline
$M^{(1)}$   &0.0375	&0.0831	&0.5317	&0.5589&	$M^{(2)}$   &0.0374	&0.0511	&0.1502	&0.1098\\
$2M^{(1)}$  &0.0392	&0.2454	&0.9955	&0.8693&	$2M^{(2)}$  &0.0399	&0.1099	&0.7505	&0.2449\\
$4M^{(1)}$  &0.0405	&0.8535	&1.0000	&0.9907&	$4M^{(2)}$  &0.0386	&0.4395	&1.0000	&0.5020\\
$8M^{(1)}$  &0.0414	&1.0000	&1.0000	&1.0000&	$8M^{(2)}$  &0.0375	&0.9956	&1.0000	&0.8341\\
$16M^{(1)}$ &0.0445	&1.0000	&1.0000	&1.0000&	$16M^{(2)}$ &0.0424	&1.0000	&1.0000	&0.9897\\
$32M^{(1)}$ &0.0429	&1.0000	&1.0000	&1.0000&	$32M^{(2)}$ &0.0432	&1.0000	&1.0000	&0.9999\\
$64M^{(1)}$ &0.0396	&1.0000	&1.0000	&1.0000&	$64M^{(2)}$ &0.0390	&1.0000	&1.0000	&1.0000\\
$100M^{(1)}$&0.0442	&1.0000	&1.0000	&1.0000&	$100M^{(2)}$&0.0452	&1.0000	&1.0000	&1.0000\\
\hline
\end{tabular}
}
\end{center}
\end{table}

\begin{table}[!htbp]
\begin{center}
{\footnotesize
\caption{\label{t4} Simulated values for $\frac{\#\{\text{rescaled}\ \ \lambda_1^{\bbM_1}>c_{\alpha}\}}{10000}$ based on 10000 replications. The first two combinations record estimated sizes and the last two record estimated powers.
The significance level is  $\alpha=0.05$.}
\begin{tabular}{c| c c c c c c c c}
\hline
&\multicolumn{8}{c}{$M^{(0)}=(p_1,p_2,N)=(10,6,25)$}\\
\hline
$(\bbC_1,\bbB_2,\bbgam_1)$&$M^{(0)}$&$2M^{(0)}$&$3M^{(0)}$&$4M^{(0)}$&$6M^{(0)}$&$8M^{(0)}$&$10M^{(0)}$&$20M^{(0)}$\\
\hline
$(\bbC_1^{(b)},\bbB_2^{(d)},\bbgam_1^{(a)})$    &0.0400&    0.0447&    0.0453&    0.0469&    0.0487&    0.0460&    0.0466&    0.0468\\
$(\bbC_1^{(a)},\bbB_2^{(d)},\bbgam_1^{(b)})$    &0.0397&    0.0467&    0.0450&    0.0490&    0.0466&    0.0470&    0.0501&    0.0481\\
$(\bbC_1^{(a)},\bbB_2^{(d)},\bbgam_1^{(a)})$    &0.2298&    0.8923&    0.9999&    1.0000&    1.0000&    1.0000&    1.0000&    1.0000\\
$(\bbC_1^{(a)},\bbB_2^{(s)},\bbgam_1^{(a)})$    &0.8337&    0.9451&    0.9821&    0.9940&    0.9992&    1.0000&    0.9999&    1.0000\\
\hline
\end{tabular}
}
\end{center}
\end{table}

\newpage
\section{Appendix}\label{app}
\subsection{Outline of The Proof for Theorem \ref{1016-1}}
We first give an outline of the whole proof due to its complexity.
 Note that the matrix $\bbT$ does not influence the largest eigenvalue of $\bold{\Omega}$ in (\ref{1016.3}) and hence we can directly work on the matrix $(\bbX\bbU_2\bbU^T_2\bbX^T)^{-1}\bbX\bbU_1\bbU_1^T\bbX^T$.
 However it involves four $\bbX$ unlike sample covariance matrices.
Moreover $\bbX\bbU_1\bbU_1^T\bbX^T$ is not independent of $\bbX\bbU_2\bbU^T_2\bbX^T$ for general $\bbX$ (not necessarily consisting of Gaussian entries),
which makes it even harder to work on this matrix directly. In view of this, we construct a Wigner-type linearization matrix
\begin{eqnarray}\label{1125.1}
\bbH=\bbH(\bbX)\coloneqq \left(
  \begin{array}{ccc}
    -zI &  \bbU^T_1 \bbX^T &0 \\
    \bbX \bbU_1 &0 & \bbX \bbU_2 \\
    0& \bbU^T_2\bbX^T & I\\
  \end{array}
\right).
\end{eqnarray}
As will be seen, the linearization matrix is much more convenient when taking derivative with respect to the entries of $\bbX$ than $\bold{\Omega}$. By the Schur complement formula (\ref{0202.1}) below it turns out that the upper-left block of the $3\times 3$ block matrix $\bbH^{-1}$ is the Steiltjes transform
 of $\bbU_1^T\bbX^T(\bbX\bbU_2\bbU^T_2\bbX^T)^{-1}\bbX\bbU_1$ 
 (one can also refer to (\ref{1219.1}) below).
It then suffices to consider the linearization matrix $\bbH$ instead. First the strong local law of $\bbH^{-1}$ around $\mu_N$ (Theorem \ref{0817-1} below) is developed
which is the main body of the proof. 
The overall strategy of proving Theorem \ref{0817-1} is similar to that used in \cite{KY14} and it consists of two main parts.
Part one is to prove Theorem \ref{0817-1} by applying a new Linderberg's comparison approach raised by \cite{KY14} under the first three moments of the entries of $\bbX$ matching those of standard Gaussian entries.
This part is similar to \cite{WY}. However, in order to make this paper more self-consistent and clear,
we also repeat the necessary steps but omit some parts done in \cite{WY}. Building on part one, part two further proves Theorem \ref{0817-1}
when the first two moments of the entries of $\bbX$ match those of standard Gaussian entries (by dropping the 3rd moment matching condition).
After that, we use this local law to prove the edge universality (i.e. (\ref{h0417.1}) is not affected by the distribution of $\bbX$) by adopting the strategy stated in \cite{BPZ2014a} and \cite{LHY2011}.
The proof of Theorem \ref{1016-1} is complete by the fact that (\ref{h0417.1}) holds because $\bold{\Omega}$ becomes a F matrix when $\bbX$ consists of the Gaussian random variable (one can refer to Theorem 1 of \cite{J08} and Theorem 2.1 of \cite{WY}).

We would highlight the difference between the proof of this paper and that of \cite{WY}. The result about the edge university for F matrices
 (corresponding to $\bold{\Omega}$ in the Gaussian case) in \cite{WY} is our starting point because we need to use Linderberg's comparison approach to link
the edge universality of $\bold{\Omega}$ in the general case to that of F matrices.
However, in order to prove the strong local law,
 a main difficulty is that our main result about $\bold{\Omega}$ doesn't assume $\mathbb{E}\bbZ^3_{ij}=0$ (matching the Gaussian third moment), which is much different from the paper \cite{WY}
  when handling the dimension is bigger than the sample size there. As a consequence, the expectation of the higher moments of the variable of interest has to be
  evaluated by a much more complicated method. For example, in order to calculate the higher moments, we need to extract the $i$-th row of $\bbX$ from ${\bold\Pi}(z)$ defined at (\ref{a17}) below.
  However ${\bold\Pi}(z)$ is a complex function of $\bbX$, which is not easy to deal with. To handle this, we introduce a transition matrix $\bold{\Pi_1}(z)$ (defined at (\ref{0528.5}) in the supplementary file) to find out a compact and manageable expansion of $\Pi(z)$.

\subsection{Strong local law}


This subsection is to present the strong local law. To this end we present some necessary notations and definitions.

As in \cite{KY14}, we use the following definition to provides a simple way to describe the relationship between two random variables $\xi$ and $\zeta$.
\begin{deff}\label{1123.1}
Let
$$\xi=\{\xi^{(N)}(u):N\in \mathbb{N}, u\in U^{(N)}\}, \ \ \zeta=\{\zeta^{(N)}(u):N\in \mathbb{N}, u\in U^{(N)}\}$$
be two families of nonnegative random variables, where $U^{(N)}$ is a parameter set (can be either dependent on or independent of $N$). If for all small positive $\ep$ and $\sigma$, there exists a number $N(\ep,\sigma)$ only depending on $\ep$ and $\sigma$ such that
$$\sup_{u\in U^{(N)}}\mathbb{P}\left[|\xi^{(N)}(u)|>N^{\ep}|\zeta^{(N)}(u)|\right]\le N^{-\sigma}$$
for large enough $N\ge N(\ep,\sigma)$, then we say that $\zeta$ stochastically dominates $\xi$ uniformly in u. We denote this relationship by $\xi \prec \zeta$ or $\xi =O_{\prec}(\zeta)$. \textbf{If there exists a positive constant c such that $\xi\le c \zeta$, then we write $\xi\lesssim\zeta$}.
\end{deff}

  Recall the definition of $F$ in Theorem 2.1. If the entries of $\bbX$ are Gaussian distributed, then $\bbX\bbU_2\bbU^T_2\bbX^T$ and $\bbX\bbU_1\bbU^T_1\bbX^T$ are independent
and hence $(\bbX\bbU_2\bbU^T_2\bbX^T)^{-1}\bbX\bbU_1\bbU_1^T\bbX^T$ reduces to the F matrix in \cite{WY}. From \cite{BS06} one can then see that $m(z)$ is a unique solution in $\{z\in\mathcal{C}^+\}$ to the following equation
   \begin{eqnarray}\label{h0529.1}
 \frac{1}{m(z)}=-z+\frac{M_1}{N_1}\int \frac{t}{1+tm(z)}dF(t).
 \end{eqnarray}
Define $\rho(x)=\lim_{z\in \mathcal{C}^+\rightarrow x}\Im m(z)$. One can see that $\mu_N$ defined in (\ref{2.8}) is the rightmost end point of the support of $\rho(x)$.
For the positive constants $\tau$ and $\tau'$, we define the domains
\begin{eqnarray}\label{1121.1}
D(\tau, N)\coloneqq \{z=E+i\eta \in \mathbb{C}^+: |z|\ge \tau, |E|\le \tau^{-1}, N^{-1+\tau}\le \eta\le \tau^{-1} \},
\end{eqnarray}
\begin{eqnarray}\label{1121.2}
D_+=D_+(\tau,\tau', N)\coloneqq \{z\in D(\tau, N): E \ge \mu_N-\tau'\}.
\end{eqnarray}
Let $\bbG(z)=\bbH^{-1}$.  The explicit expression of $\bbG(z)$ can be calculated by the following formula
\begin{eqnarray}\label{0202.1}
\left(
  \begin{array}{cc}
    \bbK & \bbB\\
    \bbC & \bbD\\
  \end{array}
\right)^{-1}
=\left(\begin{array}{cc}
    0 & 0\\
    0 & \bbD^{-1}\\
  \end{array}\right)
  +\left(\begin{array}{cc}
    \bbI\\
    -\bbD^{-1}\bbC  \\
  \end{array}
  \right)
  (\bbK-\bbB\bbD^{-1}\bbC)^{-1}
 \left( \begin{array}{cc}
    \bbI & -\bbB\bbD^{-1}\\
  \end{array}\right).
\end{eqnarray}
To characterize the limit of $\bbG(z)$ introduce
 $\Gamma(X,z)= (\bbX\bbU_2\bbU^T_2\bbX^T+m(z)\bbI)^{-1}$ and
\begin{eqnarray}\label{a17}\Pi(z)=\left(
  \begin{array}{ccc}
    m(z)\bbI & 0&0 \\
    0& \Gamma(X,z)&0 \\
    0&0&\bbI+\bbU^T_2\bbX^T\Gamma(X,z)\bbX\bbU_2.\\
  \end{array}
\right) \end{eqnarray} 
As will be seen $\bbG(z)$ is close to $\Pi(Z)$ in $D_+$.  Set
$$\Psi=\Psi(z)=\sqrt{\frac{\Im m(z)}{N\eta}}+\frac{1}{N\eta}.$$

 \begin{thm} (Strong local law)\label{0817-1}
Suppose that $\bbX$ satisfies  Condition \ref{0603-1}. Then
\begin{itemize}
\item[(i)]  For any deterministic unit vectors $\bbv$, $\bbw\in \mathbb{R}^{N_1+N+M_1-N_2}$, 
\begin{eqnarray}\label{0310.1}\langle\bbv,(\bbG(z)-\Pi(z))\bbw\rangle\prec \Psi\end{eqnarray}
uniformly $z\in D_+$ and  
\item[(ii)]
\begin{eqnarray}\label{0310.2}|m_N(z)-m(z)|\prec \frac{1}{N\eta}\end{eqnarray}
uniformly in $z\in D_+$, where $m_N(z)=\frac{1}{N_1}\sum_{i=1}^{N_1}G_{ii}$. 
\end{itemize}
\begin{proof}
The proof of this theorem is delegated to the supplement.
\end{proof}
\end{thm}

\subsection{Fluctuation at the right edge and universality}
\subsubsection{Fluctuation at the right edge}
Once Theorem \ref{0817-1} is ready it is not hard to show the following Lemma.
\begin{lem}\label{0525-1}
Under conditions of Theorem \ref{0817-1},
$$\lambda_1-\mu_N=O_{\prec}(N^{-\frac{2}{3}}).$$
\end{lem}
\begin{proof}
The proof of this theorem is given in the supplement.
\end{proof}
\subsubsection{Universality}
We now need edge universality at the rightmost edge of the support. i.e. the limiting distribution of $P(\sigma_{N} N_1^{2/3} (\lambda_{1}-\mu_{N}) \leq s)$
is not affected by the distribution of $\bbX$. This guarantees Theorem \ref{1016-1}. Similar to Theorem 6.3 of \cite{LHY2011},
 in order to show Theorem \ref{1016-1}, it suffices to show the following green function comparison theorem
(one can also refer to page 48 of \cite{WY} to understand the connection between Theorem \ref{0826-3} and (\ref{h0417.1})). The corresponding proof is also provided in the supplement.
\begin{thm}\label{0826-3}
Let $\ep >0$, $\eta=N^{-2/3+\ep}$, $E_1$, $E_2\in \mathbb{R}$ satisfy $E_1<E_2$ and
$$|E_1-\mu_N|,|E_2-\mu_N|\le N^{-2/3+\ep}.$$
Set $K:\mathbb{R}\rightarrow \mathbb{R}$ to be a smooth function such that
$$\max_{x}|K^{(l)}(x)|\le C, \ \ l=1,2,3,4, 5$$
for some constant $C$.
Then there exists a constant $\phi>0$ such that for large enough N and small enough $\ep$, we have
\begin{eqnarray}\label{1119.9}
|\mathbb{E}K(N\int_{E_1}^{E_2}\Im m_{X^1}(x+i\eta)dx)-\mathbb{E}K(N\int_{E_1}^{E_2}\Im m_{\bbX^0}(x+i\eta)dx)|\le N^{-\phi},
\end{eqnarray}
where the definitions of $\bbX^0$ and $\bbX^1$ are given in Section \ref{h0418.1} in the supplement.
\end{thm}
\section{Local law (\ref{0310.1})}

Throughout the proof we use $c$, $C$, $K_1$ and $M_0$ to denote
some positive constants whose values may differ from line to line. 
 We may assume that $\mathbb{E} X_{ij}^2=\mathbb{E} X_{it}^2=\frac{1}{N_1}$ in the sequel. Since $N_1$ and $N$ are of the same order,
when we calculate the upper (lower) bound of some terms, $\mathbb{E} X_{ij}^2$ is usually regarded as $1/N$ for convenience.
 Before starting the proof, we present some definitions and notations first.

\begin{deff}(Matrix Norms) Let $\bbA=(A_{ij})$ be a matrix. We define the following norms
 $$\|\bbA\|\coloneqq \max_{\|\bbx\|=1}|\bbA\bbx|,\ \ \|\bbA\|_{\infty}\coloneqq\max_{i,j}|A_{ij}|, \ \ \|\bbA\|_F\coloneqq\sqrt{tr \bbA\bbA^*},$$
 where $|\bbx|$ is the $L_2$ norm of a vector $\bbx$.
 Notice that we have the simple inequality
 $$\|\bbA\|_{\infty}\le \|\bbA\|\le \|\bbA\|_F.$$
 \end{deff}

\begin{deff}\label{0420.1}
We say that an event
$\Lambda$ holds with high probability if for any large
positive constant D, there exists $n_0(D)$ such that
$$\mathbb{P}(\Lambda^c)\le n^{-D}, \ \text{for any} \   n\ge n_0(D).$$
\end{deff}

 In the later proof, we need the following Lemma to control the smallest eigenvalue of $\bbX\bbU_2\bbU^T_2\bbX^T$.
 \begin{lem}\label{1121-1}
Suppose that $\bbX$ satisfies Condition \ref{0603-1} (see the main paper). Then $\bbX\bbU_2\bbU^T_2\bbX^T$ is invertible and
\begin{equation}\label{a39}\|(\bbX\bbU_2\bbU^T_2\bbX^T)^{-1}\|\le M_0
\end{equation}for some large constant $M_0$ with high probability. Moreover,
\begin{equation}\label{b5}\|\bbX\bbX^T\|\leq M_0
\end{equation} with high probability under Condition \ref{0603-1} as well.
\end{lem}
\begin{proof}
The proof of this lemma is exactly the same as that of Theorem 3.12 in \cite{KY14}.
\end{proof}
Moreover, we define the following smooth cutoff function
$$\mathcal{X}(x)=\begin{cases}
   1 &\mbox{if $|x|\le K_1 N^{-2}$}\\
   0 &\mbox{if $|x|\ge 2K_1 N^{-2}$},
\end{cases}$$
whose derivatives satisfy $|\mathcal{X}^{(k)}|\le C N^{2k}$, k=1,2,...  and $K_1$ is a positive constant. The purpose of introducing the cutoff function is to
help control the minimum eigenvalue and maximum eigenvalue of the random matrices of interest when taking derivatives.

Order the eigenvalues of $\bbX\bbU_2\bbU^T_2\bbX^T$ in a decreasing order as $\tilde \lambda_1\geq...,\geq\tilde \lambda_{M_1}$ and denote its Stieltjes transform
by $\tilde m_N(z)$. Since
\begin{eqnarray}\label{1209.1}
\Im (\tilde m_N(i N^{-2}))=M_1^{-1}N^{-2}\sum_{i=1}^{M_1}\frac{1}{\tilde \lambda_i^2+N^{-4}},
\end{eqnarray}
we conclude that if $|\Im (\tilde m_N(i N^{-2}))|\lesssim N^{-2}$, then $\tilde \lambda_{N-M_2} \gtrsim \frac{1}{N}$. By Lemma \ref{1121-1}, choosing a sufficiently small constant c, we have
\begin{eqnarray}\label{1209.1}
1-o(N^{-l})=\mathbb{P}(\tilde \lambda_{M_1}\ge c)\le \mathbb{P}(\Im (\tilde m_N(i N^{-2}))\le K_1N^{-2}), \ \ \text{for any positive integer } l.
\end{eqnarray}
Therefore, we have
\begin{eqnarray}\label{0603.2}
\mathbb{P}(\mathcal{X}(\Im (\tilde m_N(i N^{-2})))\neq 1)\le o(N^{-l}), \ \ \text{for any positive integer } l.\end{eqnarray}
Similarly, for $\bbX\bbX^T$, we have
\begin{eqnarray}\label{0603.3}\mathbb{P}(\mathcal{X}(N^{-3} \|\bbX\|^2_F)\neq 1)\le o(N^{-l}), \ \ \text{for any positive integer } l.\end{eqnarray}
We set $\mathcal{T}_N(X)\coloneqq \mathcal{X}(\Im (\tilde m_n(i N^{-2}))\mathcal{X}(N^{-3} \|\bbX\|^2_F) $. In view of (\ref{0603.2}) and (\ref{0603.3}), we can show
\begin{eqnarray}\label{0821.4}\mathcal{T}_N(\bbX)=1\end{eqnarray}
with high probability directly. We will use this conclusion frequently without mention.

Denote the spectral decomposition of $\bbU_1^T\bbX^T(\bbX\bbU_2\bbU^T_2\bbX^T)^{-1}\bbX\bbU_1$ by
$$\bbU_1^T\bbX^T(\bbX\bbU_2\bbU^T_2\bbX^T)^{-1}\bbX\bbU_1=\sum_{k=1}^{N_1}\lambda_k \bbv_k\bbv^*_k,$$
where $\{\bbv_k\}_{k=1}^{N_1}$  are orthogonal bases of $\mathbb{R}^{I_{N_1}}$ . For $1\le i,j\le N_1$ write
\begin{eqnarray}\label{1119.14}\bbG_{ij}=\sum_{k=1}^{N_1}\frac{\bbv_k(i)\bbv^*_k(j)}{\lambda_k-z},\end{eqnarray}
where $\bbG_{ij}$ is the entry at the $i$-th row and $j$-th column of $\bbG$ and $\bbv_k(i)$ is the $i$-th element of $\bbv_k$.
Define a new matrix $\bbG_{N_1}$ to be $(\bbG_{ij})_{1\le i,j \le N_1}$.

Moreover, we define
$$\bbA_2\coloneqq  \left(
  \begin{array}{ccc}
     \bbI & \bbA_4\\
  \end{array}
\right)^T=\left(
  \begin{array}{ccc}
    \bbI & -\bbU^T_1\bbX^T(\bbX\bbU_2\bbU^T_2\bbX^T)^{-1} & \bbU^T_1\bbX^T(\bbX\bbU_2\bbU^T_2\bbX^T)^{-1}XU_2\\
  \end{array}
\right)^T
$$ and
 $$\bbA_3\coloneqq\left(
\begin{array}{ccc}
    0 & 0\\
    0 & \bbA_5\\
  \end{array}
\right)= \left(
  \begin{array}{ccc}
    0 & 0 &0\\
    0 & -(\bbX\bbU_2\bbU^T_2\bbX^T)^{-1} & (\bbX\bbU_2\bbU^T_2\bbX^T)^{-1}\bbX\bbU_2\\
    0 & \bbU_2^T\bbX^T(\bbX\bbU_2\bbU^T_2\bbX^T)^{-1}& \bbI-P_{\bbU^T_2\bbX^T}
  \end{array}
\right),$$
Via (\ref{0202.1}) we then have an explicit expression of $\bbG$
\begin{eqnarray}\label{1219.1}
\bbG=\bbA_3+\sum_{k=1}^{N_1}\frac{\bbA_2\bbv_k\bbv^*_k\bbA^T_2}{\lambda_k-z}=\bbA_3+\bbA_2\bbG_{N_1}\bbA_2^T.
\end{eqnarray}

\begin{deff}[moment matching]
Let $\bbX^1=(x^1_{ij})_{M\times N}$ and $\bbX^0=(x^0_{ij})_{M\times N}$ be two complex(or real) matrices satisfying Condition \ref{0603-1}. We say that $X^1$ matches $X^0$ to order m, if for all $i\in [1,M]$, $j\in [1,N]$, $k,l\ge 0$ and $k+l\in [0,m]$, it has the relationship
\begin{eqnarray}\label{1129.1}
\mathbb{E}(\Re(\sqrt N x_{ij}^1)^k\Im (\sqrt N x_{ij}^1)^l)=\mathbb{E}(\Re(\sqrt N x_{ij}^0)^k\Im (\sqrt N x_{ij}^0)^l)+O(e^{-(\log N)^C}),
\end{eqnarray}
where C is a constant larger than 1.
\end{deff}

We next collect some frequently used bounds. Recall the definition of $m(z)$ in (\ref{h0529.1}).  For $z\in D(\tau,n)$ one may verify that
\begin{equation}\label{a18}
1 \lesssim |m(z)| \lesssim 1
\end{equation}
and
\begin{equation}\label{a19}
\eta \lesssim \Im(m(z)).
\end{equation}
(see Lemma 2.3 in \cite{BPWZ2014b} or Lemma 3.1 and Lemma 3.2 in \cite{s1}).
It is obvious that $m(z)$ decides a unique spectral density $\rho(x)$. Recalling the definition of $\mu_N$
 we write $\bbc=-\lim_{z\in \mathcal{C}^+\rightarrow \mu_N}m(z)$. By checking the proof of Lemmas 1 and 2 in \cite{WY} carefully and noting that the proof only relies on the rigidity property of $\bbX\bbU_2\bbU^T_2\bbX^T$,  there exists a constant $c'$ such that
$$\limsup_{N}\left[\bbc\lambda_{\max}((\bbX\bbU_2\bbU^T_2\bbX^T)^{-1})\right]\le 1-c', $$
with high probability.
By Lemma A.4 of \cite{KY14}, there exists a constant $c''$ such that
\begin{eqnarray}\label{0821.8}|1+m(z)\lambda_i((\bbX\bbU_2\bbU^T_2\bbX^T)^{-1})|\ge c'', \end{eqnarray} for all $z\in D_+$ with high probability.
Moreover, for $z\in D_{+}$ it follows from Lemma \ref{1121-1} that
\begin{equation}\label{0304.5}
\|\Pi(z) \|\prec 1,\ \text{and} \ \|\bbA_2\|+\|\bbA_3\|\prec 1.
\end{equation}

To simplify notation, we introduce the following notations with bold lower indices and if the lower index of a matrix is bold,
 then it represents the inner product and otherwise it means the entry of the corresponding matrix. Specifically
\begin{equation}\label{0822.2}\bbA_{\bbv s}= \langle\bbv,\bbA\bbe_s\rangle,\  \bbA_{ s\bbv}= \langle\bbe_s,\bbA\bbv\rangle \ \text{and} \  \bbA_{\bbv \bbw}= \langle\bbv,\bbA\bbw\rangle,\end{equation}
where $\bbe_s$ is the unit vector with the s-th coordinate equal to 1.
For any $z\in D(\tau,n)$ and fixed $\tau>0$, we claim that
\begin{eqnarray}\label{0821.2}
\|\bbG(z)\mathcal{T}_n(\bbX)\|\lesssim N^{9}\eta^{-1}, \  \|\partial_z\bbG(z)\mathcal{T}_n(\bbX)\|\lesssim N^{9}\eta^{-2},
\end{eqnarray}
\begin{eqnarray}\label{1119.12}
\|\bbG(z)\|\prec\eta^{-1}, \ \ \|\partial_z\bbG(z)\|\prec \eta^{-2}.
\end{eqnarray}
Moreover, suppose that $\bbv=(\bbv_1^T,\bbv_2^T)^T$. Then
\begin{eqnarray}\label{1124.13}
\sum_{i=1}^{M_1}|\bbG_{\mathbf{v}i}|^2=\frac{\Im \bbG_{\mathbf{v}\mathbf{v}}}{\eta}, \ \|\Pi(z)\mathcal{T}_n(\bbX)\|\lesssim N^{4}\eta^{-1},
\end{eqnarray}
and
\begin{eqnarray}\label{1230.1}
|\bbG_{\mathbf{v}\mathbf{v}}|^2\prec\frac{\Im \bbG_{\mathbf{v}\mathbf{v}}}{\eta}+1,
\end{eqnarray}
Indeed, the estimates (\ref{1119.12}) and (\ref{1230.1}) 
 follow from 
 Lemma \ref{1121-1}  and (\ref{b5}). The first equality in (\ref{1124.13}) is straightforward and the second one is from the definition of $\mathcal{T}_n(\bbX)$ directly.

 When the entries of $\bbX$ are Gaussian distributed Theorem \ref{0817-1} can be obtained by Theorem 3.6 of \cite{KY14}.
 Actually, a key observation is that each block matrix of $\bbG(z)$ ($3\times 3$ block matrix) can be represented as a linear combination of the block matrices of (4.3) in \cite{KY14} by (\ref{1219.1}) under the Gaussian case.
We demonstrate this observation by checking three block matrices of $\bbG(z)$ and the other blocks can also be inspected similarly. For example, by (\ref{1219.1}), the upper left block of $\bbG(z)$ is
 $$\bbG_{N_1}=(\bbU_1^T\bbX^T(\bbX\bbU_2\bbU^T_2\bbX^T)^{-1}\bbX\bbU_1-z\bbI)^{-1}.$$
Since $(\bbX\bbU_2\bbU^T_2\bbX^T)^{-1}$ is independent of $\bbX\bbU_1$ under the gaussian case $(\bbX\bbU_2\bbU^T_2\bbX^T)^{-1}$ can be regarded as a
 population covariance matrix ``$\Sigma$''. Hence $\bbG_{N_1}$ is just one block matric of (4.3) in \cite{KY14}. A second block matrix of $\bbG(z)$ is $-(\bbX\bbU_2\bbU^T_2\bbX^T)^{-1}\bbX\bbU_1\bbG_{N_1}$.
It is also a block of (4.3) in \cite{KY14} by the same reason that $(\bbX\bbU_2\bbU^T_2\bbX^T)^{-1}$ is regarded as ``$\Sigma$''. A third block is the second diagonal block matrix of $\bbG(z)$:
$$-(\bbX\bbU_2\bbU^T_2\bbX^T)^{-1}+(\bbX\bbU_2\bbU^T_2\bbX^T)^{-1}\bbX\bbU_1\bbG_{N_1}\bbU^T_1\bbX^T(\bbX\bbU_2\bbU^T_2\bbX^T)^{-1},$$
which is also a block of (4.3) in \cite{KY14}. In fact, the matrix above corresponds to $(\Sigma\bbX\bbG_N\bbX^*\Sigma-\Sigma)$ belonging to (4.3) of \cite{KY14}.

 \subsection{Proving (\ref{0310.1}) for general distributions under the first three moments matching condition}
We now prove (\ref{0310.1}) for general distributions under the condition that $\bbX$ matches $\bbX^{Gauss}$  to order 3 in this section,
 where the entries of $\bbX^{Gauss}$ follow standard Gaussian distribution. However, the proof of this section is very similar to that of Section 7.1 of \cite{WY} (following the strategy in \cite{KY14}).
Hence, we below only give an outline of the arguments in order to prepare notations and tools for the proof under the first two moment matching condition in the next section. One may refer to Section 7.1.1 in \cite{WY} for more details. 

It suffices to show that for any orthogonal matrix $\bbB_1$ and $\bbB_2$,
\begin{eqnarray}\label{1116}
\|\bbB_1(\bbG(z)-\Pi(z))\bbB_2^*\|_{\infty}\prec \Psi,
\end{eqnarray}
for all $z\in S$, where $S$ is an $\ep$-net of $D_+$ with $\ep=N^{-10}$.
  Setting $\delta$ to be a sufficient small positive constant such that $N^{24\delta}\Psi\ll1$, for any given $\eta\ge \frac{1}{N}$, we define a serial numbers $\eta_0\le \eta_1\le \eta_2...\le \eta_L$ based on $\eta$, where
$$L\equiv L(\eta)\coloneqq\max\{l\le \mathbb{N}:\eta N^{l\delta}< N^{-\delta}\}.$$
So
$$\eta_l\coloneqq \eta N^{l\delta}, \ \ (l=0,1,...,L-1), \ \ \eta_L\coloneqq 1.$$
 We work on  the net $S$ satisfying the condition that $E+i\eta_l\in  S, \ \ l=0,...,L,$ from now on.
We define $S_m\coloneqq \{z\in S:\Im z\ge N^{-\delta m}\}$ corresponding to the following events:

\begin{eqnarray}\label{1116.4}
A_m= \{\|\bbB_1(G(z)-\Pi(z))B^*_2\mathcal{T}_N(\bbX)\|_{\infty}\prec 1, \   \text{for any} \  z\in  S_m\},
\end{eqnarray}
and
\begin{eqnarray}\label{1116.5}
C_m= \{\|\bbB_1(G(z)-\Pi(z))B^*_2\mathcal{T}_N(\bbX)\|_{\infty}\prec \Psi, \   \text{for any} \  z\in S_m\}.
\end{eqnarray}
We start the induction by considering the event $A_0$ first. In fact, it is not hard to prove that
event $A_0$ holds. 
By the assumption that  $N^{24\delta}\Psi\ll 1$, it is easy to see that the event $C_m$ implies the event $A_{m}$. We will prove the event $(A_{m-1})$ implies the event $(C_m)$ for all $1\le m\le \delta^{-1}$ in the sequel, which ensure that (\ref{1116}) holds on the set S uniformly.

For the purpose, we should calculate the upper bound of the higher moments of the following functions
\begin{eqnarray}\label{a5}
F_{st}(X,z)= (\bbB_1G(z)\bbB_2^*)_{st}-(\bbB_1\Pi(z)\bbB_2^*)_{st}\mathcal{T}_N(\bbX),\end{eqnarray}
$\bbB_1, \bbB_2 \in \mathcal{L}= \{1,\mathbf{\Delta}, \bbV\}$, $\mathbf{\Delta}$ is defined in (\ref{1125.10}) below and $\bbV$ is any deterministic orthogonal matrix.
By Markov's inequality and (\ref{0821.4}), in order to prove (\ref{1116}), it suffices to prove Lemma \ref{0821-1} below.
\begin{lem}\label{0821-1}
Let p be an positive even constant and $m\le \delta^{-1}$. Suppose (\ref{1116.4}) for all $z\in S_{m-1}$. Then we have
$$\mathbb{E}|F^p_{st}(X,z)|\prec (N^{24\delta}\Psi)^p,$$
for all $1\le s,t\le N+N_1+M_1-N_2$ and $z\in S_m$.
\end{lem}
The proof of Lemma \ref{0821-1} is almost the same as that of \cite{WY} under the order 3 moment matching condition. 
\begin{lem}[Lemma 5 of \cite{WY}]\label{1123-3}
 Let $\zeta$ be a random variable satisfying $\zeta \prec \nu$ where positive $\nu$ may be random or deterministic. Suppose $|\zeta|\leq N^{C}$ for some positive constant $C$.
  Then \begin{equation}\label{b6} \mathbb{E}\zeta \prec (E\nu+N^{C-D}),\end{equation}
  where $D$ is a sufficiently large positive constant.
\end{lem}
\subsubsection{The proof of Lemma \ref{0821-1} by the interpolation method}\label{h0418.1}
We define the interpolation matrix $\bbX^t$ between $\bbX^1=(X_{i\mu}^1)= \bbX$ and $\bbX^0=\bbX^{Gauss}$ consisting of standard Gaussian random variables below, where $1\le i \le M_1$ and $1\le \mu \le N$. 
\begin{deff}\label{1119-4}
For $u\in \{0,1\}$, $1\le i\le M_1$ and $1\le \mu \le N$, denote the distribution function of $X_{i\mu}^{u}$ by $F_{i\mu}^u$. For $\theta\in [0,1]$, we define the distribution function by
$$F_{i\mu}^{\theta}= \theta F^1_{i\mu}+(1-\theta)F_{i\mu}^0.$$
The interpolation matrix $\bbX^{\theta}$ is $(X_{i\mu}^{\theta})$ with $F_{i\mu}^{\theta}$ being the distribution of $X_{i\mu}^{\theta}$ and the entries $\{X_{i\mu}^{\theta}\}$ are
mutually independent for all $i,\mu$.
Moreover, we introduce the matrix
\begin{equation}\label{a7}\bbX_{(i\mu)}^{\theta, \lambda}= \bbX^{\theta}+(\lambda-X_{i\mu}^{\theta}) \bbe_i\bbe_{\mu}^T,\end{equation}
which differs from $\bbX^t$ at the $(i,\mu)$ position only
and the corresponding green functions
\begin{eqnarray}\label{0821.7}\bbG^{\theta}(z)= \bbG(\bbX^{\theta},z), \ \ \bbG^{\theta,\lambda}_{(i\mu)}(z)= \bbG(\bbX^{\theta,\lambda}_{(i\mu)},z),\end{eqnarray}
by replacing $\bbX$ in $\bbG(z)$ by $\bbX^{\theta}$ and $\bbX^{\theta,\lambda}_{(i\mu)}$ respectively.
\end{deff}
To calculate the difference of $\mathbb{E}|F_{st}(X^1,z)|^p$ and $\mathbb{E}|F_{st}(X^0,z)|^p$, we introduce the following Lemma.
\begin{lem}[Lemma 7.9 of \cite{KY14}]\label{0822-1}
For any function $F:\mathbb{R}^{M\times N}\rightarrow \mathbb{C}$, we have
\begin{eqnarray}\label{1124.14}
\mathbb{E}F(\bbX^1)-\mathbb{E}F(\bbX^0)=\int_0^1d\theta\sum_{i=1}^{M_1}\sum_{\mu=1}^N\left[\mathbb{E}F(\bbX^{\theta, X_{(i\mu)}^1}_{(i\mu)})-\mathbb{E}F(\bbX^{\theta, X_{(i\mu)}^0}_{(i\mu)})\right].
\end{eqnarray}
\end{lem}
To deal with the right hand side of (\ref{1124.14}), we need to prove the following Lemma.
\begin{lem}\label{1119-5}
Fix an even positive integer p and $m\le \delta^{-1}$. Suppose that ($A_{m-1}$) holds. Then there exists some function $B_{st}(.,z)$  such that for $u\in \{0,1\}$
\begin{eqnarray}\label{1119.13}
\sum_{i=1}^{M_1}\sum_{\mu=1}^N\left[\mathbb{E}|F^p_{st}(\bbX^{\theta, X_{(i\mu)}^u}_{(i\mu)},z)|-\mathbb{E}|B^p_{st}(\bbX^{\theta, 0}_{(i\mu)},z)|\right]=O((N^{24\delta}\Psi)^p+\|\mathbb{E}\bbL_p(\bbX^{\theta},z)\|_{\infty}),\non
\end{eqnarray}
where $\bbL_p(\bbX^{\theta},z)=(|F^p_{st}(X^{\theta, \bbX_{(i\mu)}^1}_{(i\mu)},z)|)$.
\end{lem}
Lemma \ref{1119-5} concludes that
\begin{eqnarray}\label{0821.3}
\sum_{i=1}^{M_1}\sum_{\mu=1}^N\left[\mathbb{E}|F^p_{st}(\bbX^{\theta, X_{(i\mu)}^1}_{(i\mu)},z)|-\mathbb{E}|F^p_{st}(\bbX^{\theta, X_{(i\mu)}^0}_{(i\mu)},z)|\right]=O((N^{24\delta}\Psi)^p+\|\mathbb{E}\bbL_p(\bbX^{\theta},z)\|_{\infty}),\non
\end{eqnarray}
for all $z \in S_m$. Therefore by Gronnwall's inequality, we can prove Lemma \ref{0821-1}.
What remains to do is to prove Lemma \ref{1119-5}. In the sequel, we only consider the case u=1(u=0 is similar to u=1). First, we calculate the rough bound below, which is a direct conclusion given the event $A_{m-1}$.
\begin{lem}\label{1116-3}
Suppose that (\ref{1116.4}) holds for all $z\in  S_{m-1}$. Then
$$\langle\bbv,(\bbG(z)-\Pi(z))\bbw\rangle=O_{\prec}(N^{2\delta})$$
for all $z\in  S_m$.
\end{lem}
From (\ref{a7}), we write
$$\bbX_{(i\mu)}^{\theta, \lambda_1}-\bbX_{(i\mu)}^{\theta, \lambda_2}=(\lambda_1-\lambda_2)\bbe_i\bbe_{\mu}^T.$$
Together with (\ref{1125.1}), one can obtain that
\begin{eqnarray}\label{a12}
\bbH(\bbX_{(i\mu)}^{\theta, \lambda_1})-\bbH(\bbX_{(i\mu)}^{\theta, \lambda_2})=\bold\Delta_{(i\mu)}^{\lambda_1-\lambda_2},
\end{eqnarray}
where $\bbH(\bbX_{(i\mu)}^{\theta, \lambda_i})$ is obtained from $\bbH(\bbX)$ in (\ref{1125.1}) with $\bbX$ replaced by $\bbX_{(i\mu)}^{\theta, \lambda_i}$ respectively, i=1,2 and
\begin{eqnarray}\label{1125.10}
\mathbf{\Delta}_{(i\mu)}^{\lambda}= \lambda\Big(\mathbf{\Delta}\bbe_{\mu}\bbe_{i+N_1}^T+\bbe_{i+N_1}\bbe^T_{\mu}\mathbf{\Delta}^T\Big), \ \mathbf{\Delta}= \left(
  \begin{array}{c}
    \bbU^T_1  \\
    0 \\
    \bbU^T_2\\
  \end{array}
\right),
\end{eqnarray}
where and in the following $\bbe_{i+N_1}$ is always $(M_1+N+N_1-N_2)\times 1$ and $\bbe_{\mu}$ is $N\times 1$ vector. \textbf{From now on, we denote $i+N_1$ by $\tilde i$ for simplicity.}
%
Applying the formula $\bbA^{-1}-\bbB^{-1}=-\bbA^{-1}(\bbA-\bbB)\bbB^{-1}$ repeatedly we further obtain the following resolvent formula for any $K\in \mathbb{N}_+$
\begin{eqnarray}\label{1123.2}
\bbG_{(i\mu)}^{\theta, \lambda_1}=\bbG_{(i\mu)}^{\theta, \lambda_2}+\sum_{k=1}^K(-1)^k\bbG_{(i\mu)}^{\theta, \lambda_2}(\mathbf{\Delta}_{(i\mu)}^{\lambda_1-\lambda_2}\bbG_{(i\mu)}^{\theta, \lambda_2})^k+(-1)^{H+1}\bbG_{(i\mu)}^{\theta, \lambda_1}(\mathbf{\Delta}_{(i\mu)}^{\lambda_1-\lambda_2}\bbG_{(i\mu)}^{\theta, \lambda_2})^{K+1},
\end{eqnarray}
recalling the definition (\ref{0821.7}). Here and the remaining part of this section we drop the variable z
when there is no confusion but one should remember that $z\in S_m$.
\begin{lem}\label{1116-1}
Suppose that $\lambda$ is a random variable and satisfies $|\lambda|\prec N^{-1/2}$. Then
\begin{eqnarray}\label{1116.7}
\|\bbB_1(G_{(i\mu)}^{\theta, \lambda}-\Pi)\bbB_2\|_{\infty}\prec N^{2\delta}.
\end{eqnarray}
\end{lem}

In order to simplify the notations, we define
$$f_{(i\mu)}(\lambda)= |F^{p}_{st}(X_{(i\mu)}^{\theta, \lambda})|=(F_{st}(X_{(i\mu)}^{\theta, \lambda})\overline{F_{st}(X_{(i\mu)}^{\theta, \lambda})})^{\frac{p}{2}},$$
where we omit some parameters.
\begin{lem}\label{1123-2}
Suppose that $\lambda$ is a random variable and it satisfies $|\lambda|\prec N^{-1/2}$. Then for any fixed integer n, we have
\begin{eqnarray}\label{1123.3}
|f^{(k)}_{(i\mu)}(\lambda)|\prec N^{2\delta(p+k)}.
\end{eqnarray}
Moreover, we have
\begin{eqnarray}\label{1123.4}
f_{(i\mu)}(\lambda)=\sum_{k=1}^{4p}\frac{\lambda^k}{k!}f^{(k)}_{(i\mu)}(0)+O_{\prec}(\Psi^p)
\end{eqnarray}
by Taylor's expansion.
\end{lem}
The proof of Lemmas \ref{1116-1} and \ref{1123-2} can be found in \cite{WY}.
By Lemma \ref{1123-2} and Lemma \ref{1123-3}, we have
\begin{eqnarray}\label{1123.5}
&&\mathbb{E}|F^p_{st}(X_{(i\mu)}^{\theta,X_{i\mu}^1})|-\mathbb{E}|F^p_{st}(X_{(i\mu)}^{\theta,0})|=\mathbb{E}f_{(i\mu)}(X_{i\mu}^1)-\mathbb{E}f_{(i\mu)}(0)\non
&&=\frac{1}{2N_1}\mathbb{E}f^{(2)}_{(i\mu)}(0)+\sum_{k=4}^{4p}\frac{1}{k!}\mathbb{E}f^{(k)}_{(i\mu)}(0)\mathbb{E}(X_{i\mu}^1)^k+O_{\prec}(\Psi^p),
\end{eqnarray}
where we use the first three moment matching condition.
To show (\ref{1119.13}), it suffices to prove that
\begin{eqnarray}\label{1123.6}
N^{-k/2}\sum_{i=1}^{M_1}\sum_{\mu=1}^N\mathbb{E}f^{(k)}_{(i\mu)}(0)=O_{\prec}((N^{24\delta}\Psi)^p+\|\mathbb{E}\bbL_p(X^{\theta})\|_{\infty}),
\end{eqnarray}
for k=4,...,4p.
 We now  point out that $\mathbb{E}|B_{st}(\bbX_{(i\mu)}^{\theta,0})|^{p}$ in (\ref{1119.13}) equals
$$\mathbb{E}|F_{st}(\bbX_{(i\mu)}^{\theta,0})|^{p}+\frac{1}{2N_1}\mathbb{E}f^{(2)}_{(i\mu)}(0).$$
But we do not prove (\ref{1123.6}) directly. We instead prove (\ref{1123.7}) in order to obtain a self-consistent estimation of $X^{\theta}$ instead.
We claim that if
\begin{eqnarray}\label{1123.7}
N^{-k/2}\sum_{i=1}^{M_1}\sum_{\mu=1}^N\mathbb{E}f^{(k)}_{(i\mu)}(X^{\theta}_{i\mu})=O_{\prec}((N^{24\delta}\Psi)^p+\|\mathbb{E}\bbL_p(X^{\theta})\|_{\infty}),
\end{eqnarray}
holds for k=4,...,8p, then (\ref{1123.6}) holds for n=4,...,4p.
The proof of this claim is the same as (7.60)-(7.61) of \cite{WY}.

It then suffices to prove (\ref{1123.7}). Recall that \begin{equation}\label{a8}
f^{(k)}_{(i\mu)}(X^{\theta}_{i\mu})=\frac{\partial^k\Big(|F_{st}(\bbX_{(i\mu)}^{\theta})|^{p}\Big)}{\partial (X_{i\mu}^{\theta})^k},\end{equation}
where $F_{st}(\cdot)$ is given in (\ref{a5}). Since $\bbX^{\theta}=\bbX_{(i\mu)}^{{\theta}, X^{\theta}_{i\mu}}$ is the only matrix of interest. We below use $\bbX=(X_{i\mu})$ instead of $\bbX^{\theta}=(X_{i\mu}^{\theta})$ to simplify notation because the entries of both of them have bounded higher moments. To prove (\ref{1123.7}) we need to study (\ref{a8}).
\subsubsection{Estimation of higher order derivatives (\ref{a8}) in (\ref{1123.7})}

Before starting Section \ref{h0425-1}, we quote some notations and necessary results from \cite{WY} about estimation of higher order derivatives (\ref{a8}) in (\ref{1123.7}).
 By dropping $\bbe_i\bbe^T_\mu$ and $\bbe_\mu\bbe^T_i$ we define the set
\begin{equation}\label{1202.1}
\mathcal{Q}(k)= \{\text{The matrices constructed from sum or product of (part of)} \ \   \bbU_2, \bbX,  \  \Gamma(\bbX,z)\},
\end{equation}
where any $k$th order derivative of each block of $\Pi(z)$ with respect to $\bbX_{i\mu}$ belongs to some product(s) between some matrices in $\mathcal{Q}(k)$ and $\bbe_i\bbe^T_\mu$ or $\bbe_\mu\bbe^T_i$.
To characterize the higher order derivative conveniently we define group $g$ of size $k$ to be the set of paired indices:
$$
 g=\{s_1t_1,s_2t_2,\cdots, s_{k+1}t_{k+1}\},
 $$
 where each of $\{s_j,t_j,j=1,\cdots,k+1\}$ equals one of four letters $s,t,\tilde i, \mu$(recalling $\tilde i=i+N_1$).
Here we would like to remind the readers that the size of g is k instead of k+1 in order to simplify the arguments in the following proof. Denote the size of the group $g$ by $k=k(g)$ and introduce the set $\mathfrak{G}_k=\{g:\ k(g)=k\}$ consisting of groups of size $k$.
Moreover, we require each group in $\mathfrak{G}_k$ to satisfy three conditions specified below: 
\begin{itemize}
\item[(i)] $s_1=s$ and $t_{k+1}=t$.
\item[(ii)] For $l\in [2,k+1]$ we have $s_l\in \{\tilde i, \mu\}$ and $t_{l-1}\in \{\tilde i,\mu \}$.
\item[(iii)] For $l\in [1,n]$ we have $t_{l-1}s_l\in \{\tilde i\mu,\mu \tilde i\}$.
\end{itemize}
As will be seen, groups $g$ are connected with the high order derivatives of $(\bbB_1\bbG(z)\bbB_2^T)_{st}$.
%
%
%
%
%
Moreover write $\bbF(z)=\sum\limits_{j=1}^3\Pi_j(z)$ where each $\Pi_j(z)$ corresponds to a non-zero block of $\Pi(z)$.

Also, to characterize the higher order derivative of each block conveniently we define groups $g^{(j)}$ of size $k$ to be the set of paired indices:
$$
g^{(j)}=\{s_{j1}t_{j1},s_{j2}t_{j2},\cdots, s_{j(k+1)}t_{j(k+1)}\},
$$
where each of $s_{jm}$ and $t_{jm}$ equals one of $s,t,i,\mu$.
 Moreover introduce the set $\mathfrak{G}_{jk}=\{g^{(j)}:\ k(g^{(j)})=k\}$ consisting of groups of size $k$. 
We require each group in $\mathfrak{G}_{jk}$ to satisfy conditions:
\begin{itemize}
\item[(i)] $s_{j1}=s$ and $b_{j (k+1)}=t$.
\item[(ii)] For $l\in [2,k+1]$ we have $s_{jl}\in \{i,\mu\}$ and $t_{j(l-1)}\in \{i,\mu\}$.
\item[(iii)] For $l\in [1,k]$ we have $t_{j(l-1)}s_{jl}\in \{i \mu,\mu i\}$.
\end{itemize}
As will be seen groups $g^{(j)}$ are linked to the high order derivatives of $(\bbB_1\Pi(z)\bbB_2^T)_{st}$.

We below associate a random variable $B_{s,t,i,\mu}(g,g^{(1)},\cdots,g^{(3)})$ with each group $g,g^{(j)},j=1,\cdots,3$. When $k(g)=k(g^{(j)})=0$ we define
\begin{equation}\label{b4}
B_{s,t,i,\mu}(g,g^{(1)},\cdots,g^{(3)}))= (\bbB_1\bbG(z)\bbB_2^T)_{st}-(\bbB_1\Pi(z)\bbB_2^T)_{st}.
\end{equation}
 When $k(g)\geq1$ or $k(g^{(j)})\ge1$, define
 \begin{equation}\label{a16}B_{s,t,i,\mu,\bbR_{2,\cdots,n},\mathcal{R}_{11,\cdots,3k+1}}(g,g^{(1)},...,g^{(3)})=
C_{s,t,i,\mu,\bbR_{2,\cdots,k},\mathcal{R}_{11,\cdots,3k+1}}(g,g^{(1)},\cdots,g^{(3)}))
\end{equation}$$-\sum_{j=1}^3(\bbB_1\mathcal{R}_{j1})_{(s_{j1}t_{j1})}(\mathcal{R}_{j2})_{(s_{j2}t_{j2})}...(\mathcal{R}_{jk})_{(s_{jn}t_{jk})}
(\mathcal{R}_{jk+1}\bbB_2^T)_{(s_{jk+1}t_{jk+1})},
$$
with
\begin{equation}\label{0526.2}C_{s,t,i,\mu,\bbR_{2,\cdots,k},\mathcal{R}_{11,\cdots,7k+1}}(g,g^{(1)},\cdots,g^{(3)}))=(\bbB_1G \bbA_5)_{(s_1t_1)}(\bbR_2)_{(s_2t_2)}...(\bbR_k)_{(s_kt_k)}(\bbA_4\bbG\bbB_2^T)_{(s_{k+1}t_{k+1})},\end{equation}
where $\bbR_j (2\le j\le k)$ has the expression of $\bbR_j=\bbA_4\bbG \bbA_5$
with $\bbA_4\in \{1, \bold\Delta \}$, $\bbA_5 \in \{1,\bold\Delta^T\}$ and the non-zero block $\mathcal{R}_{jl}$ belongs to $\mathcal{Q}(k)$ in (\ref{1202.1}).
Moreover the selection of $1$ and $\bold\Delta$ in $\bbA_4$ and $\bbA_5$ is subject to the constraint that the total number of $\bold\Delta$ and $\bold\Delta^T$ contained in  $B_{s,t,i,\mu,\bbR_{2,\cdots,k},\mathcal{R}_{11,\cdots,7k+1}}(g,g^{(1)},...,g^{(3)})$ is $k$. One should also notice that if $k(g)=1$, the terms $R_j$ will disappear.


\begin{eqnarray}\label{a38}
&&\frac{\partial^k }{\partial (X_{i\mu})^k}\Big([(\bbB_1\bbG(z)\bbB_2^T)_{st}-(\bbB_1\Pi(z)\bbB_2^T)_{st}]\mathcal{T}_N(\bbX)\Big)\\
&&=(-1)^k\sum_{g\in \mathfrak{G}_{k},g^{(j)}\in \mathfrak{G}_{jk} \atop{\bbR_i, i=2,...,k \atop \mathcal{R}_{jl}, j=1,..3, l=1,...,k+1} }B_{s,t,i,\mu,\bbR_{2,\cdots,k},\mathcal{R}_{11,\cdots,7k+1}}(g,g^{(1)},...,g^{(3)})\mathcal{T}_N(\bbX)+O_{\prec}(0).\nonumber
\end{eqnarray}
To simplify the notations, we furthermore omit $\bbR_{2\cdots,k},\mathcal{R}_{11,...,3k+1}, g^{(1)},...,g^{(3)}$ in the sequel and write
\begin{equation}\label{a36}B_{s,t,i,\mu}(g)=B_{s,t,i,\mu,\bbR_{2,\cdots,k},\mathcal{R}_{11,\cdots,3k+1}}(g,g^{(1)},...,g^{(3)}),\end{equation}
\begin{equation}\label{a37}C_{s,t,i,\mu}(g)=C_{s,t,i,\mu,\bbR_{2,\cdots,k},\mathcal{R}_{11,\cdots,3k+1}}(g,g^{(1)},...,g^{(3)}),\end{equation}
(here one should notice that the sizes of $g$ and $g^{(j)}$ are the same according to definition (\ref{a16})). Hence we have
\begin{eqnarray}\label{a14}
&\frac{\partial^k }{\partial (X_{i\mu})^k}\Big(|F_{st}(\bbX)|^{p}\Big)=(-1)^k\sum\limits_{k_1,...,k_{p/2}, \tilde k_1,...,\tilde k_{p/2}\in \mathbb{N} \atop
\sum_{r}(k_r+\tilde k_r)=k}\frac{k!}{\prod_r k_r!\tilde k_r!}
\\ &\times\prod\limits_{r=1}^{\frac{p}{2}}(\sum\limits_{g_r\in\mathfrak{G}_{k_r}\cup\mathfrak{G}_{jk_r}\atop{\bbR_i, i=2,...,k \atop \mathcal{R}_{jl}, j=1,..3, l=1,...,k+1} }\sum\limits_{\tilde g_r\in \mathfrak{G}_{\tilde k_r}\cup\mathfrak{G}_{ j\tilde k_r}\atop{\bar{\bbR}_i, i=2,...,k \atop \mathcal{\bar {R}}_{jl}, j=1,..3, l=1,...,k+1}}B_{s,t,i,\mu}(g_r)\overline{B_{s,t,i,\mu}(\tilde g_r)}\mathcal{T}^2_N(\bbX))
+O_{\prec}(0),\nonumber
\end{eqnarray}
where $g_r\in \mathfrak{G}_{ k_r}\cup\mathfrak{G}_{jk_r}$ means that the groups associated with the derivatives of $\bbG(z)$ belong to $\mathfrak{G}_{k_r}$ and the groups associated with the derivatives of $\Pi(z)$ belong to $\mathfrak{G}_{ jk_r}$. In view of (\ref{a14}) and (\ref{a8}) to prove (\ref{1123.7}) it then suffices to show that
\begin{equation}\label{1117.2}
N^{-k/2}\sum_{i=1}^{M_1}\sum_{\mu=1}^N\mathbb{E}\left[\prod_{r=1}^{p/2}B_{s,t,i,\mu}(g_r)\overline{B_{s,t,i,\mu}(\tilde g_r)}\mathcal{T}^{p}_N(\bbX)\right]=O((N^{24\delta}\Psi)^{p}+\|\mathbb{E}\bbL_p(\bbX)\|_{\infty}),
\end{equation}
for $4\le k \le 8p$ and groups $g_r\in\mathfrak{G}_{k_r}\cup\mathfrak{G}_{jk_r}$ satisfying $\sum_{r}k(g_r)=k$ and $k(g_0)=0$.
Define
$$\mathcal{H}_i = \mathcal{H}_{1i}+\mathcal{H}_{sti},\ \   \mathcal{H}_{1i}= |(\bbB_1\bbG)_{s\tilde i}|+|(\bbG\bbB_2^T)_{\tilde it}|, \ \ \mathcal{H}_{sti}= \sum_{\mathcal{R}\in \mathcal{Q}(n)}(|(\bbB_1\mathcal{R})_{si}|+|(\mathcal{R}\bbB_2^T)_{it}|),$$
 $$\mathcal{H}_{\mu} = \mathcal{H}_{1\mu}+\mathcal{H}_{st\mu},\ \  \mathcal{H}_{1\mu}= |(\bbB_1\bbG\bold\Delta)_{s \mu}|+|(\bold\Delta^T\bbG\bbB_2^T)_{\mu t}|, \ \ \mathcal{H}_{st\mu}=\sum_{\mathcal{R}\in \mathcal{Q}(k)}(|(\bbB_1\mathcal{R})_{a\mu}|+|(\mathcal{R}\bbB_2^T)_{\mu t}|.$$
 By the same arguments from (7.77)-(7.85) in \cite{WY}, we have
\begin{eqnarray}\label{1117.4}
|B_{s,t,i,\mu}(g_r)|\prec N^{2\delta (k(g)+1)},
\end{eqnarray}
(recall $k(g)=k(g^{(j)})$ from definition (\ref{a16})).
Likewise, for $k(g)\ge 1$, we have
\begin{eqnarray}\label{1117.5}
|B_{s,t,i,\mu}(g_r)|\prec (\mathcal{H}_i^2+\mathcal{H}_{\mu}^2)N^{2\delta(k(g_r)-1)},
\end{eqnarray}
while k(g)=1,
\begin{eqnarray}\label{1117.6}
|B_{s,t,i,\mu}(g_r)|\prec \mathcal{H}_i\mathcal{H}_{\mu}.
\end{eqnarray}
\begin{equation}\label{g28}\sum_{i=1}^{M_1}\mathcal{H}_{1i}^2+\sum_{\mu=1}^N \mathcal{H}_{1\mu}^2\prec N\phi_s^2+N\phi_t^2,\end{equation}
\begin{equation}\label{h28}\sum_{i\ \text{or}\ s \ \text{or} \ t}^{M_1}\mathcal{H}_{sti}^2+\sum_{s\ \text{or}\ \mu}^N \mathcal{H}_{s\mu}^2\prec 1,\end{equation}
where $i\ \text{or}\ s \ \text{or} \ t$ means the summation over either $i$ or $s$ or $t$ and
$$\phi_s^2= \frac{\Im(\bbB\bbG\bbB^*)_{ss}+\eta}{N\eta},$$
with $\bbB\in \mathcal{L}$ defined in (\ref{a5}).

\subsection{(\ref{0310.1}) under the condition (\ref{20150111})}\label{h0425-1}

This subsection is to remove the 3rd moment matching condition needed in the previous subsection. The proving strategy is similar to that in \cite{KY14}. Note that we have used the first three moments matching condition
when obtaining (\ref{1123.5}). We now have to estimate the term involving the third derivative in (\ref{0310.1}) (there $n$ now starts from three). To this end, it
is enough to prove (\ref{1123.7}) for $k=3$. This further reduces to proving that (\ref{1117.2}) holds for k=3 as well.

 Define $\bbv,\bbw\in \{\bbB_{(1)},...,\bbB_{(N-N_2+N_1+M_1)}\}$,
 where $\bbB_{(i)}$ represents the $i$-th column of $\bbB$ with $\bbB\in \mathcal{L}$, recalling $\mathcal{L}$ defined immediately below (\ref{a5}).
In the sequel, we focus on $\langle\bbv,(\bbG(z)-\Pi(z))\bbv\rangle$ only because the general inner product $\langle\bbv,(\bbG(z)-\Pi(z))\bbw\rangle$
 can be handled by the equality that
$$\langle\bbv,(\bbG(z)-\Pi(z))\bbw\rangle=\frac{1}{2}\big(\langle\bbv+\bbw,(\bbG(z)-\Pi(z))(\bbv+\bbw)\rangle-\langle\bbw,(\bbG(z)-\Pi(z))\bbw\rangle-\langle\bbv,(\bbG(z)-\Pi(z))\bbv\rangle\big).$$
 Here we absorb $\bbe_s^T\bbB_1$ and $\bbB_2\bbe_t$ used in (\ref{a5}) into new vectors $\bbv$ and $\bbw$.
 As a consequence denote $B_{s,t,i,\mu}(g)$ used in (\ref{1117.2}) by $A_{\bbv,i,\mu}(g)$ and ignore the conjugate symbol there for simplicity.
Below we take derivatives of $\bbG$ and $\bold\Pi(z)$ with respect to $X_{i\mu}$.  First, we calculate the derivative of $\bbG(z)$ with respect to $X_{i\mu}$.

To this end we expand $\bbG$ in terms of the $\tilde {i}$-th row of $\bbX$. 
Let $\bbH^{(\tilde i)}$ be the submatrix obtained from $\bbH$ by deleting its $\tilde{i}$-th row and $\tilde{i}$-th column (deleting the $i$-th row of $\bbX$ involved in $\bbH$) and define $\bbG^{(\tilde i)}=(\bbH^{\tilde i})^{-1}$.
Recalling the notations (\ref{0822.2}) and referring to the resolvent expansion formula (8.3) of \cite{KY14} we have the following expansion
\begin{eqnarray}\label{expand}\nonumber
\bbG_{\bbu\bbw}&=&\bbu(\tilde i)\bbw(\tilde i)\bbG_{\tilde i\tilde i}+\bbG^{(\tilde i)}_{\bbu\bbw}+\bbG_{\tilde i\tilde i}(\bbG^{(\tilde i)}\bold\Delta\bbX^T)_{\bbu i}(\bbX\bold\Delta^T\bbG^{(\tilde i)})_{i\bbw }
-\bbu(\tilde i)\bbG_{\tilde i\tilde i}(\bbX\bold\Delta^T\bbG^{(\tilde i)})_{i\bbw }\\ &&-\bbw(\tilde i)\bbG_{\tilde i\tilde i}(\bbG^{(\tilde i)}\bold\Delta\bbX^T)_{\bbu i}.
\end{eqnarray}
In fact, if we exchange the first with second ``row'' of $\bbH$ and its first ``column'' with its second ``column'' (i.e. we convert $\bbH$ to
$$\left(
  \begin{array}{ccc}
    0 & \bbX \bbU_1 &\bbX \bbU_2 \\
    \bbU^T_1 \bbX^T  &-zI & 0 \\
    \bbU^T_2\bbX^T &0 & I\\
  \end{array}
\right)$$), then (\ref{expand}) is similar to the formula (8.3) of \cite{KY14}. Here one may change the subscripts $s$ and $t$ in (\ref{a5}) when necessary,
which can still be absorbed into the vector $\bbv$.

The next aim is to correspondingly extract the entries $X_{i\mu}$ of $\bbX$ from $\bold\Pi(z)$ such that taking expectation on $X_{i\mu}$ later may use the independence
between ${\bold\Pi}^{(i)}$ (defined below) and $X_{i\mu}$, which serves the same function as (\ref{expand}). 
Since it is complicated to extract $X_{i\mu}$ from $\bold\Pi(z)$ we construct a proxy matrix ${\bold\Pi_1}$ below. Define  \begin{eqnarray}\label{0528.5}
{\bold\Pi_1}=\bbH_1^{-1}=\left(\begin{array}{cccc}
    m(z)\bbI &    \bbX \bbU_2    \\
    \bbU_2^T\bbX^T &   -\bbI
\end{array}\right)^{-1}.\end{eqnarray}
 By (\ref{0202.1}) we have
 \begin{eqnarray}\label{0112.1}
 {\bold\Pi_1}=\left(\begin{array}{cccc}
   \Gamma(\bbX,z) &    \Gamma(\bbX,z)\bbX \bbU_2    \\
    \bbU_2^T\bbX^T\Gamma(\bbX,z) &   -\bbI+\bbU_2^T\bbX^T\Gamma(\bbX,z)\bbX\bbU_2
\end{array}\right).
\end{eqnarray}
The key observations are that the first diagonal matrix of ${\bold\Pi_1}$ is the same as the second block of $\bold\Pi$ and that the second diagonal block of ${\bold\Pi_1}$ and the third block of $\bold\Pi$ differ by $2\bbI$. Since ${\bold\Pi}(z)$ is a $3\times 3$ block diagonal matrix, we can split $\bbv$ into 3 parts $\bbv_1$, $\bbv_2$ and $\bbv_3$: $\bbv=(\bbv_1^T,\bbv_2^T,\bbv_3^T)^T$ corresponding to each block. In other words, $$\Pi_{\bbv\bbv}=m(z)\bbv_1^T\bbv_1+\Gamma(\bbX,z)_{\bbv_2\bbv_2}+(\bbI+\bbU_2^T\bbX^T\Gamma(\bbX,z)\bbX\bbU_2)_{\bbv_3\bbv_3}
=\Pi_{\hat\bbv_1\hat\bbv_1}+\Pi_{\hat\bbv_2\hat\bbv_2}+\Pi_{\hat\bbv_3\hat\bbv_3},$$ where $\hat\bbv_1=(\bbv_1^T,0,0)^T$, $\hat\bbv_2=(0,\bbv_2^T,0)^T$, $\hat\bbv_3=(0,0,\bbv_3^T)^T$ and their sizes are the same as $\bbv$'s size. Moreover, we set $\tilde \bbv_2=(\bbv_2^T,0)^T$ and $\tilde \bbv_3=(0,\bbv_3^T)^T$, where their sizes are both equal to the size of $(\bbv_2^T,\bbv_3^T)^T$. It follows that
\begin{equation}\label{b2}
\Pi_{\hat\bbv_2\hat\bbv_2}=(\Pi_1)_{\tilde\bbv_2\tilde\bbv_2}.
\end{equation}
Let $\Pi_1^{(i)}=(\bbH_1^{(i)})^{-1}$, where $\bbH_1^{(i)}$ is the sub matrix of $\bbH_1$ by deleting its i-th row and i-th column. One can similarly define
$\bold\Pi^{(i)}$ from $\bold\Pi$.
Applying (\ref{expand}) (or formula (8.3) of \cite{KY14}) to $\Pi_1$ and by (\ref{0112.1}) we have
$$(\Pi_1)_{\tilde\bbv_3\tilde\bbv_3}=(\Pi^{(i)}_1)_{\tilde\bbv_3\tilde\bbv_3}+(\Pi_1)_{ii}(\Pi^{(i)}_1\bbU_2^T\bbX^T)_{\tilde\bbv_3i}(\bbX \bbU_2 \Pi_1^{(i)})_{i \tilde\bbv_3},\quad i\in(1,M_1) $$
(in this case note that $\tilde\bbv_3(i)=0$ for $i\in\{1,2,...,M_1\}$).
Together with $\Pi_{\hat\bbv_3\hat\bbv_3}=2\bbv_3^T\bbv_3+(\Pi_1)_{\tilde\bbv_3\tilde\bbv_3}$ and $\Pi^{(i)}_{\hat\bbv_3\hat\bbv_3}=2\bbv_3^T\bbv_3+(\Pi^{(i)}_1)_{\tilde\bbv_3\tilde\bbv_3}$, we conclude that
\begin{equation}\label{b1}
\Pi_{\hat\bbv_3\hat\bbv_3}=\Pi^{(i)}_{\hat\bbv_3\hat\bbv_3}+(\Pi_1)_{ii}(\Pi^{(i)}_1\bbU_2^T\bbX^T)_{\tilde\bbv_3i}(\bbX \bbU_2 \Pi_1^{(i)})_{i \tilde\bbv_3}.
\end{equation}
Similarly via (\ref{expand}) and (\ref{b2}) we have
\begin{eqnarray}\label{b3}
\Pi_{\hat\bbv_2\hat\bbv_2}&=&\bbv(\tilde i)^2\Pi_{\tilde i\tilde i}+\Pi^{(i)}_{\hat\bbv_2\hat\bbv_2}+(\Pi_1)_{ii}(\Pi^{(i)}_1\bbU_2^T\bbX^T)_{\tilde\bbv_2i}(\bbX \bbU_2 \Pi_1^{(i)})_{i \tilde\bbv_2}\non
&&-\bbv(\tilde i)(\Pi_1)_{ii}(\Pi^{( i)}_1\bbU_2^T\bbX^T)_{\tilde\bbv_2i}-\bbv(\tilde i)(\Pi_1)_{ii}(\bbX \bbU_2 \Pi_1^{(i)})_{i \tilde\bbv_2}.
\end{eqnarray}
One should notice that we will subtract $\bbv(\tilde i)^2\Pi_{\tilde i\tilde i}$ at  (\ref{0528.2}).
In the sequel we use $\bbv_2$ (also $\bbv_3$) to represent $\bbv_2$, $\hat\bbv_2$ or $\tilde\bbv_2$ depending on the dimension of the matrix we deal with if there is no confusion.

Recalling $\tilde i=i+N_1$, since the expectation of $(\bbG_{\tilde i\tilde i}-\Pi_{\tilde i\tilde i})$ is difficult to handle in the following proof, we replace $\bbA_{\bbv,i,\mu}(g)$ by $\hat \bbA_{v,i,\mu}(g)$:
\begin{eqnarray}\label{0528.2}
\hat \bbA_{\bbv,i,\mu}(g)=
 \begin{cases}
    \bbG_{\bbv\bbv}-\Pi_{\bbv\bbv}-\bbv(\tilde i)^2(\bbG_{\tilde i\tilde i}-\Pi_{\tilde i\tilde i})  &  \text{if} \  k(g)=0\\
    \bbA_{\bbv,i,\mu}(g)
        &  \text{if} \ k(g)\ge 1\\
 \end{cases}
\end{eqnarray}
By the following Lemma, it suffices to show (\ref{1117.2}) holds for $\hat \bbA_{\bbv,i,\mu}(g)$ when k=3:
 \begin{lem}\label{0823-1}
 If (\ref{1117.2}) holds for $\hat \bbA_{\bbv,i,\mu}(g)$ when k=3, then it also holds for $\bbA_{\bbv,i,\mu}(g)$ when k=3.
 \end{lem}
 \begin{proof} The proof of this lemma is similar to that in \cite{KY14}. 
 \end{proof}
Referring to (\ref{expand}), one can find out that the expansion of $\bbG_{\bbu\bbw}$ can be reorganized as follows
 \begin{eqnarray}\label{0917.1}
\bbG_{\bbu\bbw}=\bbG_{\bbu\bbw}^{(0)}+\bbG_{\tilde i\tilde i}\bbG_{\bbu\bbw}^{(1)}+\bbv(\tilde i)G_{\tilde i\tilde i}\bbG_{\bbu\bbw}^{(2)},
 \end{eqnarray}
 where $\bbu,\bbw\in \{\bbv,\bbe_{\tilde i},\Delta\bbe_{\mu}\}$, $\bbu$ and $\bbw$ can not be equal to $\bbv$ at the same time. Moreover, $\bbG_{\bbu\bbw}^{(0)}$ is $\bbX^{(i)}$ measurable (obtained from $\bbX$ by deleting the $i$ th row) and independent of the $i$-th row of $\bbX$, and $\bbG_{\bbu\bbw}^{(1)}$ and $\bbG_{\bbu\bbw}^{(2)}$ do not include $\bbG_{\tilde i\tilde i}$ and $\bbv(\tilde i)$.
 We illustrate some examples as follows:
\begin{eqnarray}\label{0528.3}\bbG_{\bbv \tilde i}&=&\bbv(\tilde i)\bbG_{\tilde i\tilde i}-\bbG_{\tilde i\tilde i}(\bbG^{(\tilde i)}\Delta\bbX^T)_{\bbv i}=\bbG_{\tilde i\tilde i}\bbG_{\bbv\tilde i}^{(1)}+\bbv(\tilde i)\bbG_{\tilde i\tilde i}\bbG_{\bbv\tilde i}^{(2)}, \\
\label{0528.4}
  \bbG_{\bbv \bbv_{\mu}}&=&\bbG^{(\tilde i)}_{\bbv\bbv_{\mu}}+\bbG_{\tilde i\tilde i}(\bbG^{(\tilde i)}\bold\Delta\bbX^T)_{\bbv i}(\bbX\bold\Delta^T\bbG^{(\tilde i)})_{i\bbv_{\mu}}-\bbv(\tilde i)\bbG_{\tilde i\tilde i}(\bbX\bold\Delta^T\bbG^{(\tilde i)})_{i\bbv_{\mu}}\non
  &=&\bbG_{\bbv \bbv_{\mu}}^{(0)}+\bbG_{\tilde i\tilde i}\bbG_{\bbv \bbv_{\mu}}^{(1)}+\bbv(\tilde i)\bbG_{\tilde i\tilde i}\bbG_{\bbv \bbv_{\mu}}^{(2)},
  \\
  \label{h0610.1}
  \bbG_{\tilde i \bbv_{\mu}}&=&-\bbG_{\tilde i\tilde i}(\bbX\Delta^T\bbG^{(\tilde i)})_{i\bbv_{\mu} }=\bbG_{\tilde i\tilde i}\bbG_{\tilde i \bbv_{\mu}}^{(1)},
  \end{eqnarray}
 where $\bbv_{\mu}=\Delta\bbe_{\mu}$.

 To conveniently write down the expansions of $\bbG_{\bbu\bbw}$ and $\bold\Pi$ in $A_{\bbv,i,\mu}(g)$ such as (\ref{expand}), (\ref{b1}) and (\ref{b3}) in terms of $\bbG_{\bbu\bbw}^{(j)},j=0,1,2$,
we below introduce the definitions of a tagged group, a refinement of the preceding definition of group $A_{\bbv,i,\mu}(g)$, as in \cite{KY14}.
\begin{deff}\label{0826-1}
A tagged group is a pair $(g, \sigma)$, where $\sigma=(\sigma(l)),\cdots,\sigma(k(g)+1)$ with $\sigma(l)\in\{0,1,2\}$ (denote it by $\sigma=(\sigma(l))_{l=1}^{k(g)+1}$).

(i) If $k(g)=0$, we set
$$A_{\bbv,i,\mu}(g, 0)=\bbG^{(\tilde i)}_{\bbv\bbv}-\bold\Pi^{(\tilde i)}_{\bbv\bbv}=\bbG^{(\tilde i)}_{\bbv\bbv}-\Pi^{( i)}_{\bbv_1\bbv_1}-\Pi^{(i)}_{\bbv_2\bbv_2}-\Pi^{(i)}_{\bbv_3\bbv_3},$$
\begin{eqnarray*}
A_{\bbv,i,\mu}(g, 1)&=&(\bbG^{(\tilde i)}\bold\Delta\bbX^T)_{\bbv i}(\bbX\bold\Delta^T\bbG^{(\tilde i)})_{i\bbv}-(\Pi_1)_{ii}(\Pi^{( i)}_1\bbU_2^T\bbX^T)_{\bbv_2i}(\bbX \bbU_2 \Pi_1^{(i)})_{i \bbv_2}\non
&&-(\Pi_1)_{ii}(\Pi^{(i)}_1\bbU_2^T\bbX^T)_{\bbv_3i}(\bbX \bbU_2 \Pi_1^{(i)})_{i \bbv_3},
\end{eqnarray*}
$$A_{\bbv,i,\mu}(g, 2)=-(\bbG^{(\tilde i)}\bold\Delta\bbX^T)_{\bbv i}-(\bbX\bold\Delta^T\bbG^{(\tilde i)})_{i \bbv}+(\Pi_1)_{ii}(\Pi^{( i)}_1\bbU_2^T\bbX^T)_{\bbv_2i}+(\Pi_1)_{ii}(\bbX \bbU_2 \Pi_1^{(i)})_{i \bbv_2}.$$

(ii) If $k(g)\ge 1$, we define
$$A_{\bbv,i,\mu}(g, \sigma)=[\bbG_{\bbv\bbt_1}]^{\sigma(1)}[G_{\bbs_2\bbt_2}]^{\sigma(2)}...[G_{\bbs_{n(\omega)+1}\bbv}]^{\sigma(k(g)+1)}.$$
Here one should notice that the second term at the right hand side of (\ref{a16}) is ignored since we will discuss how to deal with it later.  

\end{deff}

In the above definition, we write $\Pi^{( i)}_1\bbU_2^T\bbX^T=\Pi^{( i)}_1(0 \ \bbX\bbU_2)^T$ for simplicity. When $k(g)=0$, $A_{\bbv,i,\mu}(g, j),j=1,2,3$ come from the expansion (see (\ref{expand}), (\ref{b1}) and (\ref{b3})) of $(\bbG_{\bbv\bbv}-\bold\Pi_{\bbv\bbv})$ by deducting the diagonal entry and one may refer to (\ref{b4}). Observe that $\bbA_{\bbv,i,\mu}(g, \sigma)$ is a homogeneous polynomial of the variable $X_{i\mu}, \mu=1,..., N$ with the coefficients being $\bbX^{(i)}$-measurable.  One should also notice that we have not considered the derivative of $\Pi(z)$ for $k(g)\ge 1$ to be discussed at the end of this section (refer to (\ref{a16})). 
Also in the sequel, we omit the terms involving $\Pi_1$ or $\Pi^{(i)}_1$ for $k(g)=0$,
which can be handled similarly by checking the following arguments carefully. Below we frequently replace $\mathcal{T}_N(\bbX)$ by $\mathcal{T}_N(\bbX^{(i)})$ where $\mathcal{T}_N(\bbX^{(i)})$ is obtained from
$\mathcal{T}_N(\bbX))$ with $\bbX$ replaced by $\bbX^i$. The purpose of such a replacement is that we need to extract $X_{i\mu}$, $\mu=1,...,N$ from all $\bbX$ involved in $\bbG$, $\bold\Pi$ and $\mathcal{T}_N(\bbX^{(i)})$ so that we may use independence between $X_{i\mu}$ and  $\bbG^{(\tilde i)}$, $\bold\Pi^{i}$ and $\mathcal{T}_N(\bbX^i)$.  The replacement starts from $\mathcal{T}_N(\bbX)$ to $\mathcal{T}_N(\bbX)\mathcal{T}_N(\bbX^{(i)})$ and finally to $\mathcal{T}_N(\bbX^{(i)})$. As before one should note that $\mathcal{T}_N(X^{(i)})=1$ with high probability (recalling (\ref{0821.4})). However in order to simplify notation we do not explicitly write down such steps below and only state $\mathcal{T}_N(\bbX^{(i)})$ and $\mathcal{T}_N(\bbX)$ whenever necessary.


For $\sigma=(\sigma(l))_{l=1}^{k(g)+1}$ define
$$|\sigma|_{\mathbf{i}}=\sum_l \bbI(\sigma(l)\ge 1),\  \ |\sigma|_{\bbv}=\sum_l \bbI(\sigma(l)=2).$$
Notice that $|\sigma|_{\mathbf{i}}$ and $|\sigma|_{\bbv}$ do not depend on $\bbi$ and $\bbv$ since the  expansion (\ref{expand}) always works for any $i$ and $\bbv$.
From the above we may write
$$
\hat \bbA_{\bbv,i,\mu}(g)=\sum\limits_{\sigma_r}\bbA_{\bbv,i,\mu}(g,\sigma_r)(\bbG_{\tilde i\tilde i })^{|\sigma_r|_i}\bbv(\tilde i)^{|\sigma_r|_{\bbv}}
$$
where $\sigma_r=(\sigma_r(l))_{l=1}^{k(g)+1}$ with $\sigma_r(l)\in\{0,1,2\}$.
In view of the arguments above, it suffices to show the following Lemma.
\begin{lem}\label{0823-3}
Suppose that for $r\le q$, $k(g_r)\ge 1$ and $k(g_r)=0$ for $p\ge r\ge q+1$ subject to $\sum_rk(g_r)=3$. For r=1,...,p, set $\sigma_r=(\sigma_r(l))_{l=1}^{k(g)+1}$ with $\sigma_r(l)\in\{0,1,2\}$. Then we have for all $\bbv\in \mathcal{L}$
\begin{eqnarray}\label{0823.7}
N^{-3/2}\sum_{i=1}^{M_1}\sum_{\mu=1}^{N}\left|\mathbb{E}\left((\bbG_{\tilde i\tilde i })^{d_i}\bbv(\tilde i)^{d_{\bbv}}\prod_{r=1}^p\bbA_{\bbv,i,\mu}(g_r,\sigma_r)\mathcal{T}_N(\bbX)\right)\right|=O((N^{C\delta}\Psi)^p+\mathbb{E}\|\bbL_p(\bbX)\|_{\infty}),
\end{eqnarray}
where
\begin{eqnarray}\label{0823.8}
d_{\blacktriangle}=\sum_{r=1}^p|\sigma_r|_{\blacktriangle}
\end{eqnarray}
for $\blacktriangle=i,\bbv $.
\end{lem}
Before proving Lemma \ref{0823-3}, we give a rough bound of $\bbA$ first. This bound helps us to connect the left hand side of (\ref{0823.7}) with the desired bound $\bbL_p(\bbX)$.
\begin{lem}\label{0804-1}[Rough bounds on $\bbA_{\bbv,i,\mu}(g,\sigma)$.] Assume that ($A_{m-1}$) holds. Then for $z\in S_m$, we have
\begin{eqnarray}\label{0804.2}
|\bbA_{\bbv,i,\mu}(g,\sigma)|\prec N^{2\delta(k(g)+1)}.
\end{eqnarray}
If $k(\omega)=0$, then
\begin{eqnarray}\label{0804.3}
|\bbA_{\bbv,i,\mu}(g,\sigma)|\prec N^{(C_0/2+1)\delta}\Psi+F_{\bbv}(X)+F_{\bbe_{\tilde i}}(X).
\end{eqnarray}
\end{lem}
\begin{proof}[Proof of Lemma \ref{0804-1}]
By the arguments similar to Lemma \ref{1116-3}, it is easy to get the following bound for $z\in S_m$ given $A_{m-1}$:
\begin{eqnarray}\label{0804.4}
(\bbG-\Pi)_{\bbv\bbv}=O_{\prec}(N^{2\delta}), \  \Im \bbG_{\bbv\bbv}\prec N^{2\delta}(\Im m+N^{C_0\delta}\Psi).
\end{eqnarray}
The remaining argument is similar to that for Lemma 8.9 in \cite{KY14} and we ignore details here.
\end{proof}

Noticing that $\bbG_{\tilde i\tilde i}$ in the left hand side of (\ref{0823.7}) contains $X_{i\mu}, \mu=1,...,N$, we need to extract $X_{i\mu}$ from it. To this end, we use the following resolvent expansion for the diagonal entry of $\bbG$
$$\bbG_{\tilde i\tilde i}=-(Y_{i}+Z_{i})^{-1},$$
where
$$Y_{i}=\mathbb{E}[(\bbX\bold\Delta^T\bbG^{(\tilde i)}\bold\Delta\bbX^T)_{ii}|X^{(i)}]=\frac{1}{N}\sum_{j}(\bold\Delta^T\bbG^{(\tilde i)}\bold\Delta)_{jj}, \ \ Z_{i}=(\bbX\bold\Delta^T\bbG^{(\tilde i)}\bold\Delta\bbX^T)_{ii}-Y_{i}.$$
Using the large deviation bound, we find out that $|Z_i|\prec N^{-\tau/2+2\delta}$. 
This, together with Taylor's expansion, implies that there exists a constant $K=K(\tau)$ such that
\begin{eqnarray}\label{h0608.1}
\bbG_{\tilde i\tilde i}=\sum_{k=0}^K\frac{Z_{i}^k}{k!(-Y_{i})^{k+1}}+O_{\prec}(N^{-10}).
\end{eqnarray}
Expanding further $Z_i$, we have
\begin{eqnarray}\label{h0417.2}
\bbG_{\tilde i\tilde i}=\sum_{k=0}^KY_{i,k}(\bbX\bold\Delta^T\bbG^{(\tilde i)}\bold\Delta\bbX^T)_{ii}^k+O_{\prec}(N^{-10}),
\end{eqnarray}
where the coefficients $Y_{ik}$ including $1/Y_i$ and $Y_i$ are $\bbX^{(i)}$ measurable.
In order to apply Lemma \ref{1123-3} an upper bound of $|Y_{i,k}|$ is needed. From (\ref{h0608.1}) and the definition of $Z_i$ one can see that $Y_{i,k}$ is a finite order polynomial function of $(Y_i)^{-1}$ and $Y_i$, Therefore it suffices to develop upper and lower bounds of $Y_i$. Recalling the definition of $\bold \Delta$ at (\ref{1125.10}), we have
$$ \bbU_1^T\bbU_1=\bbI , \ \bbU_2^T\bbU_2=\bbI, \ \text{and} \ \bold\Delta\bold\Delta^T=\left(
  \begin{array}{ccc}
    \bbI & 0&0 \\
    0  &0 & 0 \\
    0 &0 & \bbI\\
  \end{array}
\right).$$
From (\ref{1219.1}) we have
$$|Y_i|\ge \Im Y_i=\frac{1}{N}tr \Im (\bbG^{(\tilde i)}\bold\Delta\bold\Delta^T)=\frac{1}{N} tr \bbA_2^{(\tilde i)}\Im\bbG^{(\tilde i)}_{N_1}(\bbA_2^{(\tilde i)})^T\ge\frac{1}{N} tr \Im\bbG^{(\tilde i)}_{N_1}\geq C\eta,$$
where $\bbA_j^{(\tilde i)}$ are respectively obtained from $\bbA_j,j=2,3$ by deleting the $i$-th row of $\bbX$.
On the other hand we conclude from (\ref{1219.1}) that
 $$ |Y_i|\le \|\bbA^{(\tilde i)}_3\|+\frac{\|\bbA^{(\tilde i)}_2\|^2}{\eta}.$$
These, together with Lemma \ref{1121-1} and an estimate similar to (\ref{0821.2}),  implies that there exists a constant c such that
\begin{equation}\label{b7}
N^{-c} \le |Y_i\mathcal{T}_N(\bbX^{(i)})|\le N^{c}.
\end{equation}

We are now in a position to replace $\bbG_{\tilde i\tilde i}$ by $\sum_{k=0}^K\mathcal{Y}_{i,k}(\bbX\bold\Delta^T\bbG^{(\tilde i)}\bold\Delta\bbX^T)_{ii}^k$ because
 \begin{eqnarray}\label{0825.2}
&&\mathbb{E}\prod_{r=1}^p|\bbA_{\bbv,i,\mu}(g_r,\sigma_r)|\left|\bbG_{ii}^{d_i}-(\sum_{k=0}^K\mathcal{Y}_{i,k}(\bbX\bold\Delta^T\bbG^{(\tilde i)}\bold\Delta\bbX^T)_{ii}^k)^{d_{i}}\mathcal{T}_N(\bbX^{(i)}))\right|\prec N^{-10+2\delta d_i}\mathbb{E}\prod_{r=1}^p|\bbA_{\bbv,i,\mu}(g_r,\sigma_r)|\non
&&\prec N^{-5}((N^{C\delta}\Psi)^p+\bbL_p(\bbX)),
\end{eqnarray}
where we apply Lemma \ref{0804-1}, Lemma \ref{1123-3}, (\ref{h0417.2}), (\ref{b7}) and the fact that there are at most three r such that $k(g_r)\ge 1$. In view of (\ref{0825.2}) proving Lemma \ref{0823-3} reduces to showing
\begin{eqnarray}\label{0825.3}
&&N^{-3/2}\sum_{i=1}^{M_1}\sum_{\mu=1}^{N}\left|\mathbb{E}\left((\sum_{k=0}^K Y_{i,k}(\bbX\bold\Delta^T\bbG^{(\tilde i)}\bold\Delta\bbX^T)_{ii}^k)^{d_i}\bbv(i)^{d_{\bbv}}\prod_{r=1}^p\bbA_{\bbv,i,\mu}(g_r,\sigma_r)\mathcal{T}_N(\bbX^{(i)})\right)\right|\\
&&=O((N^{C\delta}\Psi)^p+\mathbb{E}\|\bbL_p(\bbX)\|_{\infty}).\nonumber
\end{eqnarray}
We further expand 
$$(\sum_{k=0}^KY_{i,k}(\bbX\bold\Delta^T\bbG^{(\tilde i)}\bold\Delta\bbX^T)_{ii}^k)^{d_{i}}=\sum_{k=0}^{Kd_{i}}\mathcal{Y}_{i,k}(\bbX\bold\Delta^T\bbG^{(\tilde i)}\bold\Delta\bbX^T)_{ii}^k,$$
where the coefficient $\mathcal{Y}_{i,k}$ is $\bbX^{(i)}$ measurable and bounded by
$$|\mathcal{Y}_{i,k}|\prec N^{Cd_{i}\delta}.$$
Here we don't need the explicit expression of $\mathcal{Y}_{i,k}$ and its upper bound is enough by checking the following arguments carefully.
For any tagged group $(g,\sigma)$, we write
$$\bbA_{\bbv,i,\mu}(g,\sigma)=\bbA^-_{\bbv,i,\mu}(g,\sigma)\bbA^+_{\bbv,i,\mu}(g,\sigma),$$
where $\bbA^-_{\bbv,i,\mu}(g,\sigma)$ is measurable with respect to $\bbX^i$ and $\bbA^+_{\bbv,i,\mu}(g,\sigma)$ is a product of the terms
$$(\bbG^{(\tilde i)}\bold\Delta \bbX^T)_{\bbx i}, \ \text{or} \ (\bbX\bold\Delta^T\bbG^{(\tilde i)})_{i\bbx }, \  \bbx\in \{\bbv, \bold\Delta\bbe_{\mu}\}.$$
So the left hand side of (\ref{0825.3}) is bounded by
\begin{equation}\label{0725.2}
N^{-3/2}N^{Cd_{i}\delta}\sum_{i=1}^{M_1}\sum_{\mu=1}^N \bbv(\tilde i)^{d_{\bbv}}E\Big[\Big|\mathbb{E}\prod_{r=1}^{p}\bbA^-_{\bbv,i,\mu}(g,\sigma)\Big|\Big|\mathbb{E}_i\prod_{r=1}^{p} \bbA^+_{\bbv,i,\mu}(g_r,\sigma_r)(\bbX\bold\Delta^T\bbG^{(\tilde i)}\bold\Delta\bbX^T)_{ii}^k\mathcal{T}_N(\bbX^{(i)})\Big|\Big]
\end{equation}
where $\mathbb{E}_i$ stands for taking expectation over the random variables at the i-th row of $\bbX$ (conditional expectation).

We below first show that for the inner conditional expectation and $k\le Kd_i$,
\begin{eqnarray}\label{0117.1}
\left|\mathbb{E}_i\prod_{r=1}^{p} \bbA^+_{\bbv,i,\mu}(g_r,\sigma_r)(\bbX\bold\Delta^T\bbG^{(\tilde i)}\bold\Delta\bbX^T)_{ii}^k\mathcal{T}_N(\bbX^{(i)})\right|\prec N^{-\mathbf{1}(d_{\bbv}=0)/2}(N^{(C_0/2+C)\delta}\Psi)^{d_{\bbx}-\mathbf{1}(d_{\bbx}\ge 3)}.
\end{eqnarray}
The proof of (\ref{0117.1}) is similar to that of Lemma 8.11 in \cite{KY14} and the transitional arguments from (\ref{0117.1}) to (\ref{0825.3}) are the same as
those from (8.32) in \cite{KY14} to the end of section 8.5 in \cite{KY14}. We below only list some difference involved in our derivatives when proving (\ref{0117.1}). Define
\begin{eqnarray}\label{0529.3}
d_{\bbx}=\sum_{r=1}^{p}\deg(\bbA^+_{\bbv,i,\mu}(g_r,\sigma_r))=\sum_{r=1}^{p}\deg(\bbA_{\bbv,i,\mu}(g_r,\sigma_r)),
\end{eqnarray}
where $\deg(\bbA_{\bbv,i,\mu}(g_r,\sigma_r))$ stands for the degree of the polynomial $\bbA_{\bbv,i,\mu}(g_r,\sigma_r)$ in terms of $X_{i\mu}$.
We abbreviate $d_{\bbx}$ by $d$. Recalling the definition of $\bbA^+$, we have
$$\prod_{r=1}^{p}\bbA^+_{\bbv,i,\mu}(g_r, \sigma_r)(\bbX\bold\Delta^T\bbG^{(\tilde i)}\bold\Delta\bbX^T)_{ii}^k=\sum_{j_1,...,j_{d+2k}=1}^N\mathcal{G}_{j_1,...,j_d}\tilde {\mathcal{G}}_{j_{d+1},...,j_{d+2k}}\prod_{l=1}^{d+2k}X_{ij_l},$$
where $\mathcal{G}_{j_1,...,j_d}$ is a product of $d$ terms in the set $\{(\bold\Delta\bbG^{(\tilde i)})_{j_l \bbv}, (\bbG^{(\tilde i)}\bold\Delta^T)_{\bbv j_l }, (\bold\Delta^T\bbG^{(\tilde i)}\bold\Delta)_{j_l \mu}, (\bold\Delta^T\bbG^{(\tilde i)}\bold\Delta)_{\mu j_l },$ $l=1,2,..,d.\}$ and ${\mathcal{G}}_{j_{d+1},...,j_{d+2k}}$ is a product of k terms in the set $\{(\bold\Delta\bbG^{(\mu)}\bold\Delta^T)_{j_l j_{l'}}, j_l, j_{l'}=d+1,d+2,..,d+2k.\}$. Since $\mathbb{E}_{\mu}X_{ij}=0$, it is easy to see that the conditional expectation on $\prod_{l=1}^{d+2k}X_{i j_l}$ is nonzero only if each index appears at least twice. The set $\{1,...,d+2k\}$ can be reorganized by several blocks such that each block contain the same indices $j_l$. For example, if $b_1$ and $b_2$ are two different blocks, then the indexes belonging to $b_1$ (and $b_2$) are all equal and any two indexes $a_1\in b_1$, $a_2\in b_2$ are not equal.   For a block $b\subset \{1,...,d+2k\}$, we define $d_b=|b\cap\{1,...,d\}|$ and $k_b=|b\cap\{d+1,...,d+2k\}|$. Here $d_b$ means the number of indices equal to $b$ from $\{1,\cdots,d\}$ and $k_b$ means the number of indices equal to $b$ from $\{d+1,\cdots,d+k\}$. Moreover, we suppose there are L blocks. Hence, we reorganize the  summation as follows:
\begin{eqnarray}\label{0725.3}
&&|\mathbb{E}_{i}\prod_{r=1}^{p}\bbA^+_{\bbv,i,\mu}(g_r, \sigma_r)(\bbX\bold\Delta^T\bbG^{(\tilde i)}\bold\Delta\bbX^T)_{ii}^k|\prec C_qN^{2\delta k}\max_{L}\max_{\{d_l\}}\max_{\{k_l\}}\sum_{j_1,...,j_L}\times\\
&&\prod_{l=1}^L(\mathbb{E}|\bbX_{ij_l}|^{d_l+k_l}(|(\bold\Delta^T\bbG^{(\tilde i)})_{j_l \bbv}|+|(\bbG^{(\tilde i)}\bold\Delta)_{\bbv j_l }|+|(\bold\Delta^T\bbG^{(\tilde i)}\bold\Delta)_{j_l \mu}|+|(\bold\Delta^T\bbG^{(\tilde i)}\bold\Delta)_{\mu j_l }|)^{d_l}),\nonumber
\end{eqnarray}
where $d_l$ and $k_l$ satisfy
\begin{eqnarray}\label{0725.5}\sum_{l=1}^Ld_l=d, \ \sum_{l=1}^Lk_l=2k, \  d_l+k_l\ge 2.\end{eqnarray}
It is straightforward to see that the right hand side of (\ref{0725.3}) is bounded by
\begin{equation}\label{b9}
C_qN^{2\delta k}N^{-d/2-k}\prod_{l=1}^L\big(\sum_{j}\big(|(\bold\Delta^T\bbG^{(\tilde i)})_{j \bbv}|+|(\bbG^{(\tilde i)}\bold\Delta)_{\bbv j }|+|(\bold\Delta^T\bbG^{(\tilde i)}\bold\Delta)_{j \mu}|+|(\bold\Delta^T\bbG^{(\tilde i)}\bold\Delta)_{\mu j }|\big)^{d_l}\big).
\end{equation}
The upper bound of the above term is
$$\sum_{j}(|(\bold\Delta^T\bbG^{(\tilde i)})_{j \bbv}|+|(\bbG^{(\tilde i)}\bold\Delta)_{\bbv j }|+|(\bold\Delta^T\bbG^{(\tilde i)}\bold\Delta)_{j \mu}|+|(\bold\Delta^T\bbG^{(\tilde i)}\bold\Delta)_{\mu j }|)^{d_l}\prec N(N^{(C_0/2+1)\delta}\Psi)^{d_l\wedge 2}N^{2\delta [d_l-2]_+},$$
following from $\sum_{j}|(\bold\Delta^T\bbG^{(\tilde i)})_{j \bbv}|^2\lesssim \frac{\Im G^{(\tilde i)}_{\bbv\bbv}}{\eta}$, $\sum_{j}|(\bold\Delta^T\bbG^{(\tilde i)}\bold\Delta)_{j \mu}|^2\lesssim \frac{\Im G^{(\tilde i)}_{\bbv_{\mu}\bbv_{\mu}}}{\eta}$. 
Therefore we have
$$|\mathbb{E}_{\mu}\prod_{r=1}^{p}\bbA^+_{v,i,\mu}(g_r, \sigma_r)(\bbX\bold\Delta^T\bbG^{(\tilde i)}\bold\Delta\bbX^T)_{ii}^k|\prec \max_L \max_{\{d_l\}} \max_{\{k_l\}}N^{-d/2-k+L}
(N^{(C_0/2+1)\delta}\Psi)^{\sum_l(d_l\wedge 2)}N^{2\delta \sum_l [d_l-2]_+}.$$
The remaining arguments are the same as those below (8.29) in \cite{KY14} (to the end of Section 8.5 in \cite{KY14}).

We next consider the derivative of $\bold\Pi(z)$ since we have only considered the derivative of $\bbG(z)$ for $k(g)\ge 1$(one can refer to the definition of $\bbA_{\bbv,i,\mu}(g,\sigma)$). That is to say, we aim at proving (\ref{1117.2}) for $k=3$ but  we only proved
\begin{equation}\label{0612.1}
N^{-3/2}\sum_{i=1}^{M_1}\sum_{\mu=1}^N\mathbb{E}\left[\prod_{r=1}^{q}C_{s,t,i,\mu}(g_r)\prod_{r=q+1}^{p}B_{s,t,i,\mu}(g_r)\mathcal{T}^{p}_N(\bbX)\right]=O((N^{24\delta}\Psi)^{p}+\|\mathbb{E}\bbL_p(\bbX)\|_{\infty}),
\end{equation}
for $q\le 3$, $\sum_{r=1}^qk(g_r)=3$, $k(g_r)\ge 1$, $r\le q$ and $g_r=0$, $r\ge q+1$, where the definitions of $C_{s,t,i,\mu}$ and $B_{s,t,i,\mu}$ are given at (\ref{a36}) and (\ref{a37}). One may refer to (\ref{0823.7}) for (\ref{0612.1}) (note that $C_{s,t,i,\mu}(g_r)$ has been decomposed as the sum of the terms $\bbA_{\bbv,i,\mu}(g,\sigma)$ and we in fact extract the i-th row of $\bbX$ from C and then get $\bbA_{\bbv,i,\mu}(g,\sigma)$).
This means that we have not considered the second term in (\ref{a16}).
Recalling (\ref{a16}), the $k$-th derivative of each $\bold\Pi(z)$ can be written as
\begin{equation}\label{b8}
D_{s,t,i,\mu}(g_r)=\sum_{j=1}^3(\bbB_1\mathcal{R}_{j1})_{(s_{j1}t_{j1})}(\mathcal{R}_{j2})_{(s_{j2}t_{j2})}...(\mathcal{R}_{jk})_{(s_{jk}t_{jk})}
(\mathcal{R}_{jk+1}\bbB_2^T)_{(s_{jk+1}t_{jk+1})}, \ k(g_r)=k.
\end{equation}
Therefore, in order to prove (\ref{1117.2}) for $k=3$ what remains is to show that
\begin{equation}\label{0612.3}
N^{-3/2}\sum_{i=1}^{M_1}\sum_{\mu=1}^N\mathbb{E}\left[\prod_{r=1}^{l}C_{s,t,i,\mu}(g_r)\prod_{r=l+1}^{q}D_{s,t,i,\mu}(g_r)\prod_{r=q+1}^{p}B_{s,t,i,\mu}(g_r)\mathcal{T}^{p}_N(\bbX)\right]=O((N^{24\delta}\Psi)^{p}+\|\mathbb{E}\bbL_p(\bbX)\|_{\infty}),
\end{equation}
where $q\le 3$, $\sum_{r=1}^qk(g_r)=3$, $k(g_r)\ge 1$, $r\le q$ and $g_r=0$, $r\ge q+1$.

 We first consider the case when $l=0$, which implies that there is no $C_{s,t,i,\mu}(g_r)$ for $k(g_r)\ge 1$.  We start with the case when $l=0, q=3$, which corresponds to $k(g_j)=1,j=1,2,3$ due to $\sum_{r=1}^qk(g_r)=3$ with $k(g_r)\ge 1$. As a sequence, each summand in $D_{s,t,i,\mu}(g_r)$ becomes
$ (\bbB_1\mathcal{R}_{j1})_{(s_{j1}t_{j1})}(\mathcal{R}_{j2}\bbB_2^T)_{(s_{j2}t_{j2})}$.
 Recalling the definition of $\mathcal{H}_{1i}$, $\mathcal{H}_{1\mu}$, $\mathcal{H}_{sti}$ and $\mathcal{H}_{st\mu}$ above (\ref{1117.4}), the left hand side of (\ref{0612.3}) can be bounded by
\begin{eqnarray}\label{0612.4}
&&\Big|N^{-3/2}\sum_{i=1}^{M_1}\sum_{\mu=1}^N\mathbb{E}\left[\prod_{r=1}^{3}D_{s,t,i,\mu}(g_r)\prod_{r=4}^{p}B_{s,t,i,\mu}(g_r)\mathcal{T}^{p}_N(\bbX)\right]|\Big|
\\
&=&N^{-3/2}
\Big|\sum_{i=1}^{M_1}\sum_{\mu=1}^N\mathbb{E}\left[\prod_{r=1}^{3}D_{s,t,i,\mu}(g_r)F_{st}^{p-3}\mathcal{T}^{p}_N(\bbX)\right]\Big|\le 3^3 N^{-3/2}\sum_{i=1}^{M_1}\sum_{\mu=1}^N\Big|\mathbb{E}\left[\mathcal{H}_{sti}^3\mathcal{H}_{st\mu}^3F_{st}^{p-3}\mathcal{T}^{p}_N(\bbX)\right]\Big|\nonumber
\end{eqnarray}
Recalling (\ref{h28}), we have
$$\sum_{i=1}^{M_1}\sum_{\mu=1}^{N}\mathcal{H}_{sti}^3\mathcal{H}_{st\mu}^3\prec 1.$$
Therefore the right hand side of (\ref{0612.4}) can be bounded by
$$N^{-3/2}\Big|\mathbb{E}\left[F_{st}^{p-3}\mathcal{T}^{p}_N(\bbX)\right]\Big|=O((N^{24\delta}\Psi)^{p}+\|\mathbb{E}\bbL_p(\bbX)\|_{\infty}),$$
using the fact that $N^{-1/2}\lesssim \Psi$. This ensures that (\ref{0612.3}) holds for $l=0$ and $q=3$.

We next consider the case when $l=0, q=2$, which forces $k(g_1)=1,k(g_2)=2$ due to $\sum_{r=1}^qk(g_r)=3$ with $k(g_r)\ge 1$.
Similar to (\ref{0612.4}) we have the upper bound
\begin{eqnarray}\label{0612.5}
&&N^{-3/2}\Big|\sum_{i=1}^{M_1}\sum_{\mu=1}^N\mathbb{E}\left[\prod_{r=1}^{2}D_{s,t,i,\mu}(g_r)\prod_{r=3}^{p}B_{s,t,i,\mu}(g_r)\mathcal{T}^{p}_N(\bbX)\right]
\Big|=N^{-3/2}\Big|\sum_{i=1}^{M_1}\sum_{\mu=1}^N\mathbb{E}\left[\prod_{r=1}^{2}D_{s,t,i,\mu}(g_r)F_{st}^{p-2}\mathcal{T}^{p}_N(\bbX)\right]\Big|\non
&&\le 3^2 N^{-3/2}\Big|\sum_{i=1}^{M_1}\sum_{\mu=1}^N\mathbb{E}\left[(\mathcal{H}_{sti}^3\mathcal{H}_{st\mu}+\mathcal{H}_{sti}\mathcal{H}_{st\mu}^3+\mathcal{H}_{sti}^2
\mathcal{H}_{st\mu}^2)F_{st}^{p-2}\mathcal{T}^{p}_N(\bbX)\right]\Big|\non
&&\prec \mathbb{E}\left[(\Psi^3+\Psi^2)F_{st}^{p-2}\mathcal{T}^{p}_N(\bbX)\right]=O((N^{24\delta}\Psi)^{p}+\|\mathbb{E}\bbL_p(\bbX)\|_{\infty}),
\end{eqnarray}
where we apply the Cauchy-Schwarz inequality together with (\ref{h28}) such that $\sum_{i=1}^{M_1}\mathcal{H}_{sti}\prec \sqrt N $ and $\sum_{\mu \in I_N}\mathcal{H}_{st\mu}\prec \sqrt N $.

As for $l=0, q=1$ we have $k(g_1)=3$ due to $\sum_{r=1}^qk(g_r)=3$ with $k(g_r)\ge 1$.
We also have a similar upper bound
\begin{eqnarray}\label{0612.6}
&&N^{-3/2}\Big|\sum_{i=1}^{M_1}\sum_{\mu=1}^N\mathbb{E}\left[D_{s,t,i,\mu}(g_1)\prod_{r=2}^{p}B_{s,t,i,\mu}(g_r)\mathcal{T}^{p}_N(\bbX)\right]
\Big|=N^{-3/2}\Big|\sum_{i=1}^{M_1}\sum_{\mu=1}^N\mathbb{E}\left[D_{s,t,i,\mu}(g_1)F_{st}^{p-1}\mathcal{T}^{p}_N(\bbX)\right]\Big|\non
&&\le 3 N^{-3/2}\Big|\sum_{i=1}^{M_1}\sum_{\mu=1}^N\mathbb{E}\left[(\mathcal{H}_{sti}^2+\mathcal{H}_{st\mu}^2)F_{st}^{p-1}\mathcal{T}^{p}_N(\bbX)\right]\Big|,\non
&&\prec \mathbb{E}\left[\Psi F_{st}^{p-1}\mathcal{T}^{p}_N(\bbX)\right]=O((N^{24\delta}\Psi)^{p}+\|\mathbb{E}\bbL_p(\bbX)\|_{\infty}).
\end{eqnarray}
Therefore (\ref{0612.3}) holds when $l=0$.

When $l\neq 0$ in (\ref{0612.3}) we below consider the case when $l=1, q=2$ only and the other cases can be handled similarly. In this case we need to show that
\begin{equation}\label{0612.7}
N^{-3/2}\sum_{i=1}^{M_1}\sum_{\mu=1}^N\mathbb{E}\left[C_{s,t,i,\mu}(g_1)D_{s,t,i,\mu}(g_2)\prod_{r=3}^{p}B_{s,t,i,\mu}(g_r)\mathcal{T}^{p}_N(\bbX)\right]=O((N^{24\delta}\Psi)^{p}+\|\mathbb{E}\bbL_p(\bbX)\|_{\infty}),
\end{equation}
where $k(g_1)=2, k(g_2)=1$ or $k(g_1)=1, k(g_2)=2$.  Similar to the arguments above, the left hand side of (\ref{0612.7}) can be bounded by
$$\Big|F_{st}^{p-2}N^{-3/2}\sum_{i=1}^{M_1}\sum_{\mu=1}^{N}(\mathcal{H}_{1i}^2\mathcal{H}_{sti}\mathcal{H}_{st\mu}+\mathcal{H}_{1\mu}^2
\mathcal{H}_{sti}\mathcal{H}_{st\mu}+\mathcal{H}_{1i}\mathcal{H}_{1\mu}\mathcal{H}_{sti}\mathcal{H}_{st\mu})\Big|,$$
which is bounded by $|F_{st}^{p-2}\Psi^2|$ by the inequality
$$\sum_{i=1}^{M_1}\sum_{\mu=1}^{N}\mathcal{H}_{1i}^2\mathcal{H}_{sti}\mathcal{H}_{st\mu}\prec \sum_{i=1}^{M_1}\sum_{\mu=1}^{N}\mathcal{H}_{1i}^2\mathcal{H}_{st\mu}\prec N^{3/2}\Psi^2.$$
This ensures (\ref{0612.7}).

\section{Local law in average (\ref{0310.2})}\label{sec10}
The purpose of this subsection is to prove  the following ((\ref{0310.2}) in Theorem \ref{0817-1})
\begin{eqnarray}\label{0821.11}|m_N(z)-m(z)|\prec \frac{1}{N\eta}.
\end{eqnarray}
As pointed out in the paragraph below (\ref{1230.1}), (\ref{0310.2}) holds when the entries $\{X_{ij}\}$ of $\bbX$ are the standard Gaussian random variable. We next use the interpolation method to prove (\ref{0310.2}) for the general distributions as in proving (\ref{0310.1}). However we do not need induction on the imaginary part of $z$ unlike before due to existence of (\ref{0310.1}). In order to prove (\ref{0821.11}), it is enough to prove that
\begin{eqnarray}\label{0821.12}|m_N(z)-m(z)|\mathcal{T}_N(\bbX)\prec \frac{1}{N\eta}.
\end{eqnarray}

We introduce the notation $\tilde F(X,z)$ as in the last section
$$\tilde F(X,z)=|m_N(z)-m(z)|\mathcal{T}_N(\bbX)=|\frac{1}{N_1}\sum_{k=1}^{N_1}G_{kk}(z)-m(z)|\mathcal{T}_N(\bbX).$$
Checking on Lemmas  \ref{0821-1}, \ref{0822-1}, \ref{1119-5}, (\ref{0821.3}) and (\ref{1123.7}) in the last section it suffices to show
\begin{equation}\label{a34}N^{-k/2}\sum_{i=1}^{M_1}\sum_{\mu=1}^N\mathbb{E}\Big[(\frac{\partial}{\partial \bbX_{i\mu}})^k\tilde F^{p}(\bbX,z)\Big]=O((N^{\delta}\Psi^2)^{2q}+\|\tilde
F^{p}(\bbX,z)\|_{\infty}),\ k\geq 3\end{equation}
where $\delta$ is a sufficiently small constant such that $N^\delta$ is much smaller than $N^\varepsilon$ before $(\ref{0821.12})$ due to the definition of the partial order.
Applying the definition of $B_{s,t,i,\mu}$ in the preceding section with $\bbB_1=\bbB_2=1$ and $s=t=k$, it suffices to prove that for $\sum\limits_{r} k(g_r)=k$
\begin{eqnarray}\label{1119.2}
N^{-k/2}\sum_{i=1}^{M_1}\sum_{\mu=1}^{N}\mathbb{E}\Big(\prod_{r=1}^p\left[\frac{1}{N_1}\sum_{k=1}^{N_1}B_{k,k,i,\mu}(g(r))\right]\mathcal{T}_N(\bbX)\Big)=O((N^{\delta}\Psi^2)^p+\|\mathbb{E}\tilde F^p(X)\|_{\infty}), k\geq 3.
\end{eqnarray}
One can verify (\ref{1119.2}) for $k\geq4$ by repeating the same arguments as in (7.95)-(7.97) in \cite{WY}. The key steps are the following two inequalities:
\begin{eqnarray}\label{1119.4}
\frac{1}{N_1}\sum\limits_{k=1}^{N_1}B_{k,k,i,\mu}(g_h)\prec \Psi^2, \ \ \text{for} \ \ n(g_h)\ge 1,
\end{eqnarray}
and
\begin{eqnarray}\label{0526.1}
\frac{1}{N_1}\sum\limits_{k=1}^{N_1} C_{k,k,i,\mu}(g_h)\prec \Psi^2, \ \ \text{for} \ \ n(g_h)\ge 1.
\end{eqnarray}

Consider (\ref{1119.2}) for $k=3$ now. To this end, as in (\ref{0823.7}) and (\ref{1119.2}), it suffices to prove that
\begin{equation}\label{0826.4}N^{-3/2}\sum_{i=1}^{M_1}\sum_{\mu=1}^{N}|N_1^{-p}\sum_{v_1,...,v_{p}=1}^{N_1}\mathbb{E}\Big(\prod_{r=1}^{p}\bbA_{\bbe_{v_r},i,\mu}(g_r,\sigma_r)G_{\tilde i\tilde i}^{d_{i}}\mathcal{T}_N(\bbX)\Big)|=O_{\prec}((N^{\delta}\Psi^2)^{p}+\mathbb{E}\|\tilde F^p(X)\|_{\infty}),\end{equation}
where $\bbe_{v_r}$ is an $(N+M_1+N_1-M_2)$-dimensional unit vector with the $v_r$-th element being 1 and $1\leq v_r\leq N_1$ (here the size of $\bbe_{v_r}$ is the same as the size of the matrix $\bbG(z)$ ).
One should notice that we don't consider the derivatives of $\bold\Pi(z)$ any more in this subsection since we only care about the upper left $N_1\times N_1$ block matrix of $\bold\Pi(z)$ for the purpose of proving (\ref{0821.11}), which is $m(z)\bbI$ (see (\ref{a17})). 
Hence it suffices to consider the derivative on $\bbG(z)$ and apply its expansion (\ref{expand}). Moreover, in this case, one can see that $d_{\bbv}=0$ since $\bbe_{v_k}(\tilde i)=0, k=1,...,p$ recalling $\tilde i=i+N_1$. Hence there is no factor $\bbv(\tilde{i})^{d_v}$ in (\ref{0826.4}) unlike (\ref{0823.7}). As in (\ref{0825.3}) and (\ref{0725.2}), it then suffices to prove that 
\begin{eqnarray}\label{h0613.1}
&N^{-3/2}N^{d_i\delta}\sum_{i=1}^{M_1}\sum_{\mu=1}^{N}|N_1^{-p}\sum_{v_1,...,v_{p}=1}^{N_1}\mathbb{E}\Big[\prod_{r=1}^{p}\bbA^-(r)\mathbb{E}_{i}\Big(\prod_{r=1}^{p}\bbA^+(r)(\bbX\bold\Delta^T\bbG^{(\tilde i)}\bold\Delta\bbX^T)_{ii}^k\mathcal{T}_N(\bbX)\Big)|\Big]\non
&=O_{\prec}((N^{\delta}\Psi^2)^{p}+\mathbb{E}\|\tilde F^p(X)\|_{\infty}),
\end{eqnarray}
for $k\le Cd_{i}$, where $\sum_r k(g_r)=3$ and $\bbA^{\cdot}(r)=\bbA^{\cdot}_{\bbe_{v_r},i,\mu}(g_r,\sigma_r)$ with $\cdot=-,+$. Here each factor $\bbA^+(r)$ is a product of factors $(\bbG^{(\tilde i)}\bold\Delta X^T)_{ \bbs i}$ and $(X\bold\Delta^T\bbG^{(\tilde i)})_{i \bbs }$, $\bbs\in \{\bbe_{v_r}, \bold\Delta e_{\mu}\}$. We denote the number of factors $(\Delta^T\bbG^{(\tilde i)}\bold\Delta X^T)_{ \mu i}$ and $(X\bold\Delta^T\bbG^{(\tilde i)}\Delta)_{i \mu }$ by $d_{\bbx,\mu,r}$(i.e. $\bbs=\Delta\bbe_{\mu}$) contained in $\bbA^+(r)$ and write $d_{\bbx,\mu}=\sum\limits_{r=1}^pd_{\bbx,\mu,r}$. By (\ref{0310.1}) and Definition \ref{0826-1}
, it is easy to conclude that
 \begin{eqnarray}\label{0812.2}
  \frac{1}{N_1}\sum_{v_r=1}^{N_1}|(\bbG^{(\tilde i)}\bold\Delta)_{ v_r l}(\bbG^{(\tilde i)}\bold\Delta)_{ v_r j}|\prec  \frac{\Im (\bbG^{(\tilde i)})_{jl}}{N\eta}\prec \Psi^2.
  \end{eqnarray}

 Consider $\mathbb{E}_{i}\Big(\prod_{r=1}^{p}\bbA^+(r)(\bbX\bold\Delta^T\bbG^{(\tilde i)}\bold\Delta\bbX^T)_{ii}^k\Big)$.  
 As in (\ref{0725.3})-(\ref{b9}) we obtain
   \begin{eqnarray}\label{h0613.8}
&&|\mathbb{E}_{i}\prod_{r=1}^{p}\bbA^+(r)(\bbX\bold\Delta^T\bbG^{(\tilde i)}\bold\Delta\bbX^T)_{ii}^k|\prec \max_{L}\max_{\{d_l\}}\max_{\{k_l\}}\sum_{j_1,...,j_L}\times\prod_{l=1}^L
\Bigg(N^{-d/2-k}(|(\bold\Delta^T\bbG^{(\tilde i)}\bold\Delta)_{j_l \mu}|\nonumber\\
&&+|(\bold\Delta^T\bbG^{(\tilde i)}\bold\Delta)_{\mu j_l }|)^{d_l-\sum_{r=1}^{p} d_{l,r}}\prod_{r=1}^{p}\left[(|(\bold\Delta^T\bbG^{(\tilde i)})_{j_l \bbe_{v_r}}|+|(\bbG^{(\tilde i)}\bold\Delta)_{\bbe_{v_r} j_l }|)^{d_{l,r}}\right]\Bigg),
\end{eqnarray}
where $d_{l,r}$ denotes the number of the factors $(\bold\Delta^T\bbG^{(\tilde i)})_{j_l \bbe_{v_r}}$ and $(\bbG^{(\tilde i)}\bold\Delta)_{\bbe_{v_r} j_l }$ (essentially, it is the number of the factors $(\bbG^{(\tilde i)}\bold\Delta X^T)_{ \bbs i}$ and $(X\bold\Delta^T\bbG^{(\tilde i)})_{i \bbs }$ with $\bbs=\bbe_{v_r}$ in $\bbA^+(r)$).  By (\ref{0528.3})-(\ref{0528.4}), it is easy to see that $\sum_{l}d_{l,r}\le 2$  in $A^+(r)$. We have to combine $\prod_{l=1}^{L}\left[(|(\bold\Delta^T\bbG^{(\tilde i)})_{j_l \bbe_{v_r}}|+|(\bbG^{(\tilde i)}\bold\Delta)_{\bbe_{v_r} j_l }|)^{d_{l,r}}\right]$ with $A^-(r)$ together so that we may use (\ref{0812.2}). Hence we below consider the upper bound of
\begin{equation}\label{b10}
N_1^{-p}\sum_{v_1,...,v_{p}=1}^{N_1}\prod_{r=1}^{p}\bbA^-(r)\prod_{l=1}^{L}\left[(|(\bold\Delta^T\bbG^{(\tilde i)})_{j_l \bbe_{v_r}}|+|(\bbG^{(\tilde i)}\bold\Delta)_{\bbe_{v_r} j_l }|)^{d_{l,\mu,r}}\right]
\end{equation}
first. One should notice that the above summation and product are only about $v_r$, which are independent of $l$.

For $r\ge q+1$ satisfing $\sigma_r=0$, recalling $A_{\bbv,i,\mu}(g, 0)$ in Definition \ref{0826-1}, we have
$$\frac{1}{N_1}\sum_{v_r=1}^{N_1}\bbA^-(r)=\tilde F(X,z)+O_{\prec}(\Psi^2),$$
where the $O_{\prec}(\Psi^2)$ follows from the fact that $\frac{1}{N_1}\sum_{v_r=1}^{N_1}\bbA^-(r)=m_N(z)-m(z)-\frac{1}{N_1}\sum_{v_r=1}^{N_1}\bbA_{\bbe_{v_r},i,\mu}(g_r,1)$ and using the large deviation inequality and (\ref{0812.2}) to control $\frac{1}{N_1}\sum_{v_r=1}^{N_1}\bbA_{\bbe_{v_r},i,\mu}(g_r,1)$. For the remaining $\sum_{r=q+1}^p|\sigma_r|$ indices, we always have the trivial order $\bbA^-(r)=1$ by the fact that $\bbA(r)=\bbA^+(r)$. When $\sigma_r=1$, $r\ge q+1$, by the expansion of Definition \ref{0826-1}(i), we have $\sum_{l=1}^{L} d_{l,r}=2$. Thus by (\ref{0812.2}) we have for $\sigma_r=1$, $r\ge q+1$,
\begin{eqnarray}\label{h0613.10}
\frac{1}{N_1}\sum_{v_r=1}^{N_1}\prod_{l=1}^{L}\left[(|(\bold\Delta^T\bbG^{(\tilde i)})_{j_l \bbe_{v_r}}|+|(\bbG^{(\tilde i)}\bold\Delta)_{\bbe_{v_r} j_l }|)^{d_{l,\mu,r}}\right]\prec \Psi^2.
\end{eqnarray}

Furthermore, consider $r\le q$. If there are two indices $v_r$ (associated with $\bbe_{v_r}$ in the factors $(\bbG^{(\tilde i)}\bold\Delta X^T)_{ \bbs i}$ and $(X\bold\Delta^T\bbG^{(\tilde i)})_{i \bbs }$, $\bbs\in \{\bbe_{v_r}, \bold\Delta e_{\mu}\}$ ) appearing in $\bbA^-(r)$, by (\ref{0812.2}) then we have
$$\frac{1}{N_1}\sum_{v_r=1}^{N_1}\bbA^-(r)\prec \Psi^2.$$
If there is no index $v_r$ appearing in $\bbA^-(r)$, then we use the bound
$$\bbA^-(r)\prec 1.$$
In this case($r\le q$, no $v_r$ appears in $\bbA^-(r)$) two indices $v_r$ both appear in $\bbA^+(r)$ and hence we combine them with $\bbA^-(r)$, as in (\ref{b10}). Hence we also have as in (\ref{h0613.10})
$$\frac{1}{N_1}\sum_{v_r=1}^{N_1}\bbA^-(r)\prod_{l=1}^{L}\left[(|(\bold\Delta^T\bbG^{(\tilde i)})_{j_l \bbe_{v_r}}|+|(\bbG^{(\tilde i)}\bold\Delta)_{\bbe_{v_r} j_l }|)^{d_{l,\mu,r}}\right]\prec \Psi^2.$$
If there is only one $v_r$ appearing in $\bbA^-(r)$, by Definition \ref{0826-1}(ii) we have $\sum_{l=1}^{L} d_{l,r}=1$. Hence one $v_r$ appears in $\bbA^+(r)$ and we combine such a term involving $v_r$ in $\bbA^+(r)$ with $\bbA^-(r)$, as in (\ref{b10}). Therefore by (\ref{0812.2}) we conclude that (\ref{h0613.10}) holds.
Therefore, summarizing above arguments for $r\le q$, we have for $r\le q$
$$\frac{1}{N_1}\sum_{v_r=1}^{N_1}\bbA^-(r)\prod_{l=1}^{L}\left[(|(\bold\Delta^T\bbG^{(\tilde i)})_{j_l \bbe_{v_r}}|+|(\bbG^{(\tilde i)}\bold\Delta)_{\bbe_{v_r} j_l }|)^{d_{l,r}}\right]\prec \Psi^2.$$
Furthermore, together with the arguments for $r\ge q+1$, we have
\begin{eqnarray}
&&N_1^{-p}\sum_{v_1,...,v_{p}=1}^{N_1}\prod_{r=1}^{p}\bbA^-(r)\prod_{l=1}^{L}\left[(|(\bold\Delta^T\bbG^{(\tilde i)})_{j_l \bbe_{v_r}}|+|(\bbG^{(\tilde i)}\bold\Delta)_{\bbe_{v_r} j_l }|)^{d_{l,\mu,r}}\right]\non
&&\prec (\tilde F(X,z)+O_{\prec}(\Psi^2))^{p-q-\sum_{r=q+1}^p|\sigma_r|} \Psi^{2(q+\sum_{r=q+1}^p|\sigma_r|)}.
\end{eqnarray}

We come back to analyze (\ref{h0613.8}). Similar to (\ref{0117.1}), we can show that
\begin{eqnarray}\label{0727.1}
&&\sum_{j_1,...,j_L}\prod_{l=1}^L\Bigg(N^{-d/2-k}(|(\bold\Delta^T\bbG^{(\tilde i)}\bold\Delta)_{j_l \mu}|+|(\bold\Delta^T\bbG^{(\tilde i)}\bold\Delta)_{\mu j_l }|)^{d_l- \sum_{r=1}^p d_{l,r}}\Bigg)\\
&&\prec N^{-1/2}\Psi^{d_{\bbx}-\sum_{l,r}d_{l,r}-\mathbf{1}(d_{\bbx,\mu}=3)},\nonumber
\end{eqnarray}
where $d_{\bbx,\mu}=\sum_{r=1}^pd_{\bbx,\mu,r}$. At the right hand side of (\ref{0727.1}), comparing to (\ref{0117.1}), $\mathbf{1}(d_{\bbv}=0)$ disappears since $d_{\bbv}$ is always equal to 0 for $\bbv=\bbe_{v_1},...,\bbe_{v_p}$. The reason why we can replace $d_{\bbx}$ by $d_{\bbx,\mu}$ is because we don't consider  $(\bbG^{(\tilde i)}\bold\Delta )_{\bbe_{v_r} j_l}$ and $(\bold\Delta^T\bbG^{(\tilde i)})_{j_l \bbe_{v_r}}$ in (\ref{0727.1}). Also the reason why $d_{\bbx}$ can be replaced by $d_{\bbx}-\sum_{l,r}d_{l,\mu,r}$ is because the power at the left hand side of (\ref{0727.1}) becomes $d_l-\sum_{r=1}^p d_{l,\mu,r}$.  

By the arguments above, we conclude that the LHS of (\ref{h0613.1}) is bounded by
$$N^{d_i\delta}\Psi^{d_{\bbx}-\sum_{l,r}d_{l,\mu,r}-\mathbf{1}(d_{\bbx,\mu}=3)}\Psi^{2(q+\sum_{r=q+1}^p|\sigma_r|)}\mathbb{E}(\tilde F(\bbX)+\Psi^2)^{p-q-\sum_{r=q+1}^p|\sigma_r|}.$$
So (\ref{0826.4}) holds if
\begin{eqnarray}\label{0826.5}
d_{\bbx}-\sum_{l,r}d_{l,r}-\mathbf{1}(d_{\bbx,\mu}=3)\ge 0.
\end{eqnarray}
In order to establish (\ref{0826.5}), we analyze $d_{\bbx,\mu,r}$ carefully. First, if $k(g_r)=0$, then $d_{\bbx,\mu,r}=0$. Secondly, if $1\le k(g_r)\le 3$, the following holds
$$d_{\bbx,r}\le 2 \Longrightarrow d_{\bbx,\mu,r}\le I(k(g_r)\ge 2),$$
where $d_{\bbx,r}=deg(A^+(r))$.
Hence if $d_{\bbx,\mu}=3$, then there exists an $r\le q$ such that $d_{\bbx,r}\ge3$. Then
$$d_{\bbx,r}-\sum_{l}d_{l,r}-1\ge 0,$$
from $\sum_{l}d_{l,r}\le 2$. Therefore, (\ref{0826.5}) holds by the fact that
$$d_{\bbx,r}-\sum_{l}d_{l,r}\ge 0.$$
 Therefore, we have proved the averaged local law.
\section{Proof of Lemma \ref{0525-1}}
The proof of Lemma \ref{0525-1} is exactly the same as the proof of Lemma 13 in \cite{WY} and thus we omit it.
\section{Proof of Theorem \ref{0826-3}}

\begin{proof}
%
%

Unlike \cite{KY14}, \cite{LHY2011} and \cite{BPZ2014a} we use the interpolation method (\ref{1124.14}), which is succinct and powerful when proving green function comparison theorems.
In view of (\ref{0603.2}) and (\ref{0603.3}) we have
\begin{eqnarray}\label{1119.10}
&&|\mathbb{E}K(N\int_{E_1}^{E_2}\Im m_{\bbX^1}(x+i\eta)dx)-\mathbb{E}K(N\int_{E_1}^{E_2}\Im m_{\bbX^0}(x+i\eta)dx)|=\non
&&\left|\mathbb{E}K(N\int_{E_1}^{E_2}\Im m_{\bbX^1}(x+i\eta)\mathcal{T}_N(\bbX^1)dx)-\mathbb{E}K(N\int_{E_1}^{E_2}\Im m_{\bbX^0}(x+i\eta)\mathcal{T}_N(\bbX^0)dx)\right|+O(N^{-1}).\non
\end{eqnarray}
Applying (\ref{1124.14}) with $F(\bbX)=K(N\int_{E_1}^{E_2}\Im m_{\bbX}(x+i\eta)\mathcal{T}_N(\bbX ))$ we only need to bound the following
\begin{eqnarray}\label{1120.1}
\Bigg|\sum_{i=1}^{M_1}\sum_{\mu=1}^N \mathbb{E}g(X_{i\mu}^1)-\mathbb{E}g(X_{i\mu}^0))\Bigg|,
\end{eqnarray}
where
\begin{eqnarray}\label{0826.7}
g(X_{i\mu}^u)=K(N\int_{E_1}^{E_2}\Im m_{\bbX_{(i\mu)}^{t,X_{i\mu}^u}}(x+i\eta)\mathcal{T}_N(\bbX_{(i\mu)}^{t,X_{i\mu}^u})dx),\quad u=0,1.
\end{eqnarray}
As in (\ref{1123.4}) and (\ref{1123.5}), we use Taylor's expansion up to order five to expand two functions $g(X_{i\mu}^u),u=0,1$ 
 at the point 0. Then take the difference of the Taylor's expansions of $g(X_{i\mu}^u),u=0,1$. 
  By the first two moments matching condition it then suffices to bound the third, fourth and remainder derivatives as follows
  \begin{equation}\label{0826.6}
N^{-3/2}\sum_{i=1}^{M_1}\sum_{\mu=1}^N\sum_{r=1}^3 \sum_{k_1,..,k_r\in \mathbb{N}_+ \atop k_1+..+k_r=3}C_r\Bigg| N\int_{E_1}^{E_2} \mathbb{E} K^{(r)}(0)\prod_{j=1}^r m_{\bbX_{(i\mu)}^{t,0}}^{(k_j)}(x+i\eta)\mathcal{T}_N(\bbX_{(i\mu)}^{t,0})dx \Bigg|,
\end{equation}
\begin{equation}\label{0524.1}
N^{-2}\sum_{i=1}^{M_1}\sum_{\mu=1}^N\sum_{r=1}^4 \sum_{k_1,..,k_r\in \mathbb{N}_+ \atop k_1+..+k_r=4}C_r \max_{x}|K^{(r)}(x)|  \mathbb{E}\prod_{j=1}^r \Bigg(N\int_{E_1}^{E_2}\Bigg| m_{\bbX_{(i\mu)}^{t,0}}^{(k_j)}(x+i\eta)\mathcal{T}_N(\bbX_{(i\mu)}^{t,0}) \Bigg|dx\Bigg),
\end{equation}
and the fifth derivative corresponding to the remainder of integral form
\begin{equation}\label{0524.1*}
N^{-5/2}\sum_{i=1}^{M_1}\sum_{\mu=1}^N\sum_{r=1}^5 \sum_{k_1,..,k_r\in \mathbb{N}_+ \atop k_1+..+k_r=4}C_r \max_{x}|K^{(r)}(x)|  \mathbb{E}\prod_{j=1}^r \Bigg(N\int_{E_1}^{E_2}\Bigg| m_{\bbX_{(i\mu)}^{t,\theta X_{i\mu}^u}}^{(k_j)}(x+i\eta)\mathcal{T}_N(\bbX_{(i\mu)}^{t,\theta X_{i\mu}^u }) \Bigg|dx\Bigg),
\end{equation}
where $C_r$ is a constant depending on r only, $m_{\bbX_{(i\mu)}^{t,0}}^{(k_i)}(\cdot)$ denotes the $k_i$th derivative with respect to $X_{i\mu}^{u}$ and $0\leq\theta\leq 1$. Here we ignore the terms involving the derivatives of $\mathcal{T}_N(\bbX_{(i\mu)}^{t,\theta X_{i\mu}^u})$ due to (\ref{0603.2}), (\ref{0603.3}) and (\ref{1119.12}). 

We focus on (\ref{0524.1}) and (\ref{0524.1*}) first. To investigate (\ref{0524.1}) and (\ref{0524.1*}) we claim that it suffices to prove that
\begin{eqnarray}\label{0524.2}
\Bigg(N\int_{E_1}^{E_2}\Bigg| m_{\bbX_{(i\mu)}^{u,X_{i\mu}^1}}^{(k)}(x+i\eta)\mathcal{T}_N(\bbX_{(i\mu)}^{u,X_{i\mu}^1})\Bigg|dx\Bigg)\prec(N^{\frac{1}{3}+\ep}\Psi^2),
\end{eqnarray}
where $k\ge 1$. Indeed, if (\ref{0524.2}) holds then (\ref{0524.2}) still holds when $X_{i\mu}^1$ is replaced by $\theta X_{i\mu}^1$ by checking on the argument of (\ref{0524.2}). We then conclude that the facts that $(\ref{0524.1})\prec(N^{\frac{1}{3}+\ep}\Psi^2)$ and that $(\ref{0524.1*})\prec(N^{-\frac{1}{2}+\frac{1}{3}+\ep}\Psi^2)$ follow from Lemma \ref{1123-3}, (\ref{1119.12}) and an application of (\ref{1123.4}).

By (\ref{0526.2}) and (\ref{0526.1}) we have for $k\ge 1$
\begin{eqnarray}\label{h0613.2}
\Bigg| m_{\bbX_{(i\mu)}^{u,X_{i\mu}^1}}^{(k)}(x+i\eta)\mathcal{T}_N(\bbX_{(i\mu)}^{u,X_{i\mu}^1})\Bigg|\prec \Psi^2,
\end{eqnarray}
which implies that $(\ref{0524.2})\prec(N^{\frac{1}{3}+\ep}\Psi^2)$. Here we would point out that the derivatives $m_{\bbX_{(i\mu)}^{u,X_{i\mu}^1}}^{(k)}(\cdot)$ are of the form $\frac{1}{M_1}\sum\limits_{k=1}^{M_1}C_{k,k,i,\mu}(g_h)$ from (\ref{a16}), (\ref{0526.2}), (\ref{a38}), (\ref{a37}), (\ref{a34}) and (\ref{1119.2}). By Lemma 2.3 of \cite{BPWZ2014b} we have
\begin{eqnarray}\label{h0613.3}
\Psi^2 \asymp \frac{1}{N\sqrt{\eta}}=O(N^{-\frac{2}{3}+\ep/2}).
\end{eqnarray}

From now on we consider (\ref{0826.6}). One should notice that we do not extract the summation $\sum_{i=1}^{M_1}\sum_{\mu=1}^N$ outside the expectation like (\ref{0524.1}) in order to make it easier, compared with the proof of (\ref{0826.4}). Since $K(.)$ involved in the above expectation is non random and does not affect the order of the expectation, we can ignore $K^{(r)}(0)$ in the sequel. Similar to the claim (\ref{0524.2}), it suffices to find the upper bound of
\begin{eqnarray}\label{h0613.4}
N^{-3/2}\sum_{i=1}^{M_1}\sum_{\mu=1}^N\sum_{r=1}^3 \sum_{k_1,..,k_r\in \mathbb{N}_+ \atop k_1+..+k_r=3}\Bigg|N\int_{E_1}^{E_2} \mathbb{E}\prod_{j=1}^r m_{\bbX_{(i\mu)}^{u,X_{i\mu}^1}}^{(k_j)}(x+i\eta)\mathcal{T}_N(\bbX_{(i\mu)}^{u,X_{i\mu}^1})dx \Bigg|.
\end{eqnarray}
First of all, (\ref{h0613.2})-(\ref{h0613.3}) always hold, which concludes that for $r\ge 2$
\begin{eqnarray}\label{h0613.5}
&&N^{-3/2}\sum_{i=1}^{M_1}\sum_{\mu=1}^N\sum_{r=2}^3 \sum_{k_1,..,k_r\in \mathbb{N}_+ \atop k_1+..+k_r=3}\Bigg|N\int_{E_1}^{E_2} \mathbb{E}\prod_{j=1}^r m_{\bbX_{(i\mu)}^{u,X_{i\mu}^1}}^{(k_j)}(x+i\eta)\mathcal{T}_N(\bbX_{(i\mu)}^{u,X_{i\mu}^1})dx \Bigg|\non
&&\prec N^{1/2+1/3+\ep}\Psi^4 \prec N^{-\frac{1}{2}+2\ep}.
\end{eqnarray}
 Therefore, referring to (\ref{h0613.4}), it remains to consider the case $r=1$, i.e. we need to find the upper bound of
$$N^{-3/2}\sum_{i=1}^{M_1}\sum_{\mu=1}^N \Bigg|N\int_{E_1}^{E_2} \mathbb{E} m_{\bbX_{(i\mu)}^{u,X_{i\mu}^1}}^{(3)}(x+i\eta)\mathcal{T}_N(\bbX_{(i\mu)}^{u,X_{i\mu}^1})dx \Bigg|.$$
 By checking (\ref{1119.10})-(\ref{h0613.3}) carefully one can find if we can extract one more $\frac{1}{\sqrt N}$ from the expectation above then the proof of this theorem is complete.
  In other words, the aim is to prove that
 $$\Bigg|\mathbb{E} m_{\bbX_{(i\mu)}^{u,X_{i\mu}^1}}^{(3)}(x+i\eta)\mathcal{T}_N(\bbX_{(i\mu)}^{u,X_{i\mu}^1})dx \Bigg|\prec N^{-1/2}\Psi^2.$$
 This can be proved by (\ref{0526.2}) and (\ref{0812.2}) as in (\ref{0117.1}) and (\ref{0727.1}) and the details are ignored here.
 Here we would comment that $N^{-1/2}$ comes from counting the number of the $i$-th row of $\bbX$ in the expansion of $m_{\bbX_{(i\mu)}^{u,X_{i\mu}^1}}^{(3)}(x+i\eta)\mathcal{T}_N(\bbX_{(i\mu)}^{u,X_{i\mu}^1})$ and one can also refer to the arguments above (\ref{h0613.1}) to see $d_{\bbv}=0$. In addition $\Psi^2$ follows from (\ref{0812.2}) and the fact that there are always two indices $\bbe_{v_r}$ involved in $m_{\bbX_{(i\mu)}^{u,X_{i\mu}^1}}^{(3)}(x+i\eta)\mathcal{T}_N(\bbX_{(i\mu)}^{u,X_{i\mu}^1})$.

Summarizing the above we have shown that
\begin{eqnarray}\label{0117.2}
|\mathbb{E}K(N\int_{E_1}^{E_2}\Im m_{\bbX^1}(x+i\eta)dx)-\mathbb{E}K(N\int_{E_1}^{E_2}\Im m_{\bbX^0}(x+i\eta)dx)|\prec N^{-\frac{1}{3}+2\ep}.
\end{eqnarray}
The proof is complete by choosing an appropriate $\ep$. 
\end{proof}


\end{document}